\documentclass[11pt]{amsart}

\usepackage{amsmath, amssymb, amsthm, amsfonts}
\usepackage{graphicx}
\usepackage{bbm}
\usepackage{tikz-cd}
\usepackage[colorlinks=true, linkcolor=blue, citecolor=blue, urlcolor=blue]{hyperref}
\numberwithin{equation}{section}
\usepackage{enumitem}

\usepackage{setspace}

\usepackage{mathrsfs}
\usepackage[colorinlistoftodos]{todonotes}
\usepackage{yfonts}
\usepackage{adjustbox}
\usepackage{bbold}
\usepackage{enumitem}   
\usepackage{stmaryrd}
\usepackage{cancel}
\usepackage[OT2,T1]{fontenc}
\usepackage{mathpazo}
\usepackage{eulervm}
\DeclareSymbolFont{cyrletters}{OT2}{wncyr}{m}{n}
\DeclareMathSymbol{\Sha}{\mathalpha}{cyrletters}{"58}

\title[Multiplicity formula for disconnected tori]{On the multiplicity formula for discrete automorphic representations of disconnected tori}
\author{Yi Luo}
\address{21 Lower Kent Ridge Road, NUS, 119077 Singapore}
\email{yiluo@u.nus.edu}
\subjclass[2020]{11F70, 22E55}
\date{}

\newtheorem{thm}{Theorem}[section]
\newtheorem{lem}[thm]{Lemma}
\newtheorem{prop}[thm]{Proposition}
\newtheorem{defn}[thm]{Definition}

\newtheorem{cor}[thm]{Corollary}
\newtheorem{claim}[thm]{Claim}
\newtheorem{rmk}[thm]{Remark}
\newtheorem{fact}[thm]{Fact}
\newtheorem{obs}[thm]{Observation}
\newtheorem{noneg}[thm]{Non-example}
\newtheorem{assumption}[thm]{Assumption}

\newcommand{\C}{\mathbb{C}}

\makeatletter
\@ifundefined{acknowledgements}{%
  \newenvironment{acknowledgements}{\section*{Acknowledgements}}{}%
}{}
\makeatother

\begin{document}

\begin{abstract}
Kaletha extended the local Langlands conjectures to disconnected groups that are inner forms of semidirect products $G\rtimes A$, where the finite group scheme $A$ preserves a pinning of the connected reductive group $G$, and proved the conjectures when $G$ is a torus. Our first main result is an intrinsic reinterpretation of the local Langlands correspondence for such disconnected tori. Our second main result, and the central objective of this paper, is to establish an automorphic multiplicity formula for them.
\end{abstract}
\maketitle

\setcounter{tocdepth}{2}
\tableofcontents

\section{Introduction}
\subsection{Background}
Let $F$ be a number field and let $\mathbb{A}$ be the adele ring of $F$. Let $G$ be a connected reductive group defined over $F$ with centre $Z_{G}$. Let $[G]:= A_{G}G(F)\backslash G(\mathbb{A})$ be its adelic quotient, where $A_{G}$ is the identity component of the $\mathbb{R}$-points of the largest $\mathbb{Q}$-split subtorus of $\mathrm{Res}_{F/\mathbb{Q}}Z_{G}$.  A fundamental problem in the study of automorphic representations is to understand the decomposition of the discrete part $L_{\mathrm{disc}}^{2}([G])$, regarded as a $G(\mathbb{A})$-representation via right translation, and the multiplicities of its irreducible constituents. The work of Labesse and Langlands~\cite{labesse1979indistinguishability,langlands1982debuts} and the work of Kottwitz~\cite{kottwitz1984stable} suggest a conjectural answer for tempered automorphic representations. Assume that $G$ is quasi-split with simply connected derived group. Conjecturally, an admissible tempered discrete $L$-parameter $\phi:L_{F}\to {}^{L}G$ gives rise to an adelic $L$-packet $\Pi_{\phi}$ of tempered representations of
$G(\mathbb{A})$, and each tempered discrete automorphic representation is expected to occur in some $\Pi_{\phi}$. Here, $L_{F}$ denotes the (conjectural) Langlands group of $F$~\cite[\S12]{kottwitz1984stable}. If $\eta$ is a tempered representation of $G(\mathbb{A})$, then the multiplicity of $\eta$ in $L_{\mathrm{disc}}^{2}([G])$ is expected to be the sum of $m_{\eta,\phi}$, where the sum ranges over all equivalence classes of $L$-parameters (see~\cite[\S10.4]{kottwitz1984stable}) such that $\eta\in \Pi_{\phi}$, and $m_{\eta,\phi}$ denotes the $[\phi]$-contribution towards the multiplicity of $\eta$. Moreover, it is conjectured that
\begin{align}\label{eq:intro_mult}
m_{\eta,\phi} = \frac{1}{|\mathcal{S}_{\phi}|}\sum_{x\in \mathcal{S}_{\phi}}\langle x, \eta\rangle,
\end{align}
where $\mathcal{S}_{\phi}$ is a finite group related to the centralizer of $\phi$ in $\widehat{G}$ (see~\cite[\S10.2]{kottwitz1984stable}), and  $\langle \cdot ,\cdot 
\rangle:\mathcal{S}_{\phi} \times \Pi_{\phi}\to\mathbb{C}$ is a pairing relying on the local Langlands conjectures~\cite{kaletha2016rigid}. In~\cite{kaletha2018global}, after introducing certain new global Galois gerbes that govern rigid inner forms, Kaletha makes the definitions of $\mathcal{S}_{\phi}$ and the pairing  $\langle \cdot ,\cdot
\rangle$ precise for general connected reductive groups $G$.

Consider the case when $G=T$ is a torus. For simplicity, we further assume $T$ is anisotropic. Then the space $L^{2}([T])$ is a multiplicity-free direct sum of all the Hecke characters of $T$. On the dual side, according to the axioms for the Langlands group $L_{F}$~\cite[\S12]{kottwitz1984stable}, any $L$-morphism $\phi: L_{F} \to {}^{L}T$ should factor through the quotient onto the Weil group $L_{F}\twoheadrightarrow W_{F}$, which allows us to view it as an $L$-morphism $\phi: W_{F} \to {}^{L}T$. Furthermore, in the case of tori, the group $\mathcal{S}_{\phi}$ is trivial, since it is defined as a certain subquotient of the trivial quotient $\widehat{T}/Z(\widehat{T})$ (see~\cite{kaletha2018global}). Hence the pairing $\langle \cdot ,\cdot
\rangle$ is identically $1$ and the multiplicity formula~\eqref{eq:intro_mult} holds trivially due to the multiplicity-one decomposition of $L^{2}([T])$. The dual-side parametrization of the Hecke characters recovers the global Langlands correspondence for tori~\cite{langlands1997representations}.

Turning to the local setting, let $F$ now denote a local field of characteristic $0$. In~\cite{kaletha2022local}, Kaletha initiated an extension of the local Langlands conjectures to disconnected groups.  Specifically, he considers inner forms of so-called \textit{quasi-split} disconnected reductive groups, which are of the form $G\rtimes A$, where $G$ is a connected quasi-split reductive group and $A$ is a finite group scheme whose action on $G$ preserves an $F$-pinning. Within this setting, Kaletha formulates conjectures on the (refined) local Langlands correspondence (LLC) and proves them in the case of disconnected tori, i.e. when $G = T$ is a torus. 

This paper takes an initial step toward extending the automorphic multiplicity formula to disconnected groups. Building on Kaletha's pioneering local results, we establish the formula for such disconnected tori in the case of pure inner forms. We expect that all the results can be naturally generalized to all inner forms (see Remark~\ref{rmk: rigid_inner_form}). 

Throughout, we work over fields of characteristic zero, both locally and globally, so that we stay within Kaletha's framework. However, this assumption is by no means essential. When $G = T$ is a torus acted on by a finite \'etale group scheme $A$, we expect that both the local Langlands correspondence for such disconnected tori and the global results in this paper generalize to the positive-characteristic setting. 

\subsection{Main results}
Let $F$ be a number field and let $\mathbb{A}$ be its adele ring. Let $T$ be a torus defined over $F$. Let $A$ be a finite group scheme defined over $F$ that acts on $T$, where the action is defined over $F$ as well. Based on this action, we define $\tilde{T} := T\rtimes A$. Given $z\in Z^{1}(F, T)$, we twist the rational structure of $\tilde{T}$ via $z$ through inner automorphisms, hence obtain a pure inner form $\tilde{T}_{z}$. In this introduction, we further assume that $T$ is anisotropic, in order to ease the exposition. Then the adelic quotient $[\tilde{T}_{z}] := \tilde{T}_{z}(F)\backslash \tilde{T}_{z}(\mathbb{A})$ is compact. The key subject of this paper is to decompose  the right regular representation of $\tilde{T}_{z}(\mathbb{A})$ on the space $L^{2}([\tilde{T}_{z}])$.

\subsubsection{Local aspects}
Our first main result is an intrinsic reinterpretation of Kaletha's LLC for disconnected tori. 

Let $v$ be a place of the number field $F$ and let $F_{v}$ be the completion of $F$ at $v$. Denote by $z_{v}$ the image of $z$ under the natural map $Z^{1}(F,T)\to Z^{1}(F_{v}, T)$, and write $[z_{v}]$ for its class in $H^{1}(F_{v},T)$. Consider the local pure inner form $\tilde{T}_{z_{v}}$. Its rational points fit into a short exact sequence:
$$ 1 \to T(F_{v}) \to \tilde{T}_{z_{v}}(F_{v}) \to A(F_{v})^{[z_{v}]}\to 1, $$
where $A(F_{v})^{[z_{v}]}$ denotes the stabilizer of $[z_{v}]$ in $A(F_{v})$.

Let  $\phi_{v}: W_{F_{v}} \to {}^{L}T$ be a local $L$-parameter for the torus $T$. Under the LLC for $T$, the equivalence class $[\phi_{v}]$ corresponds to a character of $T(F_{v})$, which we still denote as $[\phi_{v}]$. Define the $L$-packet for $\tilde{T}_{z_{v}}$ associated with $\phi_{v}$ to be $\Pi_{\phi_{v}}(\tilde{T}_{z_{v}}) :=\mathrm{Irr}(\tilde{T}_{z_{v}}(F_{v}),[\phi_{v}])$, namely the set of irreducible representations of $\tilde{T}_{z_{v}}(F_{v})$ whose restriction to $T(F_{v})$ contains the character $[\phi_{v}]$. 

Recall that, in the case of a connected torus $T$, the group of self-equivalences for $\phi_{v}$ is $\mathrm{Cent}(\phi_{v},\widehat{T}) = \widehat{T}^{\,\Gamma_{v}}$, the points in $\widehat{T}$ that are fixed by the local Galois group $\Gamma_{v}$. It is noteworthy that $[z_{v}]$ may be viewed as a character of the component group $\pi_{0}(\widehat{T}^{\,\Gamma_{v}})$, via Kottwitz's reformulation $H^{1}(F_{v},T)\cong\pi_{0}(\widehat{T}^{\,\Gamma_{v}})^{*}$ of the classical local Tate--Nakayama duality (see \S\ref{TNdualityforhyper}). In the disconnected setting, we consider an enlarged group of self-equivalences: $\tilde{S}_{\phi_{v}}^{[z_{v}]} := \mathrm{Cent}(\phi_{v}, \widehat{T}\rtimes A(F_{v})^{[z_{v}]})$. There is a short exact sequence:
\begin{equation*}
    1 \to \pi_{0}(\widehat{T}^{\,\Gamma_{v}}) \to \pi_{0}(\tilde{S}_{\phi_{v}}^{[z_{v}]}) \to A(F_{v})^{[\phi_{v}],[z_{v}]} \to 1,
\end{equation*}
where $A(F_{v})^{[\phi_{v}],[z_{v}]}$ consists of elements in $A(F_{v})$ that fix both $[\phi_{v}]$ and $[z_{v}]$. We denote by $\mathrm{Irr}(\pi_{0}(\tilde{S}_{\phi_{v}}^{[z_{v}]}), [z_{v}])$ the set of irreducible representations of $\pi_{0}(\tilde{S}_{\phi_{v}}^{[z_{v}]})$ whose restriction to $\pi_{0}(\widehat{T}^{\,\Gamma_{v}})$ contains $[z_{v}]$.

The local Langlands correspondence for disconnected tori asserts a natural bijection 
$$\Pi_{\phi_{v}}(\tilde{T}_{z_{v}}) \longleftrightarrow \mathrm{Irr}(\pi_{0}(\tilde{S}_{\phi_{v}}^{[z_{v}]}), [z_{v}]),$$ 
which Kaletha has constructed and for which the endoscopic character identities have been verified. We denote by $\tilde{T}_{z_{v}}(F_{v})^{[\phi_{v}]}$ the stabilizer of $[\phi_v]$ under the conjugation of $\tilde{T}_{z_{v}}(F_{v})$. In practice, it is more convenient to express the LLC via $\mathrm{Irr}(\tilde{T}_{z_{v}}(F_{v})^{[\phi_{v}]},[\phi_{v}])$, which is canonically in one-to-one correspondence with $\Pi_{\phi_{v}}(\tilde{T}_{z_{v}})$. Kaletha's construction of the LLC relies on establishing an   isomorphism between two extensions of $A(F_{v})^{[\phi_{v}],[z_{v}]}$ by $\mathbb{C}^{\times}$, where he used the Tate--Nakayama pairing for hypercohomology introduced in~\cite{kottwitz1999foundations} as a key ingredient. A minor drawback of this approach is a lack of transparency in the aforementioned isomorphism, due to the need to choose sections for the extensions.

In this paper, we follow Kaletha's idea and exploit the Tate--Nakayama mechanism. However, we are able to recast the LLC for disconnected tori in a more intrinsic manner using the notion of \textit{relatedness} (Definition~\ref{defn: relatedness}). This notion ties together a pair of representations based on their fibrewise values. To be precise, we say that
$\eta_{v}\in\mathrm{Irr}(\tilde{T}_{z_{v}}(F_{v})^{[\phi_{v}]}, [\phi_{v}])$ and $\rho_{v} \in \mathrm{Irr}(\pi_{0}(\tilde{S}_{\phi_{v}}^{[z_{v}]}),[z_{v}])$ are \textit{related} if they can be realised on the same vector space and, for any fixed $a \in A(F_{v})^{[\phi_{v}],[z_{v}]}$, satisfy
$$\rho_{v}(s,a)\,\eta_{v}(t,a^{-1}) = \langle (z_{v}^{-1},t), (\phi_{v}^{-1},s)\rangle_{\mathrm{TN}}^{-1},$$
where $(s,a) \in \pi_{0}(\tilde{S}_{\phi_{v}}^{[z_{v}]})$ and $(t,a^{-1}) \in \tilde{T}_{z_{v}}(F_{v})^{[\phi_{v}]}$ are arbitrary lifts of $a$ and $a^{-1}$, respectively. The resulting reinterpretation can be stated as follows.

\begin{thm}[($=$ Theorem~\ref{thm:llc_reinterthm} + \S \ref{section:comparison_with_kal})]
\textit{Relatedness} gives a bijection between $\mathrm{Irr}(\tilde{T}_{z_{v}}(F_{v})^{[\phi_{v}]},[\phi_{v}])$ and $\mathrm{Irr}(\pi_{0}(\tilde{S}_{\phi_{v}}^{[z_{v}]}),[z_{v}])$. This in turn provides a bijection $\Pi_{\phi_{v}}(\tilde{T}_{z_{v}}) \leftrightarrow \mathrm{Irr}(\pi_{0}(\tilde{S}_{\phi_{v}}^{[z_{v}]}),[z_{v}])$, which coincides with the local Langlands correspondence for $\tilde{T}_{z_{v}}$ constructed by Kaletha.
\end{thm}

\subsubsection{Global aspects}
We now return to the global setting over a number field $F$. Recall that a pure inner form $\tilde{T}_{z}:=(T\rtimes A)_{z}$ is fixed and $T$ is assumed to be anisotropic.

Denote by $H(T)$ the set of Hecke characters of $T$. The global Langlands correspondence for the (connected) torus $T$ gives a bijection between $H(T)$ and near-equivalence classes of global $L$-parameters for $T$. In the disconnected setting, we need to work with a weaker notion of equivalence called \textit{near-}$A(F)^{[z]}$\textit{-equivalence}, where we take into consideration the action of $A(F)^{[z]}\cong \tilde{T}_{z}(F)/T(F)$.

Let $\phi: W_{F} \to {}^{L}T$ be a global $L$-parameter for $T$ and let $\chi \in H(T)$ be the Hecke character  associated with $\phi$. Denote their local components at a place $v$ by $\phi_{v}$ and $\chi_{v}$, respectively.  We define the adelic $L$-packet of $\tilde{T}_{z}$ associated with $\phi$ as
$$\Pi_{\phi}(\tilde{T}_{z}) = \{\otimes'_{v}\;\eta_{v} \;\bigm|\; \eta_{v} \in \Pi_{\phi_{v}}(\tilde{T}_{z_{v}}), \;\iota_{v}(\eta_{v}) = \mathbb{1} \text{ for almost all\ } v\},$$
where $\iota_{v}: \Pi_{\phi_{v}}(\tilde{T}_{z_{v}}) \to \mathrm{Irr}(\pi_{0}(\tilde{S}_{\phi_{v}}^{[z_{v}]}), [z_{v}])$ is the bijection given by the LLC for $\tilde{T}_{z_{v}}$. We show that each irreducible constituent of $L^{2}([\tilde{T}_{z}])$ must lie in some $\Pi_{\phi}(\tilde{T}_{z})$. Let $\Phi(\eta)$ be the set of global $L$-parameters $\phi$ for which $\eta \in \Pi_{\phi}(\tilde{T}_{z})$. Then we finally prove that the total multiplicity of $\eta$ in $L^{2}([\tilde{T}_{z}])$ decomposes as a sum of contributions attached to $[[\phi]]$ with $\phi \in \Phi(\eta)$. Here, $[[\phi]]$ denotes the near-$A(F)^{[z]}$-equivalence class of $\phi$, and the contribution by $[[\phi]]$ is $m_{\eta,\phi}$ made precise in the theorem below. We remark that the notion of near-$A(F)^{[z]}$-equivalence (see Definition~\ref{defn: global_l_par_near_AFz_equiv}) is the analogue of the equivalence of global $L$-parameters defined in~\cite[\S10.4]{kottwitz1984stable}. See \S\ref{section:kottwitz_analogue} for the comparison.

Let $\tilde T_z(F)$ and $A(F)^{[z]}$ act on $H(T)$ by conjugation, and denote the stabilizer of $\chi$ in $\tilde T_z(F)$ and $A(F)^{[z]}$ by $\tilde T_z(F)^{\chi}$ and $A(F)^{[z],\chi}$, respectively. As pointed out in \S\ref{section:kottwitz_analogue}, the finite group $A(F)^{[z],\chi}$ is the analogue of Kottwitz's $\mathcal{S}_{\phi}$ in the conjectural formula~\eqref{eq:intro_mult}. We define a pairing $\langle\cdot,\cdot\rangle: A(F)^{[z],\chi} \times \Pi_{\phi}(\tilde{T}_{z}) \longrightarrow \mathbb{C}$ by
$$(a,\eta) \mapsto \prod_{v}\langle (z_{v}^{-1}, t), (\phi_{v}^{-1}, a^{-1}(s_{v}^{-1}))\rangle^{-1}_{\mathrm{TN}}\cdot \mathrm{tr}\left[\iota_{v}(\eta_{v})(s_{v},a)\right],
$$
where, given $a\in A(F)^{[z],\chi}$, $(s_{v},a) \in \tilde{S}_{\phi_{v}}^{[z_{v}]}$ and $(t,a) \in \tilde{T}_{z}(F)^{\chi}$ are arbitrarily chosen. When the torus $T$ satisfies the Hasse principle, the pairing admits a simpler expression that does not involve the Tate--Nakayama pairing (see \S\ref{subsec: hasse_simp}). 

\begin{thm}[($=$ Proposition~\ref{prop:pairing_welldef})]
The expression above does not depend on the auxiliary choices of $(s_{v},a)$ and $(t,a)$. Consequently, $(a,\eta)\mapsto\langle a,\eta\rangle$ is a well-defined pairing $\langle\cdot,\cdot\rangle\colon A(F)^{[z],\chi} \times \Pi_{\phi}(\tilde{T}_{z}) \to \mathbb{C}$.
\end{thm} 

Our intrinsic reinterpretation of the LLC (Theorem~\ref{thm:llc_reinterthm}) then identifies $\langle\cdot,\eta\rangle$ with the character of $\bar\eta|_{A(F)^{[z],\chi}}$ (Proposition~\ref{prop: pairing_ch_eta}), which combined with the Mackey-theoretic decomposition yields the main theorem:

\begin{thm}[(= Theorem~\ref{thm:finalthm})]
For $\eta \in \Pi_{\phi}(\tilde{T}_{z})$, the $[[\phi]]$-contribution to the multiplicity of $\eta$ is
\begin{align*}
     m_{\eta,\phi} = \frac{1}{|A(F)^{[z],\chi}|}\sum_{a\in A(F)^{[z],\chi}}\langle a,\eta\rangle.
\end{align*}
Let $\eta$ be a smooth irreducible admissible representation of $\tilde{T}_{z}(\mathbb{A})$. Then the multiplicity of $\eta$ in $L^{2}([\tilde{T}_{z}])$ is the sum of all the $[[\phi]]$-contributions:
\begin{align*}
    m_{\eta} = \sum_{[[\phi]]:\phi\in \Phi(\eta)}m_{\eta,\phi}, 
\end{align*}
where $[[\phi]]$ runs over the near-$A(F)^{[z]}$-equivalence classes in $\Phi(\eta)$. The sum is finite and $m_{\eta}<\infty$. Consequently, we have the Hilbert direct sum decomposition:
\begin{align*}
L^{2}([\tilde{T}_{z}]) = \widehat{\bigoplus}_{\eta}\ \eta^{\oplus m_{\eta}}.
\end{align*}
\end{thm}

\subsection{Structure of the paper}
Chapter~2 establishes the notation used throughout the paper. Chapter~3 reviews the (local and global) Tate--Nakayama duality for hypercohomology developed in~\cite[Appendices~A.3~\&~C]{kottwitz1999foundations}. In \S3.3 we prove a lemma that will play an important role in our reinterpretation of the LLC for disconnected tori. Chapter~4 reviews the setup for the disconnected groups from~\cite{kaletha2022local} and presents examples of disconnected tori.

Chapter~5 and Chapter~6 are devoted to the local aspects. Chapter~5 reviews the statement of the LLC for disconnected tori in terms of pure inner forms. In Chapter~6, we present our intrinsic construction of the LLC for disconnected tori, which we prove coincides with Kaletha's construction.

Chapter~7, Chapter~8 and Chapter~9 turn to the global aspects. After laying the necessary foundation for the global setting in Chapter~7, we obtain a preliminary spectral decomposition in \S8.2. Using Mackey theory, we calculate the multiplicities of each irreducible constituent on the automorphic side in \S8.4. In \S9.2 and \S9.3, we define adelic $L$-packets $\Pi_{\phi}(\tilde{T}_{z})$ and the pairing $A(F)^{[z],\chi} \times \Pi_{\phi}(\tilde{T}_{z}) \to \mathbb{C}$. Finally, in \S9.4, we establish the desired multiplicity formula. We conclude Chapter~9 with a comparison to the connected case, simplifications for tori satisfying the Hasse principle, and concrete examples. 

Appendix~A collects the necessary definitions and basic facts regarding group hypercohomology and hyperhomology.

\section{Notation}
In this paper, all fields considered are local or global fields of characteristic $0$. We denote the (absolute) Weil group of a (local or global) field $F$ by $W_{F}$, and write $\Gamma_{F}:=\mathrm{Gal}(\bar{F}/F)$ for its absolute Galois group. We abbreviate $\Gamma_{F}$ to $\Gamma$ whenever $F$ is clear from the context. If $v$ is a place of a global field $F$, we write $F_{v}$ for the completion of $F$ at $v$. 
We implicitly fix an algebraic closure $\overline{F}_{v}$ of $F_{v}$ and an embedding $\overline{F} \hookrightarrow \overline{F}_{v}$, which identifies the local Galois group $\Gamma_{v}:=\mathrm{Gal}(\bar{F}_{v}/F_{v})$ with a decomposition group of $\Gamma_{F}$. More generally, the local component at $v$ of a global object (e.g. a global $L$-parameter or a global Galois cocycle) is indicated by adding the subscript $v$. 

For an algebraic group $M$ defined over a field $F$, we abbreviate the continuous Galois cohomology $H^{i}(\Gamma_{F}, M(\overline{F}))$ to $H^{i}(F, M)$, and write $Z^{i}(F, M)$ for the corresponding group of cocycles. The same convention applies to each completion $F_{v}$.

Given a torus $T$ defined over a (local or global) field $F$, its cocharacter group $X_{*}(T) := \operatorname{Hom}_{\overline{F}}(\mathbb{G}_{\mathrm{m}}, T)$ is equipped with the natural Galois action. We denote the complex dual torus of $T$ by $\widehat{T}:= \operatorname{Hom}_{\mathbb{Z}}(X_{*}(T), \C^{\times})$, which carries the induced Galois action as well.

\section{The Tate--Nakayama duality for hypercohomology} \label{Chapter:TN-duality}
In this chapter, we review the local and global Tate--Nakayama Duality for hypercohomology. These results are due to Kottwitz--Shelstad~\cite{kottwitz1999foundations}. For the basics on group hyper(co)homology, including the notions of hyper(co)chains, hyper(co)cycles and hyper(co)boundaries used throughout this chapter, see Appendix~\ref{appendix: hyper}.

\subsection{Conventions}
Throughout this paper we adopt the following convention:
\begin{itemize}[label=\textbullet]
\item \textit{Cohomology:} (Hyper-)cohomology, (hyper-)cochains and (hyper-)cocycles of profinite groups and (absolute or relative) Weil groups should be understood in the \textbf{continuous} sense by default. When the abstract version is intended, we add the subscript \texttt{abs}.
\item \textit{Homology:} (Hyper-)homology, (hyper-)chains, and (hyper-)cycles for any group are always considered in the \textbf{abstract} sense.
\end{itemize}

\subsection{Local Tate--Nakayama duality for hypercohomology} \label{TNdualityforhyper}
Throughout this subsection, $F$ is a local field of characteristic zero, with absolute Galois group $\Gamma=\Gamma_{F}$. In~\cite{kottwitz1984stable}, Kottwitz reformulates the classical local Tate--Nakayama duality for a torus $T$ as an isomorphism
$$H^{1}(F,T) \cong \pi_{0}(\widehat{T}^{\,\Gamma})^{*},$$ 
which induces a pairing 
\begin{align}\label{eq: classical_TN_pairing}
\widehat{T}^{\,\Gamma} \times H^{1}(F,T) \to \mathbb{C}^{\times}.    
\end{align}
Similarly, the local Langlands correspondence for a torus $T$ 
\begin{align}\label{eq: classical_LLC_pairing}
H^{1}(W_{F}, \widehat{T})\cong \mathrm{Hom}_{\mathrm{cts}}(T(F),\mathbb{C}^{\times})    
\end{align}
gives rise to a pairing
$$T(F)\times H^{1}(W_{F}, \widehat{T}) \to \mathbb{C}^{\times}.$$
The above two pairings can be unified into the Tate–Nakayama pairing for hypercohomology. This duality for hypercohomology encompasses simultaneously the LLC for tori and the classical local Tate--Nakayama duality. We only sketch the outline of the duality construction here. For a detailed account, we refer the reader to pp. 122--139 of~\cite{kottwitz1999foundations}. 

Let $T$ and $U$ be $F$-tori with cocharacter groups $X$ and $Y$. Let $f:T \to U$ be a morphism defined over $F$. Let $f_{*}:X\to Y$ and $\hat{f}:\widehat{U}\to \widehat{T}$ be the maps induced by $f$. We aim to define a pairing between $H^{1}(W_{F},\widehat{U}\xrightarrow{\hat{f}}\widehat{T})$ and $H^{1}(F,T \xrightarrow{f} U)$. Throughout this section,  we fix a finite Galois extension $K$ of $F$ such that both $T$ and $U$ split over $K$. We denote the relative Weil group by $W_{K/F}$. 

\subsubsection{The isomorphism $\mathcal{H}$}
First, we recall the definition of the modified hyperhomology group $$H_{0}(W_{K/F}, X \to Y)_{0}.$$
Throughout, the chain groups $C_{\bullet}$ and the associated hyper(co)cycles and hyper(co)boundaries are understood as defined in Appendix~\ref{appendix: hyper}.
We start by noting that the image of the differential
\begin{align*}
C_{1}(W_{K/F},X)\xrightarrow{\partial}C_{0}(W_{K/F},X)
\end{align*}
lies in $C_{0}(W_{K/F},X)_{0}$, the subgroup of norm-$0$ elements in $C_{0}(W_{K/F},X) = X$. 

Given $X \xrightarrow{f_{*}} Y$, the following complex is used in defining  hyperhomology:
\begin{align}
    \cdots \to C_{1}(W_{K/F},X)\oplus C_{2}(W_{K/F},Y) \xrightarrow{\alpha} C_{0}(W_{K/F},X)\oplus C_{1}(W_{K/F},Y) \xrightarrow{\beta} C_{0}(W_{K/F},Y), 
\end{align}
where $\alpha(x,y) = (\partial x, f_{*}(x)-\partial y)$ and $\beta(x,y) = f_{*}(x) - \partial y$. The group $H_{0}(W_{K/F}, X \to Y)$ is defined as the quotient $\mathrm{ker} \,\beta/\mathrm{im}\,\alpha$. We define $(\mathrm{ker}\,\beta)_{0}$ to be the subgroup consisting of pairs $(x,y)$ in $\mathrm{ker}\,\beta$ with $x \in C_{0}(W_{K/F},X)_{0}$. According to the previous paragraph, we have $\mathrm{im}\,\alpha \subseteq (\mathrm{ker}\,\beta)_{0}$. Now we define the modified hyperhomology group
\begin{align}
    H_{0}(W_{K/F}, X \to Y)_{0} := (\mathrm{ker}\,\beta)_{0}/\,
\mathrm{im}\,\alpha\,,
\end{align}
and the usual long exact sequence associated with hyperhomology can be modified as:
\begin{align*}
H_{1}(W_{K/F},X) \to H_{1}(W_{K/F},Y)  \to H_{0}(W_{K/F}, X \to Y)_{0} \to H_{0}(W_{K/F},X)_{0} \to H_{0}(W_{K/F},Y)_{0}\,.
\end{align*}
Moreover, the above long exact sequence can be related to the one associated with hypercohomology of the Galois group $\mathrm{Gal}(K/F)$ (abbreviated as $K/F$ below) such that the following diagram commutes:
\begin{equation}\label{diagram: isom_H}
\adjustbox{scale=0.8,center}{%
\begin{tikzcd}
    H_{1}(W_{K/F},X) \ar[r] \ar["\mathcal{D}",d] & H_{1}(W_{K/F},Y) \ar[r]\ar["\mathcal{D}",d] & H_{0}(W_{K/F}, X \to Y)_{0} \ar[r] \ar["\mathcal{H}",d] & H_{0}(W_{K/F},X)_{0} \ar[r]\ar["TN",d] & H_{0}(W_{K/F},Y)_{0} \ar["TN",d]\\
    H^{0}(K/F,T(K)) \ar[r] & H^{0}(K/F,U(K)) \ar[r] &
    H^{1}(K/F,T(K) \to U(K)) \ar[r] &  H^{1}(K/F,T(K)) \ar[r] &  H^{1}(K/F,U(K)). 
\end{tikzcd}
}
\end{equation}
All the vertical arrows above are isomorphisms. $\mathcal{D}$ is a key isomorphism used in the construction of the LLC for $T$ (following Deligne's convention, see Remark~\ref{rmk: Deligne's Convention}), while $TN$ is the classical Tate--Nakayama isomorphism (note $H_{0}(W_{K/F},X)_{0} \cong \widehat{H}^{-1}(K/F,X)$, the Tate cohomology group at degree $-1$). Using the (co)chain-level maps that induce $\mathcal{D}$ and $TN$, one can define a map on the level of hyperchains and hypercochains, which induces the central isomorphism 
\begin{align} \label{isom: H}
\mathcal{H}:H_{0}(W_{K/F}, X \to Y)_{0} \xrightarrow{\sim} H^{1}(\mathrm{Gal}(K/F), T(K) \to U(K)). 
\end{align}

\begin{rmk}\label{rmk: Deligne's Convention}
We remark that there are two normalisations of LLC for the torus $T$. Indeed, if $[\phi]\mapsto \chi_{[\phi]}$ is the LLC for tori under Langlands' convention~\cite{langlands1997representations}, then the map $[\phi]\mapsto \chi_{[\phi]}^{-1}$ is clearly an isomorphism from $H^{1}(W_{K/F},\widehat{T})$ to $ \mathrm{Hom}_{\mathrm{cts}}(T(F),\mathbb{C}^{\times})$ as well, which we call ``the LLC under Deligne's convention''. 

The discrepancy between these two conventions usually lies in different ways of identifying $H_{1}(W_{K/F}, X)$ with $T(F)$. In Deligne's convention, the key isomorphism $\mathcal{D}:H_{1}(W_{K/F},X) \cong T(F)$ is constructed as follows:
\begin{align*}
\mathcal{D}:H_{1}(W_{K/F},X) \xrightarrow{\mathrm{Res}} H_{1}(K^{\times}, X) &\xrightarrow{\sim_{D}} T(K)\\ 
[x] &\mapsto \prod_{a\in K^{\times}}x_{a}(a)^{-1}.
\end{align*}
In contrast, in Langlands' convention, the key isomorphism $\mathcal{L}:H_{1}(W_{K/F},X) \xrightarrow{\sim} T(F)$ equals to $\sim_{L} \circ \, \mathrm{Res}$, where $\displaystyle \sim_{L}([x]) = \prod_{a\in K^{\times}}x_{a}(a)$.

In this work, unless specified otherwise, whenever we say ``LLC for tori'', we mean the LLC under Langlands' convention. However, we warn the reader that Deligne's convention is the one adopted in~\cite[Appendix~A.3]{kottwitz1999foundations}. In particular, following~\cite{kottwitz1999foundations}, the duality construction that we are reviewing involves $\mathcal{D}$ rather than $\mathcal{L}$.
\end{rmk}

\subsubsection{The Tate-Nakayama pairing for hypercohomology}
Let $\mathcal{P}_{\bullet}$ be the defining complex of $H_{0}(W_{K/F}, X\to Y)$. Due to injectivity of $\mathbb{C}^{\times}$, applying the functor $\mathrm{Hom}(-,\mathbb{C}^{\times})$ commutes with taking (co)homology. Thus, we have $\displaystyle \mathrm{Hom}(H_{\bullet}(\mathcal{P}_{\bullet}),\mathbb{C}^{\times}) \cong
H^{\bullet}(\mathrm{Hom}(\mathcal{P}_{\bullet}, \mathbb{C}^{\times}))$. In particular, we obtain
\begin{align*}
\mathrm{Hom}(H_{0}(W_{K/F},X \to Y),\mathbb{C}^{\times}) \cong H^{1}_{\mathrm{abs}}(W_{K/F}, \widehat{U}\to \widehat{T}),
\end{align*}
where the right-hand side is the abstract hypercohomology, defined in terms of abstract cochains (ignoring the topology on $W_{K/F}$). We can explicate the isomorphism as a pairing between $H_{0}(W_{K/F},X\to Y)$ and $H^{1}_{\mathrm{abs}}(W_{K/F}, \widehat{U}\to \widehat{T})$. Let $(x, y_{w})$ in $Z_{0}(W_{K/F},X\to Y)$ be a 0-hypercycle, i.e. $f_{*}(x) = \partial y_{w} = \sum (w^{-1}y_{w}-y_{w})$. Let $(u_{w},t)$ be an abstract 1-hypercocycle, where $u_{w} \in Z^{1}_{\mathrm{abs}}(W_{K/F},\widehat{U})$, $t \in \widehat{T}$ and $\hat{f}(u_{w})= \partial t = t^{-1}w(t)$. The pairing between $(x,y_{w})$ and $(u_{w},t)$ is given by
\begin{align}\label{formula:def_of_TNhyperpairing}
    \langle(x,y_{w}), (u_{w},t)\rangle = \langle x, t \rangle \prod_{w \in W_{K/F}} \langle y(w), u(w) \rangle^{-1}.
\end{align}
This pairing descends to hyper(co)homology. Moreover, we can restrict the pairing to the subgroups
\begin{align*}
H_{0}(W_{K/F},X \to Y)_{0} &\subseteq H_{0}(W_{K/F},X \to Y)\\
\intertext{and}
H^{1}(W_{K/F}, \widehat{U}\to \widehat{T}) &\subseteq H^{1}_{\mathrm{abs}}(W_{K/F}, \widehat{U}\to \widehat{T}),
\end{align*}
where $H^{1}(W_{K/F}, \widehat{U}\to \widehat{T})$ denotes the usual continuous hypercohomology. In view of the isomorphism $\mathcal{H}$~\eqref{isom: H}, we obtain a pairing
\begin{align*}
H^{1}(K/F, T \to U) \times H^{1}(W_{K/F}, \widehat{U}\to \widehat{T}) \to \mathbb{C}^{\times}.
\end{align*}
One can show that this pairing is compatible with inflation, hence after passing to colimits, we have actually obtained the pairing 
\begin{align}\label{TN-hyper-pairing}
H^{1}(F, T \to U) \times H^{1}(W_{F}, \widehat{U}\to \widehat{T}) \to \mathbb{C}^{\times}. 
\end{align}

\begin{prop}[\cite{kottwitz1999foundations}] \label{TN-LLCcompatibility}
The pairing~\eqref{TN-hyper-pairing} is compatible with the classical Tate--Nakayama pairing~\eqref{eq: classical_TN_pairing} and the LLC pairing (in Langlands' convention)~\eqref{eq: classical_LLC_pairing}. Precisely speaking, if we pair the long exact sequence on the group side with its counterpart on the dual side:
\begin{center}
\begin{tikzcd}
\ar[r] & H^{0}(F, U) \ar[r, "i"] \ar[drr,dotted, leftrightarrow] & H^{1}(F, T\to U) \ar[r, "p"] \ar[d, dotted, leftrightarrow] & H^{1}(F, T) \ar[r] \ar[dll, dotted, leftrightarrow] & {}\\
\ar[r] & H^{0}(W_{F},\widehat{T}) \ar[r, "\widehat{i}"]  & H^{1}(W_{F},\widehat{U}\to \widehat{T}) \ar[r, "\widehat{p}"]  & H^{1}(W_{F},\widehat{U}) \ar[r] & {},
\end{tikzcd}
\end{center}
then we have
\begin{align*}
\langle i(u), \hat{c} \rangle &= \langle u, \hat{p}(\hat{c}) \rangle\\
\intertext{and}
\langle c, \hat{i}(s) \rangle &= \langle p(c), s \rangle,
\end{align*}
for any $u\in H^{0}(F,U)$, $\hat{c}\in H^{1}(W_{F},\widehat{U}\to \widehat{T})$, $c\in H^{1}(F, T \to U)$ and $s\in H^{0}(W_{F},\widehat{T})$.
\end{prop}

\subsubsection{Duality}
In this part we explain the duality induced by the pairing~\eqref{TN-hyper-pairing}. Recall the long exact sequence
\begin{align*}
    \cdots \to T(F) \to U(F) \xrightarrow{i} H^{1}(F,T\xrightarrow{f} U) \to H^{1}(F,T) \to H^{1}(F, U) \to \cdots.
\end{align*}
The hypercohomology group $H^{1}(F,T\xrightarrow{f} U)$ can be topologized in such a way that any character of $H^{1}(F,T\xrightarrow{f} U)$ is continuous if and only if it is continuous on $i(U(F))$. By Proposition~\ref{TN-LLCcompatibility}, the pairing~\eqref{TN-hyper-pairing} is continuous on $i(U(F))$ indeed, and hence induces the map
\begin{align*}
H^{1}(W_{F}, \widehat{U}\to \widehat{T}) \to   \mathrm{Hom}_{\mathrm{cts}}(H^{1}(F,T\xrightarrow{f} U),\mathbb{C}^{\times}).
\end{align*}

By virtue of Proposition~\ref{TN-LLCcompatibility}, the following diagram commutes:
\begin{center}
\begin{tikzcd}
H^{1}(F,U)^{'} \arrow[r]          & H^{1}(F,T)^{'} \arrow[r]          & {H^{1}(F,T\xrightarrow{f} U)'} \arrow[r]     & U(F)' \arrow[r]                    & T(F)'                    \\
\widehat{U}^{W_{F}} \arrow[u] \arrow[r] & \widehat{T}^{W_{F}} \arrow[u] \arrow["\hat{i}",r] & H^{1}(W_{F}, \widehat{U}\xrightarrow{\hat{f}} \widehat{T}) \arrow[u] \arrow[r] & H^{1}(W_{F},\widehat{U}) \arrow[u] \arrow[r] & H^{1}(W_{F},\widehat{T}) \arrow[u].
\end{tikzcd}
\end{center}
In the above diagram, we denoted $\mathrm{Hom}_{\mathrm{cts}}(-,\mathbb{C}^{\times})$ by $'$ for brevity. The two vertical maps to the right are the LLC for tori, while the two vertical maps to the left come from the Tate--Nakayama pairing and are surjective with kernels $(\widehat{U}^{W_{F}})^{\circ}$ and $(\widehat{T}^{W_{F}})^{\circ}$, the identity components of Galois-invariant subgroup of the dual tori $\widehat{U}$ and $\widehat{T}$, respectively.

Define the quotient
\begin{align*}
H^{1}(W_{F}, \widehat{U}\xrightarrow{\hat{f}} \widehat{T})_{\mathrm{red}}:= \frac{H^{1}(W_{F}, \widehat{U}\xrightarrow{\hat{f}} \widehat{T})}{\hat{i}((\widehat{T}^{W_{F}})^{\circ})},
\end{align*}
and modify the above commutative diagram into
\begin{center}
\begin{tikzcd}
H^{1}(F,U)^{'} \arrow[r]          & H^{1}(F,T)^{'} \arrow[r]          & {H^{1}(F,T\xrightarrow{f} U)'} \arrow[r]     & U(F)' \arrow[r]                    & T(F)'                    \\
\pi_{0}(\widehat{U}^{W_{F}}) \arrow[u] \arrow[r] & \pi_{0}(\widehat{T}^{W_{F}}) \arrow[u] \arrow["\hat{i}",r] & H^{1}(W_{F}, \widehat{U}\xrightarrow{\hat{f}} \widehat{T})_{\mathrm{red}} \arrow[u] \arrow[r] & H^{1}(W_{F},\widehat{U}) \arrow[u] \arrow[r] & H^{1}(W_{F},\widehat{T}) \arrow[u].
\end{tikzcd}
\end{center}
By the five lemma, the vertical map in the middle must be an isomorphism, because all the other four are so. This concludes the construction of Tate--Nakayama duality for hypercohomology:
\begin{prop}[\cite{kottwitz1999foundations}] \label{prop:TN-main-thm}
The pairing~\eqref{TN-hyper-pairing} induces a functorial isomorphism
\begin{align*}
H^{1}(W_{F}, \widehat{U}\xrightarrow{\hat{f}} \widehat{T})_{\mathrm{red}} \cong   \mathrm{Hom}_{\mathrm{cts}}(H^{1}(F,T\xrightarrow{f} U),\mathbb{C}^{\times}).
\end{align*}
Here, functoriality means the following: Let $\alpha \colon (T \xrightarrow{f} U) \to (V \xrightarrow{g} W)$ be a morphism of complexes of $F$-tori, and let its dual morphism be $\hat{\alpha} \colon (\widehat{W} \xrightarrow{\hat{g}} \widehat{V}) \to (\widehat{U} \xrightarrow{\hat{f}} \widehat{T})$. Let $\alpha_{*}$ and $\hat{\alpha}_{*}$ denote the induced maps on hypercohomology. For any $c \in H^{1}(F,T\xrightarrow{f} U)$ and $\hat{c} \in H^{1}(W_{F}, \widehat{W}\xrightarrow{\hat{g}} \widehat{V})_{\mathrm{red}}$, we have
$$ \langle c, \hat{\alpha}_*(\hat{c}) \rangle_{T \to U} = \langle \alpha_*(c), \hat{c} \rangle_{V \to W}. $$
\end{prop}

\subsection{A cohomological lemma}
For later convenience, we record an interesting result, which will play an important role in our reinterpretation of the local Langlands correspondence for disconnected tori (Theorem~\ref{thm:llc_reinterthm}) and its proof.

\begin{thm} \label{thm:appoffunctor}
Let  $T$, $U$ and $V$ be $F$-tori, and let $f$ and $g$ be $F$-morphisms:
\begin{center}
\begin{tikzcd}
    T \arrow{r}{f} &U \arrow{r}{g} &V.
\end{tikzcd}
\end{center}
Let $\hat{f}$ and $\hat{g}$ be the induced morphisms between the dual tori:
\begin{center}
\begin{tikzcd}
    \widehat{T}& \arrow[l,"\hat{f}"']\widehat{U}& \arrow[l,"\hat{g}"'] \widehat{V}.
\end{tikzcd}
\end{center}
We consider the natural map induced by $g$
\begin{center}
\begin{tikzcd}
H^{1}(F, T\xrightarrow{f} U) \arrow[r,"g_{*}"] &H^{1}(F, T\xrightarrow{g f} V),
\end{tikzcd}
\end{center}
sending the class of $(z,t)$ to that of $(z, g(t))$, and the natural map induced by $\hat{f}$
\begin{center}
\begin{tikzcd}
H^{1}(W_{F}, \widehat{V}\xrightarrow{\hat{g}} \widehat{U}) \arrow[r,"\hat{f}_{*}"] &H^{1}(W_{F}, \widehat{V}\xrightarrow{\hat{f}\hat{g}} \widehat{T}),
\end{tikzcd}
\end{center}
sending the class of $(\phi, s)$ to that of $(\phi, \hat{f}(s))$. Then the Tate--Nakayama pairing between the image of $g_{*}$ and the image of $\hat{f}_{*}$ must vanish:
\begin{align*}
\langle g_{*}(z,t), \hat{f}_{*}(\phi,s) \rangle = 1,
\end{align*}
for any $(z,t)$ and any $(\phi,s)$. 
\end{thm}

\begin{proof}
    Let $\alpha: (T \xrightarrow{f} U) \to (T \xrightarrow{g f} V)$ be the morphism of complexes induced by $U\xrightarrow{g}V$. Then the dual morphism of $\alpha$ is $\hat{\alpha}: (\widehat{V} \xrightarrow{\hat{f}\hat{g}} \widehat{T}) \to (\widehat{U} \xrightarrow{\hat{f}} \widehat{T})$, induced by $\widehat{V}\xrightarrow{\hat{g}}\widehat{U}$. Let $\alpha_{*}$ and $\hat{\alpha}_{*}$ denote the induced maps on hypercohomology. At the level of hypercocycles, we have $\alpha_{*}(z, t) = (z, g(t)) = g_{*}(z,t)$ and $\hat{\alpha}_{*}(\phi, \hat{f}(s)) = (\hat{g}\circ \phi, \hat{f}(s))$.
    
    By the functoriality of the Tate-Nakayama pairing for hypercohomology (see Proposition~\ref{prop:TN-main-thm}), we have 
    $$ \langle g_{*}(z,t), \hat{f}_{*}(\phi,s) \rangle = \langle \alpha_{*}(z, t), (\phi, \hat{f}(s))\rangle = \langle (z,t) , \hat{\alpha}_{*}(\phi, \hat{f}(s))  \rangle = \langle (z,t) , (\hat{g}\circ \phi, \hat{f}(s)) \rangle.
    $$
Now, we note that $(\phi, s) \in Z^{1}(W_{F}, \widehat{V}\xrightarrow{\hat{g}} \widehat{U})$ implies $\hat{g}\circ \phi = \partial s$ by the definition of $1$-hypercocycles. Therefore, $(\hat{g}\circ \phi, \hat{f}(s)) = (\partial s, \hat{f}(s))$ lies in $B^{1}(W_{F}, \widehat{U}\xrightarrow{\hat{f}} \widehat{T})$, the group of $1$-hypercoboundaries. Because the Tate--Nakayama pairing is evaluated on cohomology classes, pairing any class with a hypercoboundary yields the trivial value $1$. This implies 
$$\langle (z,t) , (\hat{g}\circ \phi, \hat{f}(s)) \rangle = 1.$$
\end{proof}

\subsection{Global Tate--Nakayama duality for hypercohomology}
Like its local counterpart, the global Tate--Nakayama Duality for hypercohomology combines the global Langlands correspondence for tori with the classical global Tate--Nakayama duality. The constructions can be carried out in the same way as the local case (the idele class group plays the same role as the multiplicative group plays in the local setting).

Let $F$ be a number field. Let $T$ and $U$ be $F$-tori. Let $f:T\to U$ be an $F$-morphism. In the global setting, we consider the following Galois hypercohomology:
\begin{align*}
H^{i}(\mathbb{A}/F,T \xrightarrow{f} U) &:= H^{i}(F,T(\bar{\mathbb{A}})/T(\bar{F})\xrightarrow{f}U(\bar{\mathbb{A}})/U(\bar{F})),
\end{align*}
which can be topologized in a similar manner as the local case (see~\cite[Appendix~C]{kottwitz1999foundations} for details).

On the dual side, one can define the continuous hypercohomology $H^{1}(W_{F},\widehat{U}\xrightarrow{\hat{f}}\widehat{T})$ and its reduced version $H^{1}(W_{F},\widehat{U}\xrightarrow{\hat{f}}\widehat{T})_{\mathrm{red}}$ similarly as in the local setting. Now, we state the global Tate--Nakayama duality for hypercohomology: 
\begin{thm}[\cite{kottwitz1999foundations}]\label{thm: TN-duality-Global}
There is a natural functorial isomorphism
\begin{align*}
H^{1}(W_{F},\widehat{U}\xrightarrow{\hat{f}}\widehat{T})_{\mathrm{red}} \cong \mathrm{Hom}_{\mathrm{cts}}(H^{1}(\mathbb{A}/F,T \xrightarrow{f} U), \mathbb{C}^{\times}) 
\end{align*}
that is compatible with the global Langlands correspondence and the classical global Tate--Nakayama duality. Moreover, this isomorphism is compatible with the local Tate--Nakayama duality for hypercohomology, in the sense that the following diagram commutes:
\begin{center}
\begin{tikzcd}
H^{1}(W_{F}, \widehat{U}\to\widehat{T})_{\mathrm{red}} \ar[r,"\sim"]\ar[d] & \mathrm{Hom}_{\mathrm{cts}}(H^{1}(\mathbb{A}/F,T\to U),\mathbb{C}^{\times}) \ar[d]\\ 
\prod_{v} H^{1}(W_{F_{v}}, \widehat{U}\to \widehat{T})_{\mathrm{red}}
\ar[r, "\sim"] & \prod_{v} \mathrm{Hom}_{\mathrm{cts}}(H^{1}(F_{v}, T\to U),\mathbb{C}^{\times}).
    \end{tikzcd}
\end{center}
Suppose $c \in H^{1}(\mathbb{A}/F,T\to U)$ is represented by an adelic hypercocycle with local components $c_{v}\in H^{1}(F_{v},T\to U)$, unramified for almost all $v$. Then, for any $\hat{c}\in H^{1}(W_{F},\widehat{U}\to\widehat{T})_{\mathrm{red}}$ with localizations $\hat{c}_{v}$, almost all factors below equal $1$ and
\begin{align*}
\langle c, \hat{c}\rangle = \prod_{v}\langle c_{v}, \hat{c}_{v}\rangle\,.
\end{align*}
\end{thm}

\section{Disconnected reductive groups}
We will closely follow~\cite{kaletha2022local} and introduce a certain class of disconnected groups, for which we reserve the term ``disconnected reductive groups''. Due to our emphasis in this paper, we will present some examples of rank-$1$ disconnected tori.

\subsection{Convention}\label{convention}
Let $F$ be a field of characteristic $0$ with absolute Galois group $\Gamma := \mathrm{Gal}(\bar{F}/F)$. In this paper, we say an affine algebraic group $\tilde{G}$ defined over $F$ is a disconnected reductive group if $\tilde{G}$ satisfies the following conditions: 
\begin{itemize}[label=\textbullet]
    \item There is an isomorphism defined over $\bar{F}$:
    \begin{align*}
    \tilde{G} \xrightarrow{\sim} G \rtimes A,
    \end{align*} 
    where $G$ is a connected reductive group and $A$ is a finite group scheme.
    \item  The action of $A$ on $G$ preserves some fixed $\bar{F}$-pinning of $G$.
\end{itemize}

Here, we allow $A = \{1\}$ so that connected groups are naturally incorporated in this framework.

\begin{noneg}
Let $T \subset SL_2$ be a maximal torus and let $N(T)$ be the normalizer of $T$ in ${SL}_{2}$. Then $N(T)$ is  disconnected with identity component $T$. However, the non-identity component of $N(T)(\bar{F})$ contains no order-$2$ element. Hence, $N(T)$ does not split as a semi-direct product, even over $\bar{F}$. It does not fall into the class of disconnected reductive groups (or disconnected tori) treated in this paper.
\end{noneg}

\subsection{Classification}
Following~\cite{kaletha2022local}, we recall the classification of disconnected reductive groups. It is well-known that each connected reductive group has a unique split form and moreover, a unique quasi-split inner form. These notions can be extended to the disconnected setting. There turns out to be an additional type called \textit{translation form}, in the term of~\cite{kaletha2022local}. 

First, we extend the notions of \textit{quasi-splitness} and \textit{splitness}:
\begin{defn}
We call $\tilde{G}$ a split disconnected reductive group if there is an $F$-isomorphism
\begin{align*}
    \tilde{G} \to G\rtimes A,
\end{align*}
where $G$ is a split connected reductive group, and $A$ is a constant group scheme that acts on $G$ by preserving some $F$-pinning of it.
\end{defn}

\begin{defn}
We call $\tilde{G}$ a quasi-split disconnected reductive group if there is an $F$-isomorphism
\begin{align*}
    \tilde{G} \to G\rtimes A,
\end{align*}
where $G$ is a quasi-split connected reductive group, and $A$ is a (not necessarily constant) finite group scheme that acts on $G$ (where the action is defined over $F$) and preserves some $F$-pinning of it.
\end{defn}

We note that split disconnected reductive groups can be classified by the root data of $G$ along with the $A$-action on it. By definition, each disconnected reductive group has a unique split form. To classify all disconnected reductive groups, it suffices to classify the forms of a given split disconnected reductive group $G\rtimes A$. Equivalently, we hope to understand the Galois cohomology ${H}^{1}(\Gamma, \mathrm{Aut}_{\bar{F}}(G \rtimes A))$. We omit the subscript $\bar{F}$ and simply write $\mathrm{Aut}(G \rtimes A)$ for brevity.

There are three distinguished subgroups of $\mathrm{Aut}(G \rtimes A)$:
\begin{itemize}[label = \textbullet]
    \item $G/{Z(G)}^{A}$, the group of inner automorphisms. 
    \item $Z^{1}(A, Z(G))$, the group of translation automorphisms. Each $1$-cocycle $z \in Z^{1}(A, Z(G))$ induces an automorphism $(g,a) \mapsto (z(a)g,a)$.
    \item $\mathrm{Aut}_{\mathrm{pin}}(G \rtimes A)$, the group of pinned automorphisms over $\bar{F}$, by which we mean automorphisms that preserve $1\rtimes A$ as well as the pinning on $G$.
\end{itemize}

Note that the intersection of $G/{Z(G)}^{A}$ with $Z^{1}(A, Z(G))$ can be expressed as $Z(G)/{Z(G)}^{A}$ (as a subgroup of $G/{Z(G)}^{A}$) or $B^{1}(A, Z(G))$ (as a subgroup of $Z^{1}(A, Z(G))$). 

\begin{fact}[{\cite[\S3.1]{kaletha2022local}}]\label{structureofAut}
There is a semi-direct product decomposition:
\begin{align*}
\mathrm{Aut}(G \rtimes A) = (G/{Z(G)}^{A} \cdot Z^{1}(A, Z(G))) \rtimes \mathrm{Aut}_{\mathrm{pin}}(G \rtimes A).
\end{align*}
\end{fact}

Based on this fact, as explained in \cite[\S3.1]{kaletha2022local}, any general disconnected reductive group arises by twisting the rational structure of its split form in the following manner. First, any quasi-split disconnected reductive group $\tilde{G}$ is a twist of its split form $G\rtimes A$ via some $1$-cocycle in $Z^{1}(\Gamma, \mathrm{Aut}_{\mathrm{pin}}(G\rtimes A))$. Second, the quasi-split group $\tilde{G}$ can be twisted into an inner form $\tilde{G}_{\bar{z}}$, via a $1$-cocycle $\bar{z} \in Z^{1}(\Gamma, G/Z(G)^{A})$. Third, any general disconnected reductive group is a translation form of some inner form $\tilde{G}_{\bar{z}}$, which is obtained by twisting $\tilde{G}_{\bar{z}}$ via some $z'\in Z^{1}(\Gamma, Z^{1}(A,Z(G)))$.

\subsection{Examples of rank-1 disconnected tori}
In this paper, a disconnected torus is a group that satisfies Convention~\ref{convention} with $G=T$ being a torus.

To provide a glimpse into disconnected tori and particularly, their rational points, we work out instances of the simplest nontrivial case,  where the identity component is a rank-$1$ torus and the component group is simply $\mathbb{Z}/2\mathbb{Z}$. In this case, there are only two split forms: the direct product $\mathbb{G}_{m} \times \mathbb{Z}/2\mathbb{Z}$, and the semi-direct product $\mathbb{G}_{m} \rtimes \mathbb{Z}/2\mathbb{Z}$ where $\mathbb{Z}/2\mathbb{Z}$ acts on $\mathbb{G}_{m}$ by inverting. We classify their forms and describe their $F$-rational points. Write $\mathbb{Z}/2\mathbb{Z} = \{1,-1\}$.

\subsubsection{Forms of $\mathbb{G}_{m} \times \mathbb{Z}/2\mathbb{Z}$}
Let $\tilde{T} = \mathbb{G}_{m} \times \mathbb{Z}/2\mathbb{Z}$. The automorphism group of $\mathbb{G}_{m} \times \mathbb{Z}/2\mathbb{Z}$ is
\begin{align*}
\mathrm{Aut}(\tilde{T}) = Z^{1}(\mathbb{Z}/2\mathbb{Z},\bar{F}^{\times}) \rtimes \mathrm{Aut}_{\mathrm{pin}}(\mathbb{G}_{m} \times \mathbb{Z}/2\mathbb{Z}) \cong \mathbb{Z}/2\mathbb{Z} \times \mathbb{Z}/2\mathbb{Z}.
\end{align*}
We write the nontrivial element in the first (resp. second) $\mathbb{Z}/2\mathbb{Z}$ as $\mu$ (resp. $\omega$). As automorphisms of $\tilde{T}$, $\mu$ fixes the identity component point-wise and sends any point $(x,-1)$ from the non-identity component to $(-x,-1)$, while $\omega$ sends $(x,\pm1)$ to $(x^{-1},\pm1)$. Clearly, the Galois action on $\mathrm{Aut}(\tilde{T})$ is trivial. Thus, the forms of $\mathbb{G}_{m} \times \mathbb{Z}/2\mathbb{Z}$ are classified by the Galois cohomology
\begin{align*}
H^{1}(\Gamma,\mathbb{Z}/2\mathbb{Z} \times \mathbb{Z}/2\mathbb{Z}) = Z^{1}(\Gamma,\mathbb{Z}/2\mathbb{Z} \times \mathbb{Z}/2\mathbb{Z}) = \mathrm{Hom}_{\mathrm{cts}}(\Gamma,\mathbb{Z}/2\mathbb{Z} \times \mathbb{Z}/2\mathbb{Z}).
\end{align*}
Any nontrivial cocycle $z \in Z^{1}(\Gamma,\mathbb{Z}/2\mathbb{Z} \times \mathbb{Z}/2\mathbb{Z})$ falls into one of the four categories described below, and we denote by $\tilde{T}_{z}$ the group obtained from twisting $\tilde{T}$ by $z$.
\begin{itemize}[label=\textbullet]
\item \textbf{Case 1}: $z$ factors through a quadratic extension $E/F$ and maps onto the first $\mathbb{Z}/2\mathbb{Z}$-factor: $\Gamma \twoheadrightarrow \mathrm{Gal}(E/F) \xrightarrow{\sim} \{1,\mu\}$. Let $\mathrm{Gal}(E/F) = \{1,\sigma\}$. Then after passing to $E$, the Galois action  twisted by $z$ (which we denote by adding a subscript $z$) is given by $\sigma_{z}(x,1) = (\sigma(x),1)$ and $\sigma_{z}(x,-1) = (-\sigma(x),-1)$. The group of rational points is 
\begin{align*}
\tilde{T}_{z}(F) = \{(x,1)|x \in F^{\times}\} \cup \{(x,-1)|x \in E^{\times}, \sigma(x) = -x\}.
\end{align*}

\item \textbf{Case 2}: $z$ factors through a quadratic extension $E/F$ and maps onto the second $\mathbb{Z}/2\mathbb{Z}$-factor: $\Gamma \twoheadrightarrow \mathrm{Gal}(E/F) \xrightarrow{\sim} \{1,\omega\}$. Let $\mathrm{Gal}(E/F) = \{1,\sigma\}$. Then after passing to $E$, the Galois action  twisted by $z$ is given by $\sigma_{z}(x,\pm1) = (\sigma(x)^{-1},\pm1)$. Thus we have
\begin{align*}
\tilde{T}_{z}(F) = \{(x,\pm1)|x\in E^{\times}, \mathrm{Nm}_{E/F}(x) = 1\}.
\end{align*}

\item \textbf{Case 3}: $z$ factors through a quadratic extension $E/F$ and maps into $\mathbb{Z}/2\mathbb{Z} \times \mathbb{Z}/2\mathbb{Z}$ diagonally: $\Gamma \twoheadrightarrow \mathrm{Gal}(E/F) \xrightarrow{\sim} \{1,\mu\omega\}$. Let $\mathrm{Gal}(E/F) = \{1,\sigma\}$. Then after passing to $E$, the Galois action  twisted by $z$  is given by $\sigma_{z}(x,1) = (\sigma(x)^{-1}, 1)$ and $\sigma_{z}(x,-1) = (-\sigma(x)^{-1}, -1)$. So we have
\begin{align*}
\tilde{T}_{z}(F) = \{(x,1)|x\in E^{\times}, \mathrm{Nm}_{E/F}(x) = 1\}\cup \{(x,-1)|x\in E^{\times}, \mathrm{Nm}_{E/F}(x) = -1 \}.
\end{align*}

\item \textbf{Case 4}: $z$ factors through a biquadratic extension $K/F$ and maps isomorphically to $\mathbb{Z}/2\mathbb{Z} \times \mathbb{Z}/2\mathbb{Z}$: $\Gamma \twoheadrightarrow \mathrm{Gal}(K/F) \xrightarrow{\sim} \langle \mu \rangle \times \langle \omega \rangle$. We write $\mathrm{Gal}(K/F) = \langle \sigma \rangle \times \langle \tau \rangle$ and assume $\sigma \mapsto \mu$ and $\tau \mapsto \omega$. After passing to $K$, the twisted Galois action thus obtained is $\sigma_{z}(x,1) = (\sigma(x),1)$, $\sigma_{z}(x,-1) = (-\sigma(x),-1)$, and $\tau_{z} (x,\pm1) = (\tau(x)^{-1},\pm1)$. Then we have 
\begin{align*}
\tilde{T}_{z}(K^{\sigma}) &= \{(x,1)|x \in {K^{\sigma}}^{\times}\} \cup \{(x,-1)|x \in K^{\times}, \sigma(x) = -x\},\\
\tilde{T}_{z}(K^{\tau}) &= \{(x,\pm1)|x\in K^{\times}, \mathrm{Nm}_{K/K^{\tau}}(x) = 1\},\\
\tilde{T}_{z}(F) &= \{(x,1)|x \in {K^{\sigma}}^{\times}, \mathrm{Nm}_{K^{\sigma}/F}(x) = 1\}\\ 
    &\cup \{(x,-1)|x \in K^{\times}, \sigma(x) = -x,\mathrm{Nm}_{K/K^{\tau}}(x) = 1 \}.
\end{align*}
\end{itemize}

\subsubsection{Forms of $\mathbb{G}_{m} \rtimes \mathbb{Z}/2\mathbb{Z}$}

Let $\tilde{T} = \mathbb{G}_{m} \rtimes \mathbb{Z}/2\mathbb{Z}$, where the semi-direct product is given by the inverting action of $\mathbb{Z}/2\mathbb{Z}$ on $\mathbb{G}_{m}$. An elementary calculation shows that the inner automorphisms and translation automorphisms coincide over $\bar{F}$. For convenience, we consider them as translation automorphisms, and one can check
\begin{align*}
\mathrm{Aut}(\tilde{T}) = Z^{1}(\mathbb{Z}/2\mathbb{Z},\bar{F}^{\times}) \rtimes \mathrm{Aut}_{\mathrm{pin}}(\mathbb{G}_{m} \rtimes \mathbb{Z}/2\mathbb{Z}) \cong \bar{F}^{\times} \rtimes \mathbb{Z}/2\mathbb{Z}, 
\end{align*}
in which $\mathbb{Z}/2\mathbb{Z}$ acts on $\bar{F}^{\times}$ by inverting. For $y \in \bar{F}^{\times}$, we denote by $\mu_{y}$ the automorphism that fixes $\mathbb{G}_{m}\rtimes {1}$ and sends $(x,-1)\mapsto (xy,-1)$. The nontrivial element in $\mathbb{Z}/2\mathbb{Z}$ is again denoted by $\omega$, which sends $(x,\pm1)\mapsto (x^{-1},\pm1)$. 

One can further check that the Galois group $\Gamma$ acts on $\mathrm{Aut}(\tilde{T}) = \bar{F}^{\times}\rtimes \mathbb{Z}/2\mathbb{Z}$ by $\sigma(y, \pm1) = ({\sigma(y)}, \pm1)$. The forms of $\mathbb{G}_{m}\rtimes \mathbb{Z}/2\mathbb{Z}$ are classified by the Galois cohomology
\begin{align*}
H^{1}(\Gamma, \mathrm{Aut}(\tilde{T})) = H^{1}(\Gamma, \bar{F}^{\times} \rtimes \mathbb{Z}/2\mathbb{Z}).
\end{align*}

To compute the Galois cohomology $H^{1}(\Gamma,\bar{F}^{\times} \rtimes \mathbb{Z}/2\mathbb{Z})$, we consider the following long exact sequence
\begin{align*}
\cdots \to \mathbb{Z}/2\mathbb{Z} \to H^{1}(\Gamma, \bar{F}^{\times}) \to H^{1}(\Gamma,\bar{F}^{\times} \rtimes \mathbb{Z}/2\mathbb{Z}) \to H^{1}(\Gamma,\mathbb{Z}/2\mathbb{Z}),
\end{align*}
where the surjectivity of the last map follows from the existence of a splitting due to the semi-direct product. In view of Hilbert 90, we have $H^{1}(\Gamma, \bar{F}^{\times}) = 0 $, and hence we conclude that the fibre over the trivial element in $H^{1}(F,\mathbb{Z}/2\mathbb{Z})$ is a singleton, which corresponds to nothing but the split form. It remains to understand the fibre $\mathcal{F}_{E}$ over any nontrivial element $[E] \in H^{1}(\Gamma, \mathbb{Z}/2\mathbb{Z})$ (which corresponds to a quadratic extension $E/F$). To this end, we consider the cocycle $z^{[E]}: \Gamma \twoheadrightarrow \mathrm{Gal}(E/F) = \{1,\sigma\} \hookrightarrow \bar{F}^{\times} \rtimes \mathbb{Z}/2\mathbb{Z}$ sending $\sigma$ to $\omega$. Clearly, we have $z^{[E]} \in \mathcal{F}_{E}$.

It remains to investigate whether the fibre $\mathcal{F}_{E}$ contains any other element than $[z^{[E]}]$. We refer the reader to \cite[Chapter~I]{serre1997galois} for the facts about nonabelian Galois cohomology used here. From now on, we fix the quadratic extension $E/F$. We use $z^{[E]}$ (abbreviated as $z$ below) to twist the $\Gamma$-action on $1\to \bar{F}^{\times}\to\bar{F}^{\times} \rtimes \mathbb{Z}/2\mathbb{Z}\to \mathbb{Z}/2\mathbb{Z}\to 1$ and take the long exact sequence:
\begin{align*}
\cdots \to (\bar{F}^{\times}\rtimes \mathbb{Z}/2\mathbb{Z})_{z}^{\Gamma} \to \mathbb{Z}/2\mathbb{Z} \to H^{1}(\Gamma, \bar{F}^{\times}_{z}) \to H^{1}(\Gamma,(\bar{F}^{\times} \rtimes \mathbb{Z}/2\mathbb{Z})_{z}) \to H^{1}(\Gamma,\mathbb{Z}/2\mathbb{Z}).
\end{align*}
We note that there are identifications
\begin{align*}
\mathcal{F}_{E} &\cong \mathrm{ker}\left(H^{1}(\Gamma,(\bar{F}^{\times} \rtimes \mathbb{Z}/2\mathbb{Z})_{z}) \to H^{1}(\Gamma,\mathbb{Z}/2\mathbb{Z})\right)\\
&\cong H^{1}(\Gamma,\bar{F}_{z}^{\times})/ (\mathbb{Z}/2\mathbb{Z}).
\end{align*}
Here, $H^{1}(\Gamma,\bar{F}_{z}^{\times})/ (\mathbb{Z}/2\mathbb{Z})$ 
denotes the orbit space of $H^{1}(\Gamma,\bar{F}_{z}^{\times})$ under the right action of $\mathbb{Z}/2\mathbb{Z}$, defined in \cite[Chapter~I, \S5.5]{serre1997galois}. We lift $\omega = -1 \in \mathbb{Z}/2\mathbb{Z}$ to $(1, -1) \in (\bar{F}^{\times} \rtimes \mathbb{Z}/2\mathbb{Z})_{z}$. Then the right action of $\omega$ on $[\alpha] \in H^{1}(\Gamma,\bar{F}_{z}^{\times})$ is given by: 
\begin{equation}\label{eq:serre-action}
(\alpha\cdot \omega)(\sigma) := (1, -1)^{-1}\alpha(\sigma)\sigma_{z}(1, -1) = \alpha(\sigma)^{-1},
\end{equation}
where $\sigma_{z}$ is the action of $\sigma$ twisted by $z$, which we can show fixes $(1, -1)$. 

By construction, the $\Gamma$-action on $\bar{F}_{z}^{\times}$ coincides with that on $\bar{F}$-points of the norm torus $\mathrm{Res}_{E/F}^{1}\mathbb{G}_{m}$ associated with $E/F$. We have $H^{1}(\Gamma, \bar{F}_{z}^{\times}) \cong H^{1}(E/F, E^{\times}_{z}) \cong F^{\times}/\mathrm{Nm}E^{\times}$. Explicitly, the isomorphism is given by sending $\alpha \in Z^{1}(E/F, E_{z}^{\times})$ to $\alpha(\sigma)$, where $\sigma \in \mathrm{Gal}(E/F)$ is the nontrivial element. In particular, every element in $H^{1}(\Gamma, \bar{F}_{z}^{\times})$ is of order $1$ or $2$, which implies that $\mathbb{Z}/2\mathbb{Z}$ acts on $H^{1}(\Gamma,\bar{F}_{z}^{\times})$ trivially by \eqref{eq:serre-action}.

Therefore, $\mathcal{F}_{E}$ is in bijection with $H^{1}(\Gamma,\bar{F}_{z}^{\times})$. To summarize, each form in the fibre $\mathcal{F}_{E}$ arises in the following way. Given $[\alpha] \in H^{1}(\Gamma, \bar{F}_{z}^{\times})$ represented by a coset $[y] \in F^{\times}/ \mathrm{Nm}E^{\times}$, we write the image of $\alpha$ in $Z^{1}(\Gamma, (\bar{F}^{\times} \rtimes \mathbb{Z}/2\mathbb{Z})_{z})$ as $\alpha'$. Composing $\alpha'$ with $z$, we obtain a $1$-cocycle $\alpha''\in Z^{1}(\Gamma,\bar{F}^{\times} \rtimes \mathbb{Z}/2\mathbb{Z})$ in the original (untwisted) sense: $\alpha''(\sigma) = \alpha'(\sigma)z(\sigma) = (y,\omega)$ (after passing to $\mathrm{Gal}(E/F)$). The $F$-rational points can be easily computed:
\begin{itemize}[label=\textbullet]
\item After passing to $E$, the Galois action twisted by $\alpha''$ is given by $\sigma_{\alpha''}(x, 1) = (\sigma(x)^{-1}, 1)$ and $\sigma_{\alpha''}(x, -1) = (y\sigma(x)^{-1}, -1)$. The $F$-points are given by
\begin{align*}
\tilde{T}_{\alpha''}(F) = \{(x,1)|x\in E^{\times}, \mathrm{Nm}_{E/F}(x) = 1\} \cup \{(x,-1)|x\in E^{\times}, \mathrm{Nm}_{E/F}(x) = y\}.
\end{align*}
We note that, whenever $y \notin \mathrm{Nm}_{E/F}E^{\times}$, there are no rational points on the non-identity component.
\end{itemize}

\subsection{Rational points on inner forms of quasi-split disconnected groups}\label{sec: innerform}
In this work, we focus on inner forms of quasi-split groups and do not treat translation forms in general (that do not fall into the former category). Let $\tilde{G} = G\rtimes A$ be a quasi-split disconnected reductive group. 

Let $\bar{z} \in Z^{1}(F, G/Z(G)^{A})$. We obtain $\tilde{G}_{\bar{z}}$ by twisting the rational structure of $\tilde{G}$ via $\bar{z}$. After this twisting process, there is still a short exact sequence of $\Gamma$-groups:
\begin{align*}
1 \to G_{\bar{z}}(\bar{F}) \to \tilde{G}_{\bar{z}}(\bar{F}) \to A(\bar{F}) \to 1.
\end{align*}
Explicitly, the twisted Galois action on $\tilde{G}_{\bar{z}}(\bar{F})$ is given by 
\begin{align*}
\sigma_{\bar{z}}(g,a) = \bar{z}(\sigma)(\sigma(g),\sigma(a)) \bar{z}(\sigma)^{-1} = \left(\bar{z}(\sigma)\sigma(g)\sigma(a)[\bar{z}(\sigma)^{-1}], \sigma(a)\right),
\end{align*}
for $\sigma \in \Gamma$. In the formula above and in~\eqref{defnofA(F)[barz]} below, $\bar{z}(\sigma)\in (G/Z(G)^{A})(\bar{F})$ is understood through an arbitrary lift to $G(\bar{F})$. As $Z(G)^{A}$ is central and pointwise fixed by $A$, the conjugation by $\bar{z}(\sigma)$ does not depend on this choice.

Taking $\Gamma$-fixed points, we obtain
\begin{align*}
1 \to G_{\bar{z}}(F) \to \tilde{G}_{\bar{z}}(F) \to A(F).
\end{align*}
The last projection is not always surjective. In fact, one can see that $(g,a) \in \tilde{G}_{\bar{z}}(F)$ if and only if $a \in A(F)$ and
\begin{align}\label{defnofA(F)[barz]}
\bar{z}(\sigma)\sigma(g)a(\bar{z}(\sigma)^{-1}) = g  
\end{align}
for all $\sigma \in \Gamma$. We define $A(F)^{[\bar{z}]}$ as the subgroup of $A(F)$ consisting of elements $a$ for which there exists some $g \in G(\bar{F})$ such that~\eqref{defnofA(F)[barz]} is satisfied for all $\sigma \in \Gamma$. In other words, $A(F)^{[\bar{z}]}$ is exactly the image of the projection of $\tilde{G}_{\bar{z}}(F)$ onto $A(F)$. Therefore, there is a short exact sequence: 
\begin{align*}
1 \to G_{\bar{z}}(F) \to \tilde{G}_{\bar{z}}(F) \to A(F)^{[\bar{z}]} \to 1.
\end{align*}

\section{The LLC for disconnected tori}

Let $F$ be a local field of characteristic zero with absolute Galois group $\Gamma$. The goal of this chapter is to state the local Langlands correspondence (LLC) for disconnected tori in terms of pure inner forms.

\subsection{Pure inner forms}\label{section: local_pure_inner}
We start with a quasi-split disconnected torus $\tilde{T} = T\rtimes A$ defined over $F$. To be precise, $T$ is a (not necessarily split) torus over $F$, $A$ is a (not necessarily constant) finite group scheme defined over $F$ acting on $T$, and the action of $A$ on $T$ in the semi-direct product is defined over $F$ as well. 

Let $z \in Z^{1}(F, T)$. Its image under the natural map $Z^{1}(F, T) \to Z^{1}(F, T/T^{A})$ is a $1$-cocycle $\bar{z}$ valued in the group $T/T^{A}$ of inner automorphisms. As in \S\ref{sec: innerform}, twisting the rational structure of $\tilde{T}$ by $\bar{z}$ yields an inner form $\tilde{T}_{\bar{z}}$. Since $\bar{z}$ is the image of $z$, we denote this form simply by $\tilde{T}_{z}$ and call it a pure inner form.

The discussion in \S\ref{sec: innerform} implies that $(t,a)\in \tilde{T}_{z}(F)$ if and only if $a \in A(F)$ and 
\begin{align}
z(\sigma)\sigma(t)a(z(\sigma)^{-1}) = t \label{defnofA(F)^[z]},
\end{align}
for all $\sigma \in \Gamma$. We write the group $A(F)^{[\bar{z}]}$  defined in \S\ref{sec: innerform} as $A(F)^{[z]}$. Then there is a short exact sequence:
\begin{align*}
1 \to T(F) \to \tilde{T}_{z}(F) \to A(F)^{[z]} \to 1.
\end{align*}
Moreover, we can rewrite~\eqref{defnofA(F)^[z]} as $t^{-1}z(\sigma)\sigma(t) = a(z(\sigma))$. Thus we see that, for a given $a\in A(F)$, there exists some $(t,a)\in \tilde{T}_{z}(F)$ lying above $a$ if and only if the $1$-cocycle $a\cdot z: \sigma \mapsto a(z(\sigma))$ is cohomologous to $z$. Therefore, $A(F)^{[z]}$ is exactly the stabilizer of the cohomology class $[z] \in H^{1}(F, T)$ in $A(F)$.

\begin{rmk}\label{rmk: rigid_inner_form}
Although we are restricted to pure inner forms in this paper, our results can be adapted to treat all inner forms. On the local side, Kaletha~\cite{kaletha2022local} already established the generalized Tate--Nakayama duality for the hypercohomology of Galois gerbes and applied it to establish the LLC for all inner forms of disconnected tori. Our reinterpretation of the LLC carries over to all inner forms verbatim. On the global side, the global Galois gerbes in~\cite{kaletha2018global} are needed to extend our global results to general inner forms of disconnected tori.
\end{rmk}

\subsection{The local Langlands correspondence}
We fix a pure inner form $\tilde{T}_{z}$ once and for all, and set out to precisely state the local Langlands correspondence for $\tilde{T}_{z}$.

\subsubsection{Dual side}
We recall the notions we need on the dual side. The $L$-group of the connected torus $T$ is defined to be the semi-direct product ${}^{L}T = \widehat{T}\rtimes W_{F}$. An $L$-parameter $\phi: W_{F}\to {}^{L}T$ is a continuous morphism such that the projection of $\phi(w)$ to $W_{F}$ is $w$. Two $L$-parameters are said to be equivalent if they are conjugate under $\widehat{T}$. Moreover, we identify the set of $L$-parameters with the set of continuous $1$-cocycles $Z^{1}(W_{F}, \widehat{T})$, and the set of equivalence classes of $L$-parameters with the continuous cohomology group $H^{1}(W_{F},\widehat{T})$. Accordingly, we allow ourselves the following slight abuse of notation:
\begin{itemize}[label=\textbullet]
\item $\phi$ may denote either an $L$‑parameter or its associated $1$‑cocycle.
\item $[\phi]$ may denote, depending on the context,
\begin{enumerate}[label   = (\roman*),   
        widest  = (iii),       
        align   = left,        
        leftmargin = *        
      ]
\item the equivalence class of the $L$‑parameter $\phi$;
\item the cohomology class of $\phi$ in $H^{1}\bigl(W_F,\widehat{T}\bigr)$; or
\item the character of $T(F)$ determined by $\phi$ via the local Langlands correspondence.
\end{enumerate}
\end{itemize}

As we proceed to studying the disconnected torus $\tilde{T}_{z}$, we keep using the same set of $L$-parameters, i.e. those for the connected torus $T$. Moreover, the notations $\phi$ and $[\phi]$ will always be understood as explained in the last paragraph.

The action of $A$ on $T$ gives rise to a natural action of $A(\bar{F})$ on $X_{*}(T)$. Furthermore, in view of the identification $X_{*}(T) = X^{*}(\widehat{T})$, $A(\bar{F})$ acts on $\widehat{T} = \mathrm{Hom}(X^{*}(\widehat{T}),\mathbb{C}^{\times}) = \mathrm{Hom}(X_{*}(T),\mathbb{C}^{\times})$ by $(a\cdot t)(x) = t(a^{-1}\cdot x)$ for $t\in \widehat{T}$ and $x\in X_{*}(T)$. This action of $A(\bar{F})$ on $\widehat{T}$ is Galois-equivariant: $\sigma(a\cdot t) = \sigma(a)\cdot\sigma(t)$ for all $\sigma\in\Gamma$, $a\in A(\bar{F})$ and $t\in\widehat{T}$, where $\Gamma$ acts on $\widehat{T}$ via its action on $X_{*}(T)$. Using this action, we may consider the semi-direct product $\widehat{T}\rtimes A(\bar{F})$. In fact, it turns out to be more convenient to work with a smaller group $\widehat{T}\rtimes A(F)^{[z]}$. Finally, we take the Weil group $W_{F}$ into consideration and introduce the ``enlarged $L$-group''   $(\widehat{T}\rtimes A(F)^{[z]})\rtimes W_{F}$. Composing with the natural embedding ${}^{L}T \hookrightarrow (\widehat{T}\rtimes A(F)^{[z]})\rtimes W_{F}$, any $L$-parameter $\phi$ can be seen as taking values in $(\widehat{T}\rtimes A(F)^{[z]})\rtimes W_{F}$.

A weaker equivalence relation between the $L$-parameters is needed in the disconnected setting:

\begin{defn}\label{l-parameterequiv}
Two $L$-parameters $\phi_{1}, \phi_{2}: W_{F} \to  {}^{L}{T}$ are said to be $A(F)^{[z]}$-equivalent, if there exists some $(t,a)\in \widehat{T}\rtimes A(F)^{[z]}$ such that $(t,a)\phi_{1}(w)(t,a)^{-1} = \phi_{2}(w)$ for all $w\in W_{F}$. 
\end{defn}

We see that there is a one-to-one correspondence between the set of $A(F)^{[z]}$-equivalence classes of $L$-parameters for $T$ and the orbit space $H^{1}(W_{F},\widehat{T})/A(F)^{[z]}$.

From now on, we fix an $L$-parameter $\phi$. It is well-known that the centralizer of $\phi$ in $\widehat{T}$ is nothing but the  Galois-fixed point of $\widehat{T}$, which we denote by $S_{\phi}$:
\begin{align*}
    S_{\phi} := \mathrm{Cent}(\phi, \widehat{T}) = \widehat{T}^{\,W_{F}} = \widehat{T}^{\,\Gamma}.
\end{align*}
Denote the centralizer of $\phi$ in $\widehat{T}\rtimes A(F)^{[z]}$ by 
$$\tilde{S}_{\phi}^{[z]} := \mathrm{Cent}(\phi, \widehat{T}\rtimes A(F)^{[z]}).$$
There is a short exact sequence:
\begin{align}
    1 \to S_{\phi} \to \tilde{S}_{\phi}^{[z]} \to A(F)^{[\phi],[z]} \to 1, \label{SES: dual1}
\end{align}
where $A(F)^{[\phi],[z]}$ is defined to be the subgroup of $A(F)$ comprising elements that fix both $[\phi]$ and $[z]$. Indeed, any $(s,a) \in \tilde{S}^{[z]}_{\phi} \subset \widehat{T}\rtimes A(F)^{[z]}$ should satisfy $(s,a)\phi(w) = \phi(w)(s,a)$ for any $w \in W_{F}$. If  we write $\phi(w) = (\phi_{0}(w),w)$, then we have $s\cdot a(\phi_{0}(w)) = \phi_{0}(w)w(s)$ for any $w$, which implies that $a\cdot \phi_{0}$ is cohomologous to $\phi_{0}$ and hence $\tilde{S}^{[z]}_{\phi}$ maps onto $A(F)^{[\phi],[z]}$ with kernel $\widehat{T}^{W_{F}}$. Taking component groups in~\eqref{SES: dual1}, we further obtain
\begin{align}
      1 \to \pi_{0}(S_{\phi}) \to \pi_{0}(\tilde{S}^{[z]}_{\phi}) \to A(F)^{[\phi],[z]} \to 1. \label{SES: dual2}
\end{align}

Recall that $[z] \in H^{1}(F,T)$ can be viewed as a character of $ \pi_{0}(S_{\phi}) = \pi_{0}(\widehat{T}^{\,\Gamma})$ due to Kottwitz's interpretation of the classical Tate--Nakayama pairing~\cite{kottwitz1984stable} (see also \S\ref{TNdualityforhyper}).  In~\cite[\S4.6]{kaletha2022local}, Kaletha pointed out the following observation, which follows from standard Clifford theory~\cite[Ch.~6]{isaacs1994character}, since $\pi_{0}(\tilde{S}_{\phi}^{[z]})$ is precisely the stabilizer of the character $[z]$ under the conjugation action.
\begin{claim}
\label{claimdual1}
Suppose $\rho$ is an irreducible representation of $\pi_{0}(\tilde{S}_{\phi}^{[z]})$. Then the restriction of $\rho$ to $\pi_{0}(S_{\phi})$ either does not contain $[z]$, or is $[z]$-isotypic.
\end{claim}

We define
\begin{align*}
\mathrm{Irr}(\pi_{0}(\tilde{S}_{\phi}^{[z]}),[z]):=  \{ \rho \in \mathrm{Irr}(\pi_{0}(\tilde{S}_{\phi}^{[z]}))\bigm | \mathrm{Res}_{\pi_{0}(S_{\phi})}^{\pi_{0}(\tilde{S}_{\phi}^{[z]})}\rho \hookleftarrow [z]\}.
\end{align*}

Rather than $\tilde{S}_{\phi}^{[z]}$, one could also consider the centralizer of $\phi$ in $\widehat{T}\rtimes A(\bar{F})$: 
$$\tilde{S}_{\phi}:= \mathrm{Cent}(\phi, \widehat{T}\rtimes A(\bar{F})).$$ 
For this larger centralizer, there are short exact sequences similar to~\eqref{SES: dual1},~\eqref{SES: dual2}, and one can define $\mathrm{Irr}(\pi_{0}(\tilde{S}_{\phi}),[z])$ analogously. In fact, $\pi_{0}(\tilde{S}_{\phi}^{[z]})$ is the stabilizer of $[z]$ when $\pi_{0}(\tilde{S}_{\phi})$ acts on $\mathrm{Irr}(\pi_{0}(S_{\phi}))$ by conjugation. Indeed, this conjugation action factors through $A(F)^{[\phi]}$, and an element lying over $a\in A(F)^{[\phi]}$ fixes $[z]$ precisely when $a\in A(F)^{[\phi],[z]}$, that is, when it lies in $\pi_{0}(\tilde{S}_{\phi}^{[z]})$. As Kaletha pointed out in~\cite[\S4.6]{kaletha2022local}, there is no essential difference in considering $\mathrm{Irr}(\pi_{0}(\tilde{S}_{\phi}),[z])$ or $\mathrm{Irr}(\pi_{0}(\tilde{S}^{[z]}_{\phi}),[z])$, due to the canonical bijection between them which follows from Clifford theory~\cite[Ch.~6]{isaacs1994character}:
\begin{claim}\label{claimdual2}
Induction gives a bijection from $\mathrm{Irr}(\pi_{0}(\tilde{S}_{\phi}^{[z]}),[z])$ to $\mathrm{Irr}(\pi_{0}(\tilde{S}_{\phi}),[z])$.
\end{claim}

\subsubsection{Group side}
We continue to work in the former setting, keeping a pure inner form $\tilde{T}_{z}$ and an $L$-parameter $\phi: W_{F}\to {}^{L}T$ fixed.
\begin{defn}\label{L-packet}
We define the $L$-packet for $\tilde{T}_{z}$ associated with $\phi$ to be 
\begin{align*}
\Pi_{\phi}(\tilde{T}_{z}) := \mathrm{Irr}(\tilde{T}_{z}(F),[\phi]):=  \{ \eta \in \mathrm{Irr}(\tilde{T}_{z})\bigm | \mathrm{Res}_{T(F)}^{\tilde{T}_{z}(F)} \eta \hookleftarrow [\phi]\}.
\end{align*}
\end{defn}

It is clear that $\Pi_{\phi}(\tilde{T}_{z})$ is a finite set. Moreover, we note that two $L$-packets $\Pi_{\phi_{1}}(\tilde{T}_{z})$ and $\Pi_{\phi_{2}}(\tilde{T}_{z})$ intersect non-trivially if and only if $\phi_1$ and $\phi_2$ are $A(F)^{[z]}$-equivalent (see Definition~\ref{l-parameterequiv}), in which case we actually have $\Pi_{\phi_{1}}(\tilde{T}_{z}) = \Pi_{\phi_{2}}(\tilde{T}_{z})$.

We let $\tilde{T}_{z}(F)^{[\phi]}$ be the subgroup of $\tilde{T}_{z}(F)$ fixing $[\phi]: T(F)\to \mathbb{C}^{\times}$ when acting on $T(F)$ by conjugation. Then $\tilde{T}_{z}(F)^{[\phi]}$ sits in the following short exact sequence
\begin{align}
    1 \to T(F) \to \tilde{T}_{z}(F)^{[\phi]} \to A(F)^{[\phi],[z]} \to 1. \label{SESgppullback}
\end{align}
The quotient $A(F)^{[\phi],[z]}$ here is the same as the one on the dual side in~\eqref{SES: dual1}, since the LLC $H^{1}(W_{F},\widehat{T})\cong\mathrm{Hom}_{\mathrm{cts}}(T(F),\mathbb{C}^{\times})$ is $A(F)$-equivariant. We can also define $\mathrm{Irr}(\tilde{T}_{z}(F)^{[\phi]},[\phi])$ in a self-evident way. Since any smooth irreducible representation of $\tilde{T}_{z}(F)$ or $\tilde{T}_{z}(F)^{[\phi]}$ must be finite-dimensional, analogues of Claims~\ref{claimdual1} and~\ref{claimdual2} hold for the group side:
\begin{claim}
Suppose $\eta$ is an irreducible representation of $\tilde{T}_{z}(F)^{[\phi]}$. Then the restriction of $\eta$ to $T(F)$ either does not contain $[\phi]$, or is $[\phi]$-isotypic. \label{claim:groupside1}
\end{claim}

\begin{claim}
Induction gives a bijection from $\mathrm{Irr}(\tilde{T}_{z}(F)^{[\phi]}, [\phi])$ to $\mathrm{Irr}(\tilde{T}_{z}(F), [\phi])$.\label{claim:groupside2}
\end{claim}

\subsubsection{The local Langlands correspondence for disconnected tori}

The LLC for disconnected tori is established by Kaletha:
\begin{thm}[{\cite[\S8]{kaletha2022local}}]
There is a natural bijection 
\begin{align*}
\Pi_{\phi}(\tilde{T}_{z}) \longleftrightarrow \mathrm{Irr}(\pi_{0}(\tilde{S}_{\phi}^{[z]}), [z])
\end{align*}
such that the endoscopic character identities (\cite[Conjecture~4.7]{kaletha2022local}) hold.
\end{thm}

Although the endoscopic character identities pin down the bijection, we do not introduce them here, since they play no role in our arguments. Instead, we take Kaletha's results as given and recover the bijection via the intrinsic construction in Theorem~\ref{thm:llc_reinterthm}.

In view of the observations following Definition~\ref{L-packet}, the local Langlands correspondence for $\tilde{T}_{z}$ actually gives a bijection between the disjoint unions:
\begin{align*}
\mathrm{Irr}(\tilde{T}_{z}) = \coprod_{[\phi] \in H^{1}(W_{F},\widehat{T})/A(F)^{[z]}} \Pi_{\phi}(\tilde{T}_{z})
\longleftrightarrow  \coprod_{[\phi] \in H^{1}(W_{F},\widehat{T})/A(F)^{[z]}}  \mathrm{Irr}(\pi_{0}(\tilde{S}_{\phi}^{[z]}), [z]).
\end{align*}

\section{Construction of the LLC for disconnected tori}

Rather than replicating Kaletha's construction of the LLC for disconnected tori~\cite{kaletha2022local}, we develop a more conceptual approach here. The key idea is to introduce a notion of \textit{relatedness} that ties a representation on the group side to one on the dual side through the Tate–Nakayama pairing in hypercohomology.

We retain the notations from the previous chapter. Throughout this chapter, we fix an $L$-parameter $\phi: W_{F} \to {}^{L}T$ and a 1-cocycle $z \in Z^{1}(F, T)$. First, we transfer elements of $\tilde{T}_{z}(F)^{[\phi]}$ and $\tilde{S}_{\phi}^{[z]}$ to 1-hypercocycles of the Galois group and the Weil group, thereby bringing them within the scope of the Tate--Nakayama machinery.

\subsection{Relation to hypercohomology: group side}\label{section: comparisongp}
As a subgroup of $T(\bar{F})\rtimes A(\bar{F})$, $\tilde{T}_{z}(F)^{[\phi]}$ consists of elements $(t,a)$ such that
\begin{itemize}[label = \textbullet]
    \item $a \in A(F)^{[\phi],[z]}$, and
    \item $z(\sigma)\sigma(t)a(z(\sigma)^{-1}) = t$ for any $\sigma \in \Gamma$. 
\end{itemize}
Given $a \in A(F)^{[\phi],[z]}$, we consider the complex of length two:
\begin{align*}
T(\bar{F}) \xrightarrow{1-a} T(\bar{F})
\end{align*}
which sends $t$ to $ta(t)^{-1}$. 

\begin{obs}\label{obs:relation-groupside} For a fixed $a \in A(F)^{[\phi],[z]}$, we have
$$(t,a) \in \tilde{T}_{z}(F)^{[\phi]}$$
if and only if
$$(z^{-1},t) \in Z^{1}(F, T\xrightarrow{1-a}T).$$ 
\end{obs}

In fact, the correspondence between rational points and hypercohomology can be sharpened, culminating in an isomorphism as $T(F)/(1-a)T(F)$-torsors between the relevant fibres over $a$ and $[z]^{-1}$. What follows in this section is not used elsewhere in the paper, but we present it for its intrinsic interest.

We resume our discussion with a fixed $a \in A(F)^{[\phi],[z]}$. We recall from~\eqref{SESgppullback} that we have the following short exact sequence
\begin{align*}
    1 \to T(F) \to \tilde{T}_{z}(F)^{[\phi]} \to A(F)^{[\phi],[z]} \to 1.
\end{align*}
In other words, the projection of $\tilde{T}_{z}(F)^{[\phi]}$ onto $A(F)^{[\phi],[z]}$ is trivial on $T(F)$, and in particular on the subgroup $(1-a)T(F) := \{ta(t)^{-1}: t\in T(F)\}$. The projection induces the natural map from the set of right $(1-a)T(F)$-cosets to $A(F)^{[\phi],[z]}$, which we denote by $q_{z}$:
\begin{align*}
q_{z}: (1-a)T(F)\backslash \tilde{T}_{z}(F)^{[\phi]} \to A(F)^{[\phi],[z]}.
\end{align*}
Next, we endow the fiber $q_{z}^{-1}(a)$ with the structure of a $T(F)/(1-a)T(F)$-torsor. Consider the natural left action of $[x] \in T(F)/(1-a)T(F)$ on $q_{z}^{-1}(a)$ by group multiplication:
\begin{align*}
    [x] \star [(t, a)] = [(x, 1)(t, a)] = [(xt, a)].
\end{align*}

On the other hand, the hypercohomology group $H^{1}(F, T\xrightarrow{1-a}T)$ lies in the following long exact sequence:
\begin{align*}
    \cdots \to T(F) \xrightarrow{1-a} T(F) \to H^{1}(F, T\xrightarrow{1-a}T) \to  H^{1}(F,T) \xrightarrow{1-a} H^{1}(F,T) \to \cdots.
\end{align*}
After truncations, the above sequence can be modified into
\begin{align*}
    1 \to \frac{T(F)}{(1-a)T(F)} \to H^{1}(F, T\xrightarrow{1-a}T) \xrightarrow{p_{a}} H^{1}(F,T)[1-a]  \to 1,
\end{align*}
where we denote the kernel of the map $1-a: [z]\mapsto [z]a([z])^{-1}$ by $H^{1}(F,T)[1-a]$ and denote the projection onto it by $p_{a}$. Since $T(F)/(1-a)T(F)$ can be seen as a subgroup of the abelian group $H^{1}(F, T\xrightarrow{1-a}T)$, we can naturally endow $p_{a}^{-1}([z]^{-1})$ with the structure of a $T(F)/(1-a)T(F)$-torsor. Let $[x] \in T(F)/(1-a)T(F)$, then
\begin{align*}
[x]\star [(z^{-1}, t)] = [(z^{-1}, xt)].
\end{align*}

Now, we define
\begin{align*}
    \Phi: q_{z}^{-1}(a) &\to p_{a}^{-1}([z]^{-1})\\
    [(t,a)] &\mapsto [(z^{-1},t)]. 
\end{align*}
Indeed, $\Phi$ is well-defined and bijective because of Observation~\ref{obs:relation-groupside} and the fact that,
for any $t \in T(F)$, $(1, t) \in B^1(F, T\xrightarrow{1-a}T)$ if and only if $t \in (1-a)T(F)$. It is clear that $\Phi$ respects the action of $T(F)/(1-a)T(F)$. We obtain the following result:
\begin{prop}
The natural map $\Phi$ is an isomorphism of $T(F)/{(1-a)T(F)}$-torsors.
\end{prop}

\subsection{Relation to hypercohomology: dual side}\label{section: comparisondual}
As in the previous section, we fix $z \in Z^{1}(F,T)$, $a \in A(F)^{[\phi],[z]}$ and an $L$-parameter $\phi$. The following observation is parallel to Observation~\ref{obs:relation-groupside}.
 
\begin{obs}\label{obs:relation-dualside} For a fixed $a \in A(F)^{[\phi],[z]}$, we have 
$$(s,a) \in \tilde{S}_{\phi}^{[z]}$$
if and only if
$$(\phi^{-1}, s) \in Z^{1}(W_{F},\widehat{T}\xrightarrow{1-a}\widehat{T}).$$
\end{obs}

Furthermore, as on the group side, we may update the comparison to an isomorphism between certain fibres as $\pi_{0}(\widehat{T}^{\,\Gamma})/(1-a)\pi_{0}(\widehat{T}^{\,\Gamma})$-torsors. The projection of $\pi_{0}(\tilde{S}_{\phi}^{[z]})$ onto $A(F)^{[\phi],[z]}$ gives rise to the map
\begin{align*}
q'_{\phi}: (1-a)\pi_{0}(\widehat{T}^{\,\Gamma})\backslash \pi_{0}(\tilde{S}_{\phi}^{[z]}) \to A(F)^{[\phi],[z]}.
\end{align*}
On the other hand, there is a short exact sequence obtained by truncating the long exact sequence involving the hypercohomology group $H^{1}(W_{F}, \widehat{T}\xrightarrow{1-a}\widehat{T})_{\mathrm{red}}$:
\begin{align*}
    1 \to \frac{\pi_{0}(\widehat{T}^{\,\Gamma})}{(1-a)\pi_{0}(\widehat{T}^{\,\Gamma})} \to H^{1}(W_{F}, \widehat{T}\xrightarrow{1-a}\widehat{T})_{\mathrm{red}} \xrightarrow{p'_{a}} H^{1}(W_{F},\widehat{T})[1-a] \to 1.
\end{align*}
As on the group side, we let $\pi_{0}(\widehat{T}^{\,\Gamma})/(1-a)\pi_{0}(\widehat{T}^{\,\Gamma})$ act on ${q'_{\phi}}^{-1}(a)$ and ${p'_{a}}^{-1}([\phi]^{-1})$ via left multiplication within $\pi_{0}(\tilde{S}_{\phi}^{[z]})$ and $H^{1}(W_{F}, \widehat{T}\xrightarrow{1-a}\widehat{T})_{\mathrm{red}}$, respectively.

Finally, we define
\begin{align*}
    \Psi: {q'_{\phi}}^{-1}(a) &\to {p'_{a}}^{-1}([\phi]^{-1}) \\
    [(s,a)] &\mapsto [(\phi^{-1},s)] 
\end{align*}
and obtain:
\begin{prop}
The natural map $\Psi$ is an isomorphism of $\pi_{0}(\widehat{T}^{\,\Gamma})/(1-a)\pi_{0}(\widehat{T}^{\,\Gamma})$-torsors. 
\end{prop}

\subsection{The intrinsic construction}
In this section, we present our intrinsic construction of the LLC for disconnected tori. By virtue of the canonical bijection in Claim~\ref{claim:groupside2} between $\Pi_{\phi}(\tilde{T}_{z})$ and $\mathrm{Irr}(\tilde{T}_{z}(F)^{[\phi]},[\phi])$, it suffices to give a one-to-one correspondence:
\begin{align*}
\mathrm{Irr}(\tilde{T}_{z}(F)^{[\phi]},[\phi]) \longleftrightarrow \mathrm{Irr}(\pi_{0}(\tilde{S}_{\phi}^{[z]}),[z]).
\end{align*}

\begin{defn}\label{defn: relatedness}
The representations $\eta\in \mathrm{Irr}(\tilde{T}_{z}(F)^{[\phi]},[\phi])$ and $\rho \in \mathrm{Irr}(\pi_{0}(\tilde{S}_{\phi}^{[z]}),[z])$ are said to be \textit{related} if the following conditions hold:
\begin{itemize}[label=\textbullet]
    \item $\mathrm{dim}\,\eta = \mathrm{dim}\,\rho$, so that we can write $(\eta, V)$ and $(\rho,V)$ for some common underlying space $V$, and moreover
    \item After possibly replacing $\eta$ by $g\,\eta(\cdot)\,g^{-1}$ and/or $\rho$ by $g'\rho(\cdot)\,g'^{-1}$ for some $g, g'\in GL(V)$ (which leaves their isomorphism classes unchanged), the following identity holds in $GL(V)$ for every $a\in A(F)^{[\phi],[z]}$ and any choice of $(s,a) \in \pi_{0}(\tilde{S}_{\phi}^{[z]})$ and $(t,a^{-1}) \in \tilde{T}_{z}(F)^{[\phi]}$:
\begin{align}
\rho(s,a)\,\eta(t,a^{-1}) = \langle (z^{-1},t), (\phi^{-1},s)\rangle_{\mathrm{TN}}^{-1}, \label{reinterrelation}
\end{align}
where the right-hand side is understood as a scalar in $GL(V)$, and $\langle\, \cdot\,  ,\,\cdot\,\rangle_{\mathrm{TN}}$ is the Tate--Nakayama pairing between the hypercohomology $H^{1}(F, T\xrightarrow{1-a^{-1}}T)$ and $H^{1}(W_{F}, \widehat{T}\xrightarrow{1-a}\widehat{T})_{\mathrm{red}}$. 
\end{itemize}
\end{defn}
\begin{rmk}\label{rmk:relatedness}
Regarding the replacement in the second condition above, since the right-hand side of the relation~\eqref{reinterrelation} is a scalar, the need to replace both $\eta$ and $\rho$ can always be reduced to merely replacing any one of them. Moreover, we note that if the relation~\eqref{reinterrelation} holds for some choice of $(s,a) \in \pi_{0}(\tilde{S}_{\phi}^{[z]})$ and $(t,a^{-1}) \in \tilde{T}_{z}(F)^{[\phi]}$, then it holds for every such choice, by Proposition~\ref{TN-LLCcompatibility}.
\end{rmk}

\begin{lem} \label{lem:llcconstr1}
For each $\eta\in\mathrm{Irr}(\tilde{T}_{z}(F)^{[\phi]},[\phi])$, there is at most one $\rho \in \mathrm{Irr}(\pi_{0}(\tilde{S}_{\phi}^{[z]}),[z])$ related to it. Conversely, for each $\rho \in \mathrm{Irr}(\pi_{0}(\tilde{S}_{\phi}^{[z]}),[z])$, there is at most one $\eta\in\mathrm{Irr}(\tilde{T}_{z}(F)^{[\phi]},[\phi])$ related to it.
\end{lem}

\begin{proof}
Fix a homomorphism $\eta: \tilde{T}_{z}(F)^{[\phi]} \to GL(V)$ in $\mathrm{Irr}(\tilde{T}_{z}(F)^{[\phi]},[\phi])$, the relation~\eqref{reinterrelation} singles out a unique map $\rho_{\eta}: \pi_{0}(\tilde{S}_{\phi}^{[z]}) \to GL(V)$ (a priori, $\rho_{\eta}$ is merely set-theoretic). To be precise, the relation~\eqref{reinterrelation} forces that $\rho_{\eta}(s,a) = \langle (z^{-1},t), (\phi^{-1},s)\rangle_{\mathrm{TN}}^{-1}\;\eta(t,a^{-1})^{-1}$, which does not depend on the choice of $(t, a^{-1}) \in \tilde{T}_{z}(F)^{[\phi]}$ by Proposition~\ref{TN-LLCcompatibility}. By the second condition in Definition \ref{defn: relatedness}, all $\rho$ that are related to $\eta$ are equivalent with $\rho_{\eta}$. 

The converse can be shown in a similar manner.
\end{proof}
    
The one-to-one correspondence in Theorem~\ref{thm:llc_reinterthm} below constitutes our intrinsic construction of the LLC for disconnected tori.

\begin{thm}\label{thm:llc_reinterthm}
Relatedness defined in Definition~\ref{defn: relatedness} provides a 1--1 correspondence between $\mathrm{Irr}(\tilde{T}_{z}(F)^{[\phi]},[\phi])$ and $\mathrm{Irr}(\pi_{0}(\tilde{S}_{\phi}^{[z]}),[z])$.
\end{thm} 

\begin{proof}
As shown in the proof of Lemma~\ref{lem:llcconstr1}, given $\eta\in\mathrm{Irr}(\tilde{T}_{z}(F)^{[\phi]},[\phi])$, there is a unique map $\rho_{\eta}: \pi_{0}(\tilde{S}_{\phi}^{[z]}) \to GL(V)$ such that they satisfy the relation~\eqref{reinterrelation}. Similarly, there is a unique map $\eta_{\rho}: \tilde{T}_{z}(F)^{[\phi]} \to GL(V)$ attached to each $\rho \in \mathrm{Irr}(\pi_{0}(\tilde{S}_{\phi}^{[z]}),[z])$ such that the relation~\eqref{reinterrelation} is satisfied. It remains to show that $\rho_{\eta}$ and $\eta_{\rho}$ do lie in $\mathrm{Irr}(\pi_{0}(\tilde{S}_{\phi}^{[z]}),[z])$ and $\mathrm{Irr}(\tilde{T}_{z}(F)^{[\phi]},[\phi])$, respectively. 

We can rephrase and decompose the above goal into three subgoals as follows. Let $\eta:\tilde{T}_{z}(F)^{[\phi]} \to GL(V)$ and
$\rho: \pi_{0}(\tilde{S}_{\phi}^{[z]})\to GL(V)$ be maps satisfying relation~\eqref{reinterrelation} for all $(s,a) \in
\pi_{0}(\tilde{S}_{\phi}^{[z]})$ and $(t,a^{-1}) \in \tilde{T}_{z}(F)^{[\phi]}$. It then suffices to show that (1)~$\eta$ is a
homomorphism if and only if $\rho$ is, (2)~$\eta$ is irreducible if and only if $\rho$ is, and (3)~$\eta$ restricts to $[\phi]$
on $T(F)$ if and only if $\rho$ restricts to $[z]$ on $\pi_{0}(S_{\phi})$.

We start by showing the first point. The proof we present here is modeled after the proof of~\cite[Proposition~8.1]{kaletha2022local}. However, in our setting, we have larger flexibility so that the calculation therein can be significantly simplified.

For arbitrarily fixed $(s_{1},a_{1}),(s_{2},a_{2}) \in \pi_{0}(\tilde{S}_{\phi}^{[z]})$ and $(t_{1},a_{1}^{-1}),(t_{2},a_{2}^{-1}) \in \tilde{T}_{z}(F)^{[\phi]}$, we have  $(s_{1}a_{1}(s_{2}),a_{1}a_{2}) \in \pi_{0}(\tilde{S}_{\phi}^{[z]})$ and $(t_{2}a_{2}^{-1}(t_{1}), a_{2}^{-1}a_{1}^{-1}) \in \tilde{T}_{z}(F)^{[\phi]}$. Applying the relation~\eqref{reinterrelation} to these elements, we obtain
\begin{align*}
\rho(s_{1},a_{1})\eta(t_{1},a_{1}^{-1}) &= \langle (z^{-1},t_{1}), (\phi^{-1},s_{1})\rangle^{-1}, \\
\rho(s_{2},a_{2})\eta(t_{2},a_{2}^{-1}) &= \langle (z^{-1},t_{2}), (\phi^{-1},s_{2})\rangle^{-1},\\
\rho(s_{1}a_{1}(s_{2}),a_{1}a_{2})\eta(t_{2}a_{2}^{-1}(t_{1}),a_{2}^{-1}a_{1}^{-1}) &= \langle (z^{-1},t_{2}a_{2}^{-1}(t_{1})), (\phi^{-1},s_{1}a_{1}(s_{2}))\rangle^{-1}.
\end{align*}
After regrouping the terms, we have
\begin{align*}
&\rho(s_{1},a_{1})\rho(s_{2},a_{2})\rho(s_{1}a_{1}(s_{2}),a_{1}a_{2})^{-1}\\
=&\left(\eta(t_{2}a_{2}^{-1}(t_{1}),a_{2}^{-1}a_{1}^{-1})^{-1}\eta(t_{2},a_{2}^{-1})\eta(t_{1},a_{1}^{-1})\right)^{-1}\\
\cdot& \langle (z^{-1},t_{1}), (\phi^{-1},s_{1})\rangle^{-1}\langle (z^{-1},t_{2}), (\phi^{-1},s_{2})\rangle^{-1}\langle (z^{-1},t_{2}a_{2}^{-1}(t_{1})), (\phi^{-1},s_{1}a_{1}(s_{2}))\rangle.
\end{align*}
Therefore, in order to show that $\eta$ is a homomorphism if and only if $\rho$ is a homomorphism, it suffices to show
\begin{alignat}{6}
\langle &(z^{-1},t_{1}),&&(\phi^{-1},s_{1})\rangle\cdot \langle &&(z^{-1},t_{2}),&&(\phi^{-1},s_{2})\rangle = 
\langle &&(z^{-1},t_{2}a_{2}^{-1}(t_{1})), &&(\phi^{-1},s_{1}a_{1}(s_{2}))\rangle \nonumber\\
&\xrightarrow{1-a_{1}^{-1}}
&&\xrightarrow{1-a_{1}} &&\xrightarrow{1-a_{2}^{-1}} &&\xrightarrow{1-a_{2}} &&\xrightarrow{1-(a_{1}a_{2})^{-1}} &&\xrightarrow{1-a_{1}a_{2}} \label{formula: three pairings}
\end{alignat}
for all $(s_{1},a_{1}),(s_{2},a_{2}) \in \pi_{0}(\tilde{S}_{\phi}^{[z]})$ and $(t_{1},a_{1}^{-1}),(t_{2},a_{2}^{-1}) \in \tilde{T}_{z}(F)^{[\phi]}$.  The arrows on the second row indicate the hypercohomology complexes to which the elements correspond. 

Let $K$ be a finite Galois extension of $F$ over which $T$ splits, $z$ is inflated from $\mathrm{Gal}(K/F)$, and $t_1,t_2\in T(K)$. In view of the isomorphism $\mathcal{H}$~\eqref{isom: H}, there exists $[(\lambda_{i}, \mu_{i})] \in H_{0}(W_{K/F}, X\xrightarrow{1-a_{i}^{-1}} X)_{0}$ corresponding to $[(z^{-1}, t_{i})]\in H^{1}(K/F, T\xrightarrow{1-a_{i}^{-1}}T)$, for $i = 1,2$. In view of Diagram~\eqref{diagram: isom_H}, since both $[\lambda_{1}]$ and $[\lambda_{2}]$ correspond to $[z^{-1}]$ under the isomorphism $TN$, we may assume $\lambda_{1} = \lambda_{2} = \lambda$ without loss of generality. 

As explained on p.129 of~\cite{kottwitz1999foundations}, the isomorphism $\mathcal{H}$~\eqref{isom: H} is induced by the chain map $\psi\oplus \phi$ constructed there\footnote{The symbol $\phi$ in~\cite{kottwitz1999foundations} denotes a chain map $C_{1}(W_{K/F}, X) \to C^{0}(K/F, T)$, which has nothing to do with the fixed $L$-parameter $\phi$ in the present chapter.}. One can check that the map $\phi$ is $A(F)^{[\phi],[z]}$-equivariant. The pair $(\lambda, \mu_{2}+a_{2}^{-1}(\mu_{1}))$ is a $0$-hypercycle for the complex $X\xrightarrow{1-(a_{1}a_{2})^{-1}} X$ (as follows from the hypercycle property of $(\lambda,\mu_{1})$ and $(\lambda,\mu_{2})$, the action of $A(F)$ on $X$ commuting with that of $W_{K/F}$), and the equivariance of $\phi$ yields, on the level of representatives,
\begin{align*}
\mathcal{H}(\lambda, \mu_{2}+a_{2}^{-1}(\mu_{1})) = \big(\psi(\lambda),\, \phi(\mu_{2})\cdot a_{2}^{-1}(\phi(\mu_{1}))\big).
\end{align*}
Note that, although the hypercocycle $\mathcal{H}(\lambda,\mu_{i})$ need not equal $(z^{-1},t_{i})$, they must be cohomologous, say
\begin{align*}
\mathcal{H}(\lambda,\mu_{i}) = (z^{-1},t_{i})\cdot\big(\partial c_{i},\, (1-a_{i}^{-1})(c_{i})\big) \quad \text{for some } c_{i}\in T(K), \quad i = 1,2.
\end{align*}
Comparing the first components gives $\partial c_{1} = \partial c_{2}$. Therefore, we have
\begin{align*}
u := c_{1}c_{2}^{-1} \in T(F), \qquad v := a_{2}^{-1}\big((1-a_{1}^{-1})(u)\big) = a_{2}^{-1}(u)\cdot\big((a_{1}a_{2})^{-1}(u)\big)^{-1} \in T(F).
\end{align*}
A direct computation shows
\begin{align*}
\mathcal{H}(\lambda, \mu_{2}+a_{2}^{-1}(\mu_{1})) = \big(z^{-1},\, t_{2}a_{2}^{-1}(t_{1})\cdot v\big)\cdot \big(\partial c_{2},\, (1-(a_{1}a_{2})^{-1})(c_{2})\big),
\end{align*}
where the rightmost factor is a hypercoboundary for the complex $T\xrightarrow{1-(a_{1}a_{2})^{-1}}T$. In other words, the isomorphism $\mathcal{H}$ sends
\begin{align*}
[(\lambda, \mu_{2}+a_{2}^{-1}(\mu_{1}))]\in H_{0}(W_{K/F}, X\xrightarrow{1-(a_{1}a_{2})^{-1}} X)_{0}
\end{align*}
to 
\begin{align*}
[(z^{-1}, t_{2}a_{2}^{-1}(t_{1})\cdot v)] \in H^{1}(K/F, T\xrightarrow{1-(a_{1}a_{2})^{-1}}T).
\end{align*}  

It turns out that the discrepancy term $v$ does not contribute to the pairings under consideration. Indeed, since $v \in T(F)$, $(1,v)$ is a hypercocycle, whose class is $i(v)$ in the notation of Proposition~\ref{TN-LLCcompatibility}. Then we may split
\begin{align*}
[(z^{-1},\, t_{2}a_{2}^{-1}(t_{1})\cdot v)] = [(z^{-1},\, t_{2}a_{2}^{-1}(t_{1}))]\cdot i(v).
\end{align*}
By Proposition~\ref{TN-LLCcompatibility}, the contribution of $i(v)$ is
\begin{align*}
\langle i(v), (\phi^{-1}, s_{1}a_{1}(s_{2}))\rangle = [\phi]^{-1}(v) = [\phi]^{-1}\big(a_{2}^{-1}(u)\big)\cdot[\phi]\big((a_{1}a_{2})^{-1}(u)\big) = 1.
\end{align*}
The last equality uses the fact that both $a_{2}$ and $a_{1}a_{2}$ fix the character $[\phi]$ of $T(F)$. Therefore, we can simply drop the discrepancy term $v$ in our calculation below.

By the definition of Tate--Nakayama pairing~\eqref{formula:def_of_TNhyperpairing}, the pairings in~\eqref{formula: three pairings} can be expanded into
\begin{alignat}{2}
\langle(z^{-1},t_{1}),(\phi^{-1},s_{1})\rangle &= \langle \lambda, s_{1}\rangle &&\cdot \prod_{w}\langle\mu_{1}(w), \phi(w)\rangle,   \\
\langle (z^{-1},t_{2}),(\phi^{-1},s_{2})\rangle &=  \langle \lambda, s_{2}\rangle &&\cdot \prod_{w}\langle\mu_{2}(w), \phi(w)\rangle,\\
\langle(z^{-1},t_{2}a_{2}^{-1}(t_{1})),(\phi^{-1},s_{1}a_{1}(s_{2}))\rangle &= \langle \lambda, s_{1}a_{1}(s_{2})\rangle &&\cdot \prod_{w}\langle[\mu_{2}+a_{2}^{-1}(\mu_{1})](w), \phi(w)\rangle,
\end{alignat}
where the $w$'s run over $W_{K/F}$. Multiplying the first two formulae together and then cancelling with the third reduces the task to showing
\begin{align*}
\langle \lambda,(1-a_{1})s_{2} \rangle \cdot \prod_{w}\langle(1-a_{2}^{-1})\mu_{1}(w),\phi(w)\rangle = 1.
\end{align*}
The important observation is that the left-hand side equals to the pairing between 
\begin{align*}
[(z^{-1},(1-a_{2}^{-1})t_{1})] \in H^{1}(K/F, T\xrightarrow{(1-a_{2}^{-1})(1-a_{1}^{-1})}T)
\end{align*}
and 
\begin{align*}
[(\phi^{-1}, (1-a_{1})s_{2})]\in H^{1}(W_{K/F}, \widehat{T}\xrightarrow{(1-a_{1})(1-a_{2})}\widehat{T}).
\end{align*} 
To see this, note that $(\lambda, (1-a_{2}^{-1})(\mu_{1}))$ is a $0$-hypercycle for the complex $X\xrightarrow{(1-a_{2}^{-1})(1-a_{1}^{-1})}X$ whose image $\mathcal{H}(\lambda, (1-a_{2}^{-1})(\mu_{1}))$ differs from $(z^{-1},(1-a_{2}^{-1})(t_{1}))$ exactly by $\big(\partial c_{1},\, (1-a_{2}^{-1})(1-a_{1}^{-1})(c_{1})\big)$, the hypercoboundary of $c_{1}$ for this complex (no discrepancy term arises this time), so expanding by~\eqref{formula:def_of_TNhyperpairing} as before yields precisely the left-hand side. It is now clear that Theorem~\ref{thm:appoffunctor} applies. Indeed, if we let $T=U=V$, $f = 1-a_{1}^{-1}$ and $g = 1-a_{2}^{-1}$ in Theorem~\ref{thm:appoffunctor}, it follows that the pairing between them must vanish. So far, we have proved the claim that $\eta$ is a homomorphism if and only if $\rho$ is a homomorphism. 

Next, we note that $\rho$ and $\eta$ have the same image in $PGL(V)$ due to the relation~\eqref{reinterrelation}. Therefore, $\eta$ is irreducible if and only if $\rho$ is irreducible.

Finally, we must show that $\eta$ restricts to $[\phi]$ on $T(F)$ if and only if $\rho$ restricts to $[z]$ on $\pi_{0}(S_{\phi})$. For this purpose, we simply let $a = 1$. Then the hypercohomology groups involved in~\eqref{reinterrelation} are $H^{1}(W_{F}, \widehat{T}\xrightarrow{0}\widehat{T})_{\mathrm{red}}$ and  $H^{1}(F, T\xrightarrow{0}T)$, which are actually direct products $H^{1}(W_{F},\widehat{T})\times\pi_{0}(\widehat{T}^{\,\Gamma})$ and $H^{1}(F,T)\times T(F)$, respectively. In this degenerate case, the pairing between them is merely the product of the Tate--Nakayama pairing and the Langlands pairing:
\begin{align*}
\langle (z^{-1}, t), (\phi^{-1}, s) \rangle = [z]^{-1}(s)[\phi]^{-1}(t).
\end{align*}
The defining relation~\eqref{reinterrelation} becomes  
\begin{align*}
\rho(s,1)\eta(t,1) = [z](s)[\phi](t)
\end{align*}
for any $s \in \pi_{0}(S_{\phi})$ and any $t\in T(F)$. It follows that $\rho(s,1) = [z](s)$ if and only if $\eta(t,1) = [\phi](t)$.
\end{proof}

The following proposition will be used later in the global context. Suppose $z \in Z^{1}(F,T)$ is cohomologous to the trivial element, then there exists some $d \in T(\bar{F})$ such that $z(\sigma) = d^{-1}\sigma(d)$ for any $\sigma \in \Gamma$. There is an isomorphism of topological groups $\gamma_{d}$ (depending on the choice of $d$):
\begin{align*}
\gamma_{d}:\tilde{T}(F)^{[\phi]} = T(F)\rtimes A(F)^{[\phi]} &\xrightarrow{\sim} \tilde{T}_{z}(F)^{[\phi]}\\
(t,a) &\mapsto (td^{-1}a(d), a).
\end{align*}

\begin{prop}\label{prop:genericlem}
Suppose $[z]$ is trivial. Let $d$ be an element in $T(\bar{F})$ such that $z(\sigma) = d^{-1}\sigma(d)$, and let $\gamma_{d}$ be the isomorphism defined as above. Under the LLC, the trivial character $\mathbb{1} \in \mathrm{Irr}(\pi_{0}(\tilde{S}^{[z]}_{\phi}), 1)$  corresponds to $[\phi]^{*} \in \mathrm{Irr}(\tilde{T}_{z}(F)^{[\phi]}, [\phi])$, the trivial extension of $[\phi]$ across $\gamma_{d}$. To be precise, we have
\begin{align*}
[\phi]^{*} (\gamma_{d} (t, a)) = [\phi](t),
\end{align*}
for all $(t, a) \in T(F)\rtimes A(F)^{[\phi]}$.
\end{prop} 
\begin{proof} 
One can immediately check that $(z^{-1}, td^{-1}a(d))$ and $(1, t)$ are cohomologous, as elements in  $Z^{1}(F, T\xrightarrow{1-a}T)$. Therefore, according to the relation~\eqref{reinterrelation}, we have
\begin{alignat}{2}
[\phi]^{*} (\gamma_{d} (t, a)) = \langle &(z^{-1}, td^{-1}a(d)) ,&& (\phi^{-1}, s)  \rangle^{-1} = \langle (1, t), (\phi^{-1},s) \rangle^{-1} =  [\phi](t),\nonumber\\
&\xrightarrow{1-a} &&\xrightarrow{1-a^{-1}}\nonumber
\end{alignat}
for any $(t,a)\in T(F)\rtimes A(F)^{[\phi]}$. The last equality above uses Proposition \ref{TN-LLCcompatibility}.
\end{proof}

\subsection{Comparison with Kaletha's construction}\label{section:comparison_with_kal}
Suppose $\eta\in\mathrm{Irr}(\tilde{T}_{z}(F)^{[\phi]}, [\phi])$ and $\rho_{\eta} \in \mathrm{Irr}(\pi_{0}(\tilde{S}_{\phi}^{[z]}), [z])$ are related. In other words, they correspond to each other under our intrinsic construction of the LLC from the previous section, see Theorem~\ref{thm:llc_reinterthm}. On the other hand, we denote by $\rho_{\eta}^{\mathrm{Kal}} \in \mathrm{Irr}(\pi_{0}(\tilde{S}_{\phi}^{[z]}), [z])$ the element corresponding to $\eta$ under Kaletha's construction of the LLC in~\cite{kaletha2022local}. This section aims to show the two constructions coincide:   $\rho_{\eta} = \rho_{\eta}^{\mathrm{Kal}}$.

\subsubsection{Construction of $\rho_{\eta}^{\mathrm{Kal}}$}
We start by recalling Kaletha's construction. Kaletha constructs an isomorphism between certain push-outs of  $\tilde{T}_{z}(F)^{[\phi]}$ and  $\pi_{0}(\tilde{S}_{\phi}^{[z]})$, which induces a bijection between $\mathrm{Irr}(\tilde{T}_{z}(F)^{[\phi]}, [\phi])$ and $\mathrm{Irr}(\pi_{0}(\tilde{S}_{\phi}^{[z]}),[z])$. On the group side, we take the push-out of the top extension along $[\phi]: T(F) \to \mathbb{C}^\times$:
\begin{equation} \label{diagram:pushout_1}
\begin{tikzcd}
1 \arrow[r] & T(F) \arrow[d, "{[\phi]}"] \arrow[r] & {\tilde{T}_{z}(F)^{[\phi]}} \arrow[r] \arrow[d] & {A(F)^{[\phi],[z]}} \arrow[r] \ar[d,-,double equal sign distance,double] & 1 \\
1 \arrow[r] & \mathbb{C}^{\times} \arrow[r]        & {\mathcal{E}^{z}_{[\phi]}} \arrow[r]            & {A(F)^{[\phi],[z]}} \arrow[r]           & 1.
\end{tikzcd}
\end{equation}
Let $\mathrm{Irr}(\mathcal{E}^{z}_{[\phi]}, \mathrm{id})$ be the set of irreducible representations of $\mathcal{E}^{z}_{[\phi]}$ whose restriction to $\mathbb{C}^{\times}$ is the identity map $x \mapsto x$. It is straightforward to see there is a canonical bijection 
\begin{align*}
\mathrm{Irr}(\mathcal{E}^{z}_{[\phi]}, \mathrm{id}) \xrightarrow{\sim} \mathrm{Irr}(\tilde{T}_{z}(F)^{[\phi]},[\phi]).
\end{align*}

On the dual side, we have a similar push-out diagram:
\begin{equation}\label{diagram:pushout_2}
\begin{tikzcd}
1 \arrow[r] & \pi_{0}(\widehat{T}^{\,\Gamma}) \arrow[d, "{[z]}"] \arrow[r] & \pi_{0}(\tilde{S}_{\phi}^{[z]}) \arrow[r] \arrow[d] & {A(F)^{[\phi],[z]}} \arrow[r] \ar[d,-,double equal sign distance,double] & 1 \\
1 \arrow[r] & \mathbb{C}^{\times} \arrow[r]        & {\mathcal{E}^{\phi}_{[z]}} \arrow[r]            & {A(F)^{[\phi],[z]}} \arrow[r]           & 1.
\end{tikzcd}
\end{equation}
Let $\mathrm{Irr}(\mathcal{E}^{\phi}_{[z]},\mathrm{id})$ be the set of irreducible representations of $\mathcal{E}^{\phi}_{[z]}$ whose restriction to $\mathbb{C}^{\times}$ is the identity map. Then there is a canonical bijection 
\begin{align*}
\mathrm{Irr}(\mathcal{E}^{\phi}_{[z]},\mathrm{id}) \xrightarrow{\sim} \mathrm{Irr}(\pi_{0}(\tilde{S}_{\phi}^{[z]}),[z]).
\end{align*}

With a slight abuse of notation, we will again write $\eta$ and $\rho$ for their respective preimages under the above bijections.

Now, in order to construct a bijection $\mathrm{Irr}(\tilde{T}_{z}(F)^{[\phi]},[\phi]) 	\leftrightarrow \mathrm{Irr}(\pi_{0}(\tilde{S}_{\phi}^{[z]}),[z])$, it suffices to establish an isomorphism between the extensions in the second row of Diagram~\eqref{diagram:pushout_1} and Diagram~\eqref{diagram:pushout_2}. 

As usual, we regard $\tilde{T}_{z}(F)^{[\phi]}$ as a subgroup of $T(\bar{F}) \rtimes A(F)^{[\phi],[z]}$ and $\pi_{0}(\tilde{S}_{\phi}^{[z]})$ a subquotient of $\widehat{T}\rtimes A(F)^{[\phi],[z]}$. For each $a \neq \mathbb{1} \in A(F)^{[\phi],[z]}$, we fix, once and for all, some $(t_{a}, a) \in \tilde{T}_{z}(F)^{[\phi]}$ and some $(s_{a}, a) \in \pi_{0}(\tilde{S}_{\phi}^{[z]})$. When $a = \mathbb{1}$, we let $t_{\mathbb{1}} = \mathbb{1}$ and $s_{\mathbb{1}} = \mathbb{1}$. Now the extensions 
\begin{center}
    \begin{tikzcd}
        1 \arrow[r] & T(F)  \arrow[r] & {\tilde{T}_{z}(F)^{[\phi]}} \arrow[r]  & {A(F)^{[\phi],[z]}} \arrow[r]  & 1
    \end{tikzcd}
\end{center}
and
\begin{center}
    \begin{tikzcd}
        1 \arrow[r] & \pi_{0}(\widehat{T}^{\,\Gamma})  \arrow[r] & \pi_{0}(\tilde{S}_{\phi}^{[z]}) \arrow[r]  & {A(F)^{[\phi],[z]}} \arrow[r]  & 1
    \end{tikzcd}
\end{center}
can be characterized by 2-cocycles (factor sets determined by the sections $(t_{a}, a)$ and $(s_{a}, a)$):
\begin{align*}
    \alpha(a,b) := t_{a}a(t_{b})t_{ab}^{-1}
\end{align*}
and
\begin{align*}
    \beta(a,b) := s_{a}a(s_{b})s_{ab}^{-1}.
\end{align*}
Let $\bar{\alpha} = [\phi]\circ \alpha$ and $\Bar{\beta} = [z]\circ \beta$. Then the extensions $\mathcal{E}_{[\phi]}^{z}$ and $\mathcal{E}_{[z]}^{\phi}$ can be realised as twisted products via $\bar{\alpha}$ and $\bar{\beta}$, respectively:
\begin{align*}
    \mathcal{E}_{[\phi]}^{z} &\cong \mathbb{C}^{\times} \boxtimes_{\bar{\alpha}} A(F)^{[\phi],[z]}
\end{align*}
and
\begin{align*}
\mathcal{E}_{[z]}^{\phi} &\cong \mathbb{C}^{\times} \boxtimes_{\bar{\beta}} A(F)^{[\phi],[z]}.
\end{align*}
These identifications are not canonical since they depend on the auxiliary choices of sections $(t_{a}, a)$ and $(s_{a}, a)$.  However, as established by Kaletha and recalled below, this choice of identification does not affect the final correspondence between representations.

Kaletha constructs an explicit map between these twisted products:
\begin{align*}
\mathcal{I}:\mathbb{C}^{\times} \boxtimes_{\bar{\beta}} A(F)^{[\phi],[z]} &\to \mathbb{C}^{\times} \boxtimes_{\bar{\alpha}} A(F)^{[\phi],[z]}\\
x \boxtimes a &\mapsto xh(a)^{-1} \boxtimes a,
\end{align*}
where $h(a)$ is defined by means of the Tate--Nakayama pairing:
\begin{align*}
    h(a):=\bar{\alpha}\left(a^{-1}, a\right) \cdot\left\langle\left(z^{-1}, t_{a^{-1}}\right),\left(\phi^{-1}, s_{a}\right)\right\rangle.
\end{align*}
Kaletha shows that $\mathcal{I}$ induces an isomorphism of extensions. More precisely, it gives an isomorphism between the short exact sequences on the second rows of Diagrams~\eqref{diagram:pushout_2} and~\eqref{diagram:pushout_1} and it does not depend on the choices of the sections $(t_{a}, a)$ and $(s_{a}, a)$ (see~\cite[Proposition~8.1]{kaletha2022local}). Since $\mathcal{I}$ restricts to the identity on the kernel $\mathbb{C}^{\times}$, the induced correspondence between their representations is canonical.

Gathering the bijections obtained so far, for each $\eta \in \mathrm{Irr}(\tilde{T}_{z}(F)^{[\phi]}, [\phi])$, one can associate to it an element $\rho_{\eta}^{\mathrm{Kal}} \in {\mathrm{Irr}(\pi_{0}(\tilde{S}_{\phi}^{[z]}),[z])} $ via the following composition:
\begin{center}
\begin{tikzcd}
{\mathrm{Irr}(\tilde{T}_{z}(F)^{[\phi]},[\phi])} \arrow[r, "\sim"] & {\mathrm{Irr}(\mathcal{E}_{[\phi]}^{z},\mathrm{id})} \arrow[r, "\mathcal{I}^{*}"] & {\mathrm{Irr}(\mathcal{E}_{[z]}^{\phi},\mathrm{id})} \arrow[r, "\sim"] & {\mathrm{Irr}(\pi_{0}(\tilde{S}_{\phi}^{[z]}),[z])} \\
\eta \arrow[rrr, dotted, maps to] & &  & \rho_{\eta}^{\mathrm{Kal}}.     
\end{tikzcd}
\end{center}

\subsubsection{Comparison}
Recall that we have fixed sections $(t_{a}, a)$ and $(s_{a}, a)$ for each $a \in A(F)^{[\phi],[z]}$. 
\begin{prop}
$\rho_{\eta}$ coincides with $\rho_{\eta}^{\mathrm{Kal}}$.
\end{prop}
\begin{proof}
Let $s = (s_{0}, a) \in \pi_{0}(\tilde{S}_{\phi}^{[z]})$ be an element lying above $a$. We rewrite it as 
\begin{align*}
(s_{0}, a) = (s_{0}s_{a}^{-1}, 1)(s_{a}, a).
\end{align*}
Thus, $(s_{0}, a)$ will be pushed to $[z](s_{0}s_{a}^{-1})\boxtimes a$ in $\mathcal{E}_{[z]}^{\phi} \cong \mathbb{C}^{\times} \boxtimes_{\bar{\beta}} A(F)^{[\phi],[z]}$ and further brought to $[z](s_{0}s_{a}^{-1})h(a)^{-1}\boxtimes a$ under $\mathcal{I}$. As a result, we have
\begin{align}\label{equation: rho_eta_Kal}
\rho_{\eta}^{\mathrm{Kal}}(s) &= \eta\left([z](s_{0}s_{a}^{-1})h(a)^{-1}\boxtimes a\right) \nonumber\\
&= [z](s_{0}s_{a}^{-1})h(a)^{-1} \eta(t_{a}, a) \nonumber \\
&= [z](s_{0}s_{a}^{-1})\bar{\alpha}\left(a^{-1}, a\right)^{-1} \cdot\left\langle\left(z^{-1}, t_{a^{-1}}\right),\left(\phi^{-1}, s_{a}\right)\right\rangle^{-1}\cdot \eta(t_{a}, a).
\end{align}
On the other hand, our intrinsic construction, characterized by relation~\eqref{reinterrelation}, yields
\begin{align}\label{equation: rho_eta}
\rho_{\eta}(s) &= \left \langle (z^{-1},t_{a^{-1}}), (\phi^{-1}, s_{0}) \right \rangle^{-1} \eta(t_{a^{-1}}, a^{-1})^{-1} \nonumber \\
    &= \left \langle (z^{-1},t_{a^{-1}}), (\phi^{-1}, s_{a}) \right \rangle^{-1} \left \langle (z^{-1},t_{a^{-1}}), (1, s_{0}s_{a}^{-1}) \right \rangle^{-1} \eta(t_{a^{-1}}, a^{-1})^{-1} \nonumber \\
    &= \left \langle (z^{-1},t_{a^{-1}}), (\phi^{-1}, s_{a}) \right \rangle^{-1}\cdot [z](s_{0}s_{a}^{-1}) \cdot \eta(t_{a^{-1}}, a^{-1})^{-1},
\end{align}
where the last equality uses Proposition \ref{TN-LLCcompatibility}. Finally, we recall that $\bar{\alpha} = [\phi]\circ \alpha$ and hence obtain
\begin{align} \label{equation: bar_alpha}
\bar{\alpha}(a^{-1},a) := [\phi](t_{a^{-1}}a^{-1}(t_{a})) = \eta(t_{a^{-1}}, a^{-1})\eta(t_{a}, a).
\end{align}

Combining~\eqref{equation: rho_eta_Kal},~\eqref{equation: rho_eta} and~\eqref{equation: bar_alpha}, we conclude $\rho_{\eta}(s) = \rho_{\eta}^{\mathrm{Kal}}(s)$ for any $s \in \pi_{0}(\tilde{S}_{\phi}^{[z]})$.
\end{proof}

\section{Discrete automorphic representations}

Let $F$ be a number field with absolute Galois group $\Gamma = \mathrm{Gal}(\bar{F}/F)$. Let $\mathcal{O}_{F}$ (or simply $\mathcal{O}$) be the ring of integers in $F$. Let $\tilde{T} = T\rtimes A$ be a quasi-split disconnected torus defined over $F$. We recall this means that $A$ is a (not necessarily constant) finite group scheme acting on the (not necessarily split) torus $T$, and the action is defined over $F$.

\subsection{Global pure inner forms}
Once and for all, we fix some $z \in Z^{1}(F,T)$. As in \S\ref{section: local_pure_inner}, we twist the rational structure on $\tilde{T} = T\rtimes A$ by the image of $z$ in $Z^{1}(F, T/T^{A})$ and obtain the pure inner form $\tilde{T}_{z}$.

\subsubsection{Rational points}
As in the local case, we have a short exact sequence 
\begin{align*}
1 \to T(F) \to \tilde{T}_{z}(F) \to A(F)^{[z]} \to 1,
\end{align*}
where $A(F)^{[z]}$ is the subgroup of $A(F)$ stabilising the cohomology class $[z]$. 

In forthcoming sections, we will consider the action of $\tilde{T}_{z}(F)$ on the set of Hecke characters of $T$ by conjugation: $t \cdot \chi(x) = \chi(t^{-1}xt),\forall x\in T(\mathbb{A})$. Due to commutativity of $T(\mathbb{A})$,  this actually descends to an action of $A(F)^{[z]}$ on the set of Hecke characters. Let $\chi: T(F)\backslash T(\mathbb{A}) \to \mathbb{C}^{\times}$ be a Hecke character. We denote the stabilizers of $\chi$ in $\tilde{T}_{z}(F)$ and $A(F)^{[z]}$ by $\tilde{T}_{z}(F)^{\chi}$ and $A(F)^{[z],\chi}$, respectively. They sit in the following short exact sequence:
\begin{align}
1 \to T(F) \to \tilde{T}_{z}(F)^{\chi} \to A(F)^{[z],\chi} \to 1.\label{SESglobalchi}
\end{align}

\subsubsection{Adelic points} \label{sec: adelic_points}
We first lay the necessary background for the discussion of adelic points. Let $K$ be a finite Galois extension of $F$ such that $z$ factors through $\mathrm{Gal}(K/F)$ and $\tilde{T}_{z}$ splits over $K$. We still denote by the same symbol the image of $z$ under the natural map 
\begin{align*}
Z^{1}(\mathrm{Gal}(K/F), T(K)) \to Z^{1}(\mathrm{Gal}(K/F), T(\mathbb{A}_{K})).
\end{align*}
In view of the decomposition into restricted direct product $T(\mathbb{A}_{K}) \cong \sideset{}{'}\prod_{v} \big(\prod_{w|v} T(K_w)\big)$, we may project to the chosen factor to write $z = (z_{v})_{v}$, where $z_{v}$ is the image of $z$ under the natural map 
\begin{align*}
Z^{1}(\mathrm{Gal}(K/F), T(K)) \to Z^{1}(\mathrm{Gal}(K_{v}/F_{v}), T(K_{v})).
\end{align*}
Here, for each place $v$ of $F$ we fix, once and for all, a place of $K$ lying above $v$ and write $K_{v}$ for the corresponding completion, so that $\mathrm{Gal}(K_{v}/F_{v})$ is the associated decomposition group.

We may view $\tilde{T}$ as a $F_{v}$-scheme via a base change and hence consider the pure inner form $\tilde{T}_{z_{v}}$. We quickly observe $\tilde{T}_{z}(F_{v}) = \tilde{T}_{z_{v}}(F_{v})$. From now on, we will normally write $\tilde T_{z}(F_v)$, which we should understand as exactly the set of rational points of the local pure inner form.

Next, we can discuss the adelic points. Taking $\mathrm{Gal}(K/F)$-fixed points of $\tilde{T}_{z}(\mathbb{A}_{K})$, we obtain the set of adelic points:
\begin{align*}
\tilde{T}_{z}(\mathbb{A}_{F}) = \{(t, a)\vert \ t \in T(\mathbb{A}_{K}), a\in A(\mathbb{A}_{F}),\mathrm{and}\ a(z(\sigma)) = z(\sigma)t^{-1}\sigma(t) \text{ for } \forall \sigma \in \mathrm{Gal}(K/F)\},
\end{align*} 
which sits in the following short exact sequence:
\begin{align}
    1 \to T(\mathbb{A}_{F}) \to \tilde{T}_{z}(\mathbb{A}_{F}) \to A(\mathbb{A}_{F})^{[z]} \to 1, \label{adelic SES}
\end{align}
where $A(\mathbb{A}_{F})^{[z]}$ is the stabilizer of $[z]\in H^{1}(\mathrm{Gal}(K/F), T(\mathbb{A}_{K}))$ in $A(\mathbb{A}_{F})$. 

\subsubsection{Integral models} \label{section:model}
Fix an $\mathcal{O}_{F}$-model $\tilde{\mathcal{T}}_{z}$ of $\tilde{T}_{z}$, which further gives rise to an $\mathcal{O}_{F}$-model of $T$, denoted by $\mathcal{T}$. According to~\cite[Appendix~IV, \S1, Lemma]{moeglin_waldspurger_1995}, for almost all places $v$, we have 
$$\tilde{T}_{z}(F_{v}) = T(F_{v})\tilde{\mathcal{T}}_{z}(\mathcal{O}_{v}).$$ 
In other words, for almost all $v$, each connected component of $\tilde{T}_{z}(F_{v})$ contains integral points. We will write $\tilde{T}_{z}(\mathcal{O}_{v})$ in place of $\tilde{\mathcal{T}}_{z}(\mathcal{O}_{v})$ for brevity.

We fix $S$, a finite set of places, which is large enough in the sense of containing all infinite places, every place that ramifies in the extension $K/F$, every place where $z$ does not take values in $\mathcal{T}(\mathcal{O}_{K_v})$, every place where $\mathcal{T}\otimes \mathcal{O}_{K_v}$ is not a split torus scheme, every place $v$ where the equality $\tilde{T}_{z}(F_{v}) = T(F_{v})\tilde{T}_{z}(\mathcal{O}_{v})$ fails, every place $v$ where $\mathcal{T}(\mathcal{O}_{v})$ is \textit{not} the unique maximal compact subgroup of $T(F_{v})$, and every place $v$ where $\tilde{\mathcal{T}}_{z}(\mathcal{O}_{v})$ is not the fixed-point set of the $z_{v}$-twisted Galois action on $\mathcal{T}(\mathcal{O}_{K_{v}})\rtimes A(K_{v})$ (only finitely many such places, by spreading out the identification $\tilde{T}_{z}\otimes K = (T\rtimes A)\otimes K$ to the integral models).

In general, if we let $G$ be an affine algebraic group over $F$ and let $\mathcal{G}$ be an $\mathcal{O}_{F}$-model of $G$, then $G(\mathbb{A}_{F})$ can be identified with the restricted direct product of $G(F_{v})$ defined with restriction in the subgroups $\mathcal{G}(\mathcal{O}_{v})$ for finite places $v$. Therefore, the short exact sequence~\eqref{adelic SES} can be rewritten as 
\begin{align*}
    1 \to \sideset{}{'}\prod_{v} T(F_{v}) \to \sideset{}{'}\prod_{v}\tilde{T}_{z}(F_{v}) \to \prod_{v}A(F_{v})^{[z_{v}]} \to 1.
\end{align*}
Indeed, when writing $A(\mathbb{A}_{F}) = \prod'_{v} A(F_{v})$, one notices that the restricted direct product $\prod'_{v} A(F_{v})$ is actually a genuine direct product, according to~\cite[Appendix~IV, \S1, Lemma]{moeglin_waldspurger_1995}. Moreover, it is straightforward to check that $(a_{v})_{v}\in \prod A(F_v)$ fixes $[z] = ([z_{v}])_{v} \in H^{1}(\mathrm{Gal}(K/F), T(\mathbb{A}_{K}))$ if and only if $a_{v}$ fixes $[z_{v}]$ for each place $v$. Therefore, $A(\mathbb{A}_{F})^{[z]}$ coincides with the direct product $\prod_{v}A(F_{v})^{[z_{v}]}$. We endow $A(\mathbb{A}_{F})^{[z]}$ with the product topology, under which it becomes a compact topological group.

To ensure that~\eqref{adelic SES} is moreover a short exact sequence of topological groups, we must show that the product topology coincides with the natural quotient topology on $A(\mathbb{A}_{F})^{[z]}$. For this purpose, we note that $\tilde{T}_{z}(\mathbb{A}_{F})$ is $\sigma$-compact and $\tilde{T}_{z}(\mathbb{A}_{F}) \to A(\mathbb{A}_{F})^{[z]} $ is continuous with respect to the product topology on the target. Then the coincidence of the product and quotient topologies follows from~\cite[Prop.~2.46]{folland2016course}.

\subsection{The automorphic quotient}
From now on, we drop the subscript and simply write $\mathbb{A}$ for the adele ring of $F$. 

\subsubsection{Compactness of the automorphic quotient}
We embed $\tilde{T}_{z}(F)$ diagonally as a discrete closed subgroup of $\tilde{T}_{z}(\mathbb{A})$.

Let $T_{0}$ be the largest $\mathbb{Q}$-split torus of $\mathrm{Res}_{F/\mathbb{Q}}T$. Define $A_{T} := T_{0}(\mathbb{R})^{\circ}$, the identity component of the $\mathbb{R}$-points, and then embed $A_{T}$ into $T(\mathbb{A})$ and 
$\tilde{T}_{z}(\mathbb{A})$. It is noteworthy that $[T] := A_{T}T(F)\backslash T(\mathbb{A})$ is compact. 

By virtue of~\cite[Appendix~IV, \S1, Lemma]{moeglin_waldspurger_1995}, the projection $\tilde{T}_z(\mathcal{O}_v) \to A(F_v)^{[z_v]}$ is surjective for almost all $v$. This allows us to construct a factorizable section $s = \prod_v s_v: A(\mathbb{A})^{[z]} \to \tilde{T}_{z}(\mathbb{A})$ taking values in $\tilde{T}_z(\mathcal{O}_v)$ almost everywhere, which is automatically continuous. Consider the composition
\begin{align*}
[T] \times A(\mathbb{A})^{[z]} \xrightarrow{(id,s)} [T] \times \tilde{T}_{z}(\mathbb{A}) \twoheadrightarrow A_{T}\tilde{T}_{z}(F)\backslash\tilde{T}_{z}(\mathbb{A}),
\end{align*}
where the second map is induced by multiplication in $\tilde{T}_{z}(\mathbb{A})$. The composition is continuous and surjective. In view of compactness of $[T]\times A(\mathbb{A})^{[z]}$, it follows that $A_{T}\tilde{T}_{z}(F)\backslash\tilde{T}_{z}(\mathbb{A})$ is compact as well. We caution that $A_{T}\tilde{T}_{z}(F)\backslash\tilde{T}_{z}(\mathbb{A})$ is only a coset space rather than a group since $A_{T}\tilde{T}_{z}(F)$ is not a normal subgroup of $\tilde{T}_{z}(\mathbb{A})$ in general.

\subsubsection{Normalization of measures} \label{section: measures}
For the remainder of this paper, we adopt the following measure normalizations:
\begin{itemize}[label=\textbullet]
    \item Discrete groups, such as $T(F)$, $\tilde{T}_{z}(F)$, $\tilde{T}_{z}(F)^{\chi}$, and $A(F)^{[z]}$, are equipped with the counting measure.
    \item We normalize the Haar measure on the compact topological group $A(\mathbb{A})^{[z]}$ such that it has total volume $1$. We fix an arbitrary Haar measure on $A_{T}$ and normalize the Haar measure on $T(\mathbb{A})$ so that the induced measure (see the last bullet point) on the compact quotient $[T] = A_{T}T(F)\backslash T(\mathbb{A})$ has total volume $1$.
    \item Given a short exact sequence of unimodular groups $1 \to N \to G \to Q \to 1$ with specified Haar measures on $N$ and $Q$, we equip $G$ with the unique Haar measure satisfying the quotient integral formula $\int_{G} f(g)\,dg = \int_{Q} \int_{N} f(nq)\,dn\,dq$. For example, the Haar measure on $\tilde{T}_{z}(\mathbb{A})$ is normalized in this manner from the measures on $T(\mathbb{A})$ and $A(\mathbb{A})^{[z]}$. Another example is the Haar measure on $A_{T}\tilde{T}_{z}(F)$, which can be normalized as the product of the fixed Haar measure on $A_T$ and the counting measure on $\tilde{T}_z(F)$ since $A_{T} \cap \tilde{T}_{z}(F) = \{1\}$. 

    \item For any unimodular group $G$ and closed unimodular subgroup $H$ with specified Haar measures, we equip the quotient space $H \backslash G$ with the unique right $G$-invariant Radon measure $d\dot{g}$ satisfying the quotient integral formula $\int_{G} f(g)\,dg = \int_{H \backslash G} \int_{H} f(hg)\,dh\,d\dot{g}$.
\end{itemize}

\begin{prop} \label{prop: invariant-Radon}
There is a unique right $\tilde{T}_{z}(\mathbb{A})$-invariant Radon measure (see~\cite[p.~xii]{folland2016course} for the definition) on the coset space $A_{T}\tilde{T}_{z}(F)\backslash \tilde{T}_{z}(\mathbb{A})$ that satisfies the quotient integral formula under the conventions above.
\end{prop}
\begin{proof}
    It suffices to show that both $\tilde{T}_{z}(\mathbb{A})$ and the closed subgroup $A_{T}\tilde{T}_{z}(F)$ are unimodular.
    
    The group $\tilde{T}_{z}(\mathbb{A})$ is an extension of the compact group $A(\mathbb{A})^{[z]}$ by the abelian group $T(\mathbb{A})$. According to~\cite[Thm.~2.51]{folland2016course}, because the Haar measure on $A(\mathbb{A})^{[z]}$ is a $\tilde{T}_{z}(\mathbb{A})$-invariant Radon measure, we must have $\Delta_{\tilde{T}_{z}(\mathbb{A})}|_{T(\mathbb{A})} = \Delta_{T(\mathbb{A})} \equiv 1$. Therefore, the modular function $\Delta_{\tilde{T}_{z}(\mathbb{A})}$ factors through the compact quotient $A(\mathbb{A})^{[z]}$, hence $\Delta_{\tilde{T}_{z}(\mathbb{A})} \equiv 1$. 
    
    The closed subgroup $A_{T}\tilde{T}_{z}(F)$ requires a distinct argument. We start by noting that conjugation by any element in $\tilde{T}_{z}(F)$ induces an $F$-rational algebraic automorphism of $T$ and further induces a $\mathbb{Q}$-rational algebraic automorphism of $\mathrm{Res}_{F/\mathbb{Q}}T$ by the functoriality of the Weil restriction. It follows that both $T_{0}$ and $A_{T} = T_{0}(\mathbb{R})^{\circ}$ are stabilized under the conjugation. Moreover, we note that $A_{T}$ is open in $A_{T}T(F)$ while $A_{T}T(F)$ is of finite index in $A_{T}\tilde{T}_{z}(F)$, hence it follows that $A_T$ is an open normal subgroup of $A_{T}\tilde{T}_{z}(F)$. 
    
    By~\cite[Ch.~VII, \S1, No.~4, p.~13]{bourbaki2004integration}, we have $\Delta_{A_{T}\tilde{T}_{z}(F)}(g) = \operatorname{mod}(i_{g})$, where $i_{g}$ is the inner automorphism of $A_{T}\tilde{T}_{z}(F)$ induced by $g$ and the modulus $\operatorname{mod}(i_{g})$ is the factor by which $i_{g}$ scales the Haar measure. Since $A_{T}$ is an open normal subgroup of $A_{T}\tilde{T}_{z}(F)$ and the conjugation of $A_{T}$ on itself is trivial, the modulus factor can be fully determined by the conjugation action of $\tilde{T}_{z}(F)$ on $A_{T}$. Now, $g \mapsto \operatorname{mod}(i_{g})$ defines a group homomorphism $\tilde{T}_z(F) \to \mathbb{R}_{>0}$~\cite[Ch.~VII, \S1, No.~4, p.~13, Prop.~4]{bourbaki2004integration}. Because $T(F)$ commutes with $A_T$ (as they both lie in $T(\mathbb{A})$), the homomorphism $g \mapsto \operatorname{mod}(i_{g})$ factors through the finite quotient $\tilde{T}_z(F)/T(F) \cong A(F)^{[z]}$. Thus, the homomorphism must be trivial because the multiplicative group $\mathbb{R}_{>0}$ has no non-trivial finite subgroups. This proves that $A_{T}\tilde{T}_{z}(F)$ is also unimodular.
    
    Now that both groups are unimodular, the compatibility $\Delta_{\tilde{T}_{z}(\mathbb{A})}|_{A_{T}\tilde{T}_{z}(F)} = \Delta_{A_{T}\tilde{T}_{z}(F)}$ holds trivially. The desired statement follows from~\cite[Thm.~2.51]{folland2016course}. 
    \end{proof}

\subsubsection{Discrete automorphic representations}
By the definition of Radon measure, compactness of $A_{T}\tilde{T}_{z}(F)\backslash \tilde{T}_{z}(\mathbb{A})$ implies that it has finite volume.  Moreover, since $A_{T}\tilde{T}_{z}(F)\backslash \tilde{T}_{z}(\mathbb{A})$ is equipped with a right $\tilde{T}_{z}(\mathbb{A})$-invariant Radon measure, according to~\cite[Prop.~2.42 and \S3.1]{folland2016course}, the Hilbert space 
$$L^{2}(A_{T}\tilde{T}_{z}(F)\backslash \tilde{T}_{z}(\mathbb{A}))$$ 
becomes a unitary representation when equipped with the action of $\tilde{T}_{z}(\mathbb{A})$ by right translation. In \S\ref{sec: pre-decomp}, we will show $L^{2}(A_{T}\tilde{T}_{z}(F)\backslash \tilde{T}_{z}(\mathbb{A}))$ decomposes as a Hilbert direct sum of irreducible constituents. We call the irreducible constituents \textit{discrete automorphic representations} (or simply \textit{automorphic representations}). 

Our goal is to identify which representations are automorphic and determine their multiplicities. Chapter~8 approaches this question purely on the automorphic side by decomposing the $L^2$-space directly. Chapter~9 brings in the dual side and presents the answer in the form of a multiplicity formula in the spirit of~\eqref{eq:intro_mult}.

\begin{rmk}
    Slightly more generally, let $\omega: A_{T} \to S^{1}$ be a unitary character and consider the Hilbert space $L^{2}(\tilde{T}_{z}(F)\backslash\tilde{T}_{z}(\mathbb{A}), \omega)$ of $\tilde{T}_{z}(F)$-left-invariant functions $f$ with $f(tx) = \omega(t)f(x)$ for $t\in A_{T}$ and $x\in \tilde{T}_{z}(\mathbb{A})$. One may ask how this space decomposes as well. For $f\in L^{2}(\tilde{T}_{z}(F)\backslash\tilde{T}_{z}(\mathbb{A}), \omega)$, $\gamma \in \tilde{T}_{z}(F)$, and $t \in A_{T}$, we have $f(\gamma t x) = f\big((\gamma t \gamma^{-1})\gamma x\big) = \omega(\gamma t \gamma^{-1})f(\gamma x) = \omega(\gamma t \gamma^{-1})f(x)$, using $\gamma t\gamma^{-1}\in A_{T}$ together with the left $\tilde{T}_{z}(F)$-invariance of $f$. On the other hand, $f(\gamma t x) = f(t x) = \omega(t)f(x)$. Comparing the two, we see $\omega(\gamma t \gamma^{-1}) = \omega(t)$ unless $f \equiv 0$. Hence $L^{2}(\tilde{T}_{z}(F)\backslash\tilde{T}_{z}(\mathbb{A}), \omega)$ vanishes unless $\omega$ is invariant under the conjugation action of $\tilde{T}_{z}(F)$ on $A_{T}$, which factors through the finite quotient $A(F)^{[z]}$.

    Assume $\omega$ is indeed $\tilde{T}_{z}(F)$-invariant. Then all of our results in the following chapters carry over to this twisted case verbatim. Indeed, what we find in the untwisted setting is that only Hecke characters in $H(T)$ (which are trivial on $A_T$) account for the irreducible constituents of $L^{2}(A_{T}\tilde{T}_{z}(F)\backslash \tilde{T}_{z}(\mathbb{A}))$. As for the twisted case, the irreducible constituents of $L^{2}(\tilde{T}_{z}(F)\backslash\tilde{T}_{z}(\mathbb{A}), \omega)$ are parametrized by $H(T, \omega)$, comprised of Hecke characters of $T$ that restrict to $\omega$ on $A_{T}$. Because $\omega$ is $\tilde{T}_z(F)$-invariant, the set $H(T, \omega)$ is stable under the conjugation action of $\tilde{T}_{z}(F)$. This stability is the exact prerequisite needed to perform the orbit decomposition in \S\ref{sec: pre-decomp}, and all of the subsequent arguments remain valid after replacing $H(T)$ by $H(T, \omega)$.
\end{rmk}

\section{Multiplicity on the automorphic side}

\subsection{Preparation: smooth induction}
We write $\mathbb{A} = F_{\infty} \times \mathbb{A}^{\infty}$, where $F_{\infty}$ is the product of $F_{v}$ over all the infinite places $v$ and $\mathbb{A}^{\infty}$ is the restricted direct product of the $F_{v}$ with $v$ ranging in all the finite places. 

Let $G$ be a unimodular locally compact group that factors as $G = G_{\infty} \times G^{\infty}$, where $G_{\infty}$ is a real Lie group and $G^{\infty}$ is a first-countable totally disconnected locally compact group. For instance, when $\mathbb{G}$ is an affine algebraic group defined over $F$, its adelic points admit such a factorization: $\mathbb{G}(\mathbb{A}) = \mathbb{G}(F_{\infty})\times \mathbb{G}(\mathbb{A}^{\infty})$. 

First, we recall the notion of the smooth part of a representation.

\begin{defn}
Let $\pi$ be a continuous representation of $G$ on a complete locally convex topological vector space $V$. A vector $v \in V$ is termed smooth if $v$ is a smooth vector with respect to $G_{\infty}$ (that is, the iterated derivatives $\pi(X_1)\cdots\pi(X_n)v$ exist for all $n \geq 1$ and all elements $X_1, \dots, X_n$ in the Lie algebra of $G_{\infty}$) and is $K^{\infty}$-invariant for some open compact subgroup $K^{\infty} \leq G^{\infty}$. We call $V_{\mathrm{sm}}$, which consists of all smooth vectors in $V$, the smooth part of $V$. We say that $\pi$ is admissible if, for every compact open subgroup $K^{\infty}\leq G^{\infty}$, the space $V^{K^{\infty}}$ of $K^{\infty}$-fixed vectors is finite-dimensional.
\end{defn}

$V_{\mathrm{sm}}$ is stable under the action of $G$ and dense in $V$. When $V$ is a Fr\'{e}chet space, $V_{\mathrm{sm}}$ can be endowed with a natural LF-topology as follows. First, we denote by $V_{\mathrm{sm},\infty}\subseteq V$ the subspace of smooth vectors under $G_{\infty}$, which is a Fr\'{e}chet space. Then, we let $K_{n}$ be a decreasing cofinal sequence of compact open subgroups of $G^{\infty}$, and we rewrite $V_{\mathrm{sm}}$ as an inductive limit (union): $\displaystyle V_{\mathrm{sm}} = \bigcup_{n} (V_{\mathrm{sm},\infty})^{K_{n}}$. Thus, $V_{\mathrm{sm}}$ naturally becomes an LF-space equipped with the inductive topology. (Recall that an LF-space is a strict inductive limit of an increasing sequence of Fr\'{e}chet spaces~\cite[\S13]{treves2016topological}. Here, the limit is strict because the invariant subspaces $(V_{\mathrm{sm},\infty})^{K_{n}}$ are closed in the Fr\'{e}chet space $V_{\mathrm{sm},\infty}$.) 

From now on, whenever we speak of a \textit{smooth representation}, we implicitly assume its underlying topological vector space is an LF-space arising in the above manner. 

Next, we recall the notion of unitary induced representation~\cite[\S6.1]{folland2016course} and define the smooth induction as its smooth part:

\begin{defn} \label{defn:smooth_ind}
Let $H \leq G$ be a closed unimodular subgroup and let $(\rho, W)$ be a unitary representation of $H$. We define the \textbf{unitary induced representation} $L^2\text{-}\mathrm{Ind}_{H}^{G}\rho$ as the Hilbert space completion of the space of continuous functions $f: G \to W$ that are compactly supported modulo $H$ and satisfy $f(hg) = \rho(h)f(g)$, under the inner product $\int_{H \backslash G} \langle f_{1}(g), f_{2}(g) \rangle_W \, dg$. 

We then define the \textbf{smooth induction} of $(\rho, W)$ to $G$, denoted by $\mathrm{Ind}_{H}^{G}\rho$, as the dense subspace of smooth vectors of the Hilbert space $L^2\text{-}\mathrm{Ind}_{H}^{G}\rho$, equipped with its natural LF-topology.
\end{defn}

In our context, we will apply these inductions not only to the full adelic points of an algebraic group, but also to certain intermediate subgroups. We are particularly interested in smooth inductions $\mathrm{Ind}_{H}^{G}\,\rho$ when $G$, $H$, and $\rho$ satisfy the following assumptions:

\begin{assumption} \label{assump:GH}
    \leavevmode
    \begin{itemize}[label=\textbullet]
        \item $G = \prod'_{v} G_{v}$ is a restricted direct product (with respect to the compact open subgroups $G_{v}\cap \tilde{T}_{z}(\mathcal{O}_{v})$ for finite $v$), where $T(F_{v})\subseteq G_{v}\subseteq \tilde{T}_{z}(F_{v})$ for each place $v$.
        \item $H$ satisfies one of the following:
    \begin{itemize}
        \item $H$ is a restricted direct product $\prod'_{v}H_{v}$ with $T(F_{v}) \subseteq H_{v} \subseteq G_{v}$;
        \item $H \subseteq G$ contains $T(\mathbb{A})$ as a subgroup of finite index. For example, $H = T(\mathbb{A})\tilde{T}_{z}(F)$. 
    \end{itemize}
    \item The $H$-representation $(\rho, W)$ is unitary and finite-dimensional.
    \end{itemize}
\end{assumption}

\begin{rmk}
    The arguments in the proof of Proposition~\ref{prop: invariant-Radon} show that any groups $G$ and $H$ satisfying Assumption~\ref{assump:GH} are necessarily unimodular. Thus, for such $G$ and $H$, we can consider the unitary and smooth induced representations in Definition~\ref{defn:smooth_ind}.
\end{rmk}

\begin{rmk}\label{rmk: findim_smooth}
Since $W$ is finite-dimensional, the restriction $\rho|_{T(\mathbb{A})}$ is automatically smooth.
\end{rmk}

For the rest of this section, we assume that $G$, $H$ and $\rho$ satisfy Assumption~\ref{assump:GH} unless specified otherwise. For clarity of exposition, we might as well fix $G = \tilde{T}_{z}(\mathbb{A})$. However, the arguments actually hold verbatim for any $G$ satisfying the assumption above.

Instead of requiring smoothness with respect to the entire $\tilde{T}_{z}(\mathbb{A})$, we may disregard the smoothness with respect to the archimedean places. Denote by $\displaystyle (L^2\text{-}\mathrm{Ind}_{H}^{\tilde{T}_{z}(\mathbb{A})}\rho)^{\infty}_{\mathrm{sm}}$ the subspace of vectors in $\displaystyle L^2\text{-}\mathrm{Ind}_{H}^{\tilde{T}_{z}(\mathbb{A})}\rho$ that are smooth under the action of $\tilde{T}_{z}(\mathbb{A}^{\infty})$.

\begin{prop} \label{prop: automatic_smoothness}
Maintain the notation and assumptions above. A $\tilde{T}_{z}(\mathbb{A}^{\infty})$-smooth vector in $L^2\text{-}\mathrm{Ind}_{H}^{\tilde{T}_{z}(\mathbb{A})}\rho$ is automatically $\tilde{T}_{z}(F_{\infty})$-smooth. Hence we have $\displaystyle \mathrm{Ind}_{H}^{\tilde{T}_{z}(\mathbb{A})}\rho = (L^2\text{-}\mathrm{Ind}_{H}^{\tilde{T}_{z}(\mathbb{A})}\rho)^{\infty}_{\mathrm{sm}}$.
\end{prop}
\begin{proof}
Suppose $v \in L^2\text{-}\mathrm{Ind}_{H}^{\tilde{T}_{z}(\mathbb{A})}\rho$ is fixed by a compact open subgroup $K^{\infty}\leq \tilde{T}_{z}(\mathbb{A}^{\infty})$ under right translation $(R_{x}v)(g) = v(gx)$. We first show that $v$ is smooth under the Lie group $T(F_{\infty})$.

By Assumption~\ref{assump:GH}, $H$ contains $T(\mathbb{A})$, and $T(\mathbb{A})$ is normal in $\tilde{T}_{z}(\mathbb{A})$. Hence, for $x \in T(F_{\infty})$ and $g \in \tilde{T}_{z}(\mathbb{A})$, we have $gxg^{-1} \in T(\mathbb{A}) \subseteq H$, so that right translation by $T(F_{\infty})$ acts pointwise:
\begin{align}\label{equation: pointwise_translation}
(R_{x}v)(g) = v\big((gxg^{-1})g\big) = \rho(gxg^{-1})\,v(g).
\end{align}
Moreover, every double coset $HgK^{\infty}$ is open in $\tilde{T}_{z}(\mathbb{A})$. Indeed, it is stable under right multiplication by the open subgroup $T(F_{\infty})\times K^{\infty}$: for $x\in T(F_{\infty})$ and $k,k'\in K^{\infty}$ we have $hgk\cdot xk' = h\,(gxg^{-1})\,g\,kk' \in HgK^{\infty}$. Since $H\backslash\tilde{T}_{z}(\mathbb{A})$ is compact (as $H$ contains $T(\mathbb{A})$) and is partitioned by the images of the open double cosets, there are only finitely many of them, say $Hg_{1}K^{\infty}, \dots, Hg_{n}K^{\infty}$.

Let $w_{i} := v(g_{i})\in W$. Then we have $v(hg_{i}k) = \rho(h)w_{i}$ for all $h\in H$ and $k\in K^{\infty}$. By \eqref{equation: pointwise_translation}, for $x \in T(F_{\infty})$, we have
\begin{align*}
(R_{x}v)(hg_{i}k) = v\big(h(g_{i}xg_{i}^{-1})g_{i}k\big) = \rho(h)\,\rho(g_{i}xg_{i}^{-1})\,w_{i}.
\end{align*}
We note that $x\mapsto g_{i}xg_{i}^{-1}$ is a smooth automorphism of $T(F_{\infty})$ and $w_{i}$ is a smooth vector of $\rho|_{T(\mathbb{A})}$ by Remark~\ref{rmk: findim_smooth}. Therefore, each of the $W$-valued maps $x \mapsto \rho(g_{i}xg_{i}^{-1})\,w_{i}$ is smooth, which implies that the orbit map $x\mapsto R_{x}v$ is infinitely differentiable. Indeed, $R_{x}v$ is still right $K^{\infty}$-invariant ($x$ commutes with $K^{\infty}$). For right $K^{\infty}$-invariant $v'$, we have $\lVert v'\rVert^{2} = \sum_{i=1}^{n}\mathrm{vol}(H\backslash Hg_{i}K^{\infty})\,\lVert v'(g_{i})\rVert_{W}^{2}$ by unitarity of $\rho$. So convergence of the finitely many values in $W$ implies convergence under the $L^{2}$-norm.

So far, we have shown that $v$ is smooth under $T(F_{\infty})$. Because $T(F_\infty)$ is an open normal subgroup of finite index in $\tilde{T}_z(F_\infty)$, their Lie algebras coincide. Therefore, infinite differentiability carries over.
\end{proof}

The following consequence of Proposition~\ref{prop: automatic_smoothness} allows us to pass between smooth inductions and their $L^{2}$-completions.

\begin{cor} \label{cor: admissibility}
In the setting of Proposition~\ref{prop: automatic_smoothness}, write $V := L^2\text{-}\mathrm{Ind}_{H}^{\tilde{T}_{z}(\mathbb{A})}\rho$, so that $V_{\mathrm{sm}} = \mathrm{Ind}_{H}^{\tilde{T}_{z}(\mathbb{A})}\rho$.
\begin{enumerate}[label=(\arabic*)]
    \item For every compact open subgroup $K^{\infty}\leq \tilde{T}_{z}(\mathbb{A}^{\infty})$, the space $V^{K^{\infty}}$ of $K^{\infty}$-fixed vectors is finite-dimensional and consists of smooth vectors. Moreover, $V_{\mathrm{sm}} = \bigcup_{K^{\infty}}V^{K^{\infty}}$. In particular, $V$ and $V_{\mathrm{sm}}$ are admissible.
    \item If $V = \widehat{\bigoplus_{j\in J}}V_{j}$ is a Hilbert direct sum of closed invariant subspaces, then $V_{\mathrm{sm}} = \bigoplus_{j\in J}(V_{j})_{\mathrm{sm}}$, an algebraic direct sum.
    \item Suppose $V_{\mathrm{sm}} = \bigoplus_{j\in J}\pi_{j}$ with each $\pi_{j}$ an irreducible smooth subrepresentation, and write $\overline{\pi_{j}}$ for the closure of $\pi_{j}$ in $V$. Then each $\overline{\pi_{j}}$ is an irreducible unitary representation whose smooth part is exactly $\pi_{j}$. For $j,j'\in J$, the closures $\overline{\pi_{j}}$ and $\overline{\pi_{j'}}$ are unitarily isomorphic if and only if $\pi_{j}\cong\pi_{j'}$ as smooth representations. Moreover, $V$ is unitarily isomorphic to $\widehat{\bigoplus_{j\in J}}\,\overline{\pi_{j}}$. In particular, $V$ decomposes as a Hilbert direct sum of irreducible unitary subrepresentations, and $V$ and $V_{\mathrm{sm}}$ have the same irreducible constituents with the same (automatically finite) multiplicities.
\end{enumerate}
\end{cor}
\begin{proof}
(1) In the proof of Proposition~\ref{prop: automatic_smoothness}, we saw that an element of $V^{K^{\infty}}$ is determined by its values at the finitely many representatives $g_{1},\dots,g_{n}$, so $\dim V^{K^{\infty}}\leq n\dim W$. By Proposition~\ref{prop: automatic_smoothness} and the definition of smoothness, $V^{K^{\infty}}$ consists of smooth vectors and $V_{\mathrm{sm}} = \bigcup_{K^{\infty}}V^{K^{\infty}}$ holds.

(2) We consider $P_{K^{\infty}}: v\mapsto \int_{K^{\infty}}R_{k}v\,dk$ (Haar measure of total volume $1$), which is the orthogonal projection onto $V^{K^{\infty}}$ and preserves every closed invariant subspace. Hence, $V^{K^{\infty}} = \widehat{\bigoplus_{j}}V_{j}^{K^{\infty}}$, which is finite-dimensional by (1), so only finitely many summands are nonzero and the sum is algebraic. Taking the union over all $K^{\infty}$ proves the claim.

(3) Fix $j$ and write $\pi := \pi_{j}$. Since $\pi^{K^{\infty}}$ is finite-dimensional, hence closed, continuity of $P_{K^{\infty}}$ gives $\overline{\pi}^{K^{\infty}} = P_{K^{\infty}}(\overline{\pi}) \subseteq \overline{P_{K^{\infty}}(\pi)} = \pi^{K^{\infty}}$, so $\overline{\pi}^{K^{\infty}} = \pi^{K^{\infty}}$. According to (1), the smooth part of $\overline{\pi}$ equals $\bigcup_{K^{\infty}}\overline{\pi}^{K^{\infty}} = \pi$. For irreducibility, let $M\subseteq\overline{\pi}$ be a nonzero closed invariant subspace. We note $M^{K^{\infty}}\neq 0$ for some $K^{\infty}$, due to strong continuity of right translation. As $M^{K^{\infty}}\subseteq\overline{\pi}^{K^{\infty}} = \pi^{K^{\infty}}$, the subspace $M\cap\pi$ is nonzero and invariant in the irreducible $\pi$, whence $M\cap\pi = \pi$ and $M = \overline{\pi}$.

Next, let $\varphi:\pi_{j}\xrightarrow{\sim}\pi_{j'}$ be an isomorphism of smooth representations. Then $B(v,w) := \langle\varphi(v),\varphi(w)\rangle$ is a second invariant inner product on $\pi := \pi_{j}$, and we claim that $B = c\,\langle\cdot,\cdot\rangle$ for some scalar $c>0$. Indeed, on each finite-dimensional space $\pi^{K^{\infty}}$, let $T_{K^{\infty}}$ be the unique linear operator such that, for all $v,w \in \pi^{K^{\infty}}$, $B(v,w) = \langle T_{K^{\infty}}\,v,\, w\rangle$. We note that $T_{K^{\infty}}$ is positive-definite. Since both forms are invariant, each $P_{K^{\infty}}$ is self-adjoint for both. Hence, for $K'\leq K^{\infty}$, the operator $T_{K'}$ commutes with $P_{K^{\infty}}$ and restricts to $T_{K^{\infty}}$ on $\pi^{K^{\infty}}$. Therefore, the $T_{K^{\infty}}$ glue to a linear endomorphism $T$ of $\pi = \bigcup_{K^{\infty}}\pi^{K^{\infty}}$ with $B = \langle T\,\cdot\,,\,\cdot\,\rangle$, and the invariance of both forms makes $T$ equivariant. Taking an eigenvalue $c$ of $T$ on some nonzero $\pi^{K^{\infty}}$, the kernel of $T - c$ is a nonzero invariant subspace of the irreducible $\pi$, so $T = c$, and $c > 0$ by positivity. Thus, we conclude that $c^{-1/2}\varphi$ is an isometry, which extends to a unitary isomorphism $\overline{\pi_{j}}\cong\overline{\pi_{j'}}$. Conversely, a unitary isomorphism $\overline{\pi_{j}}\cong\overline{\pi_{j'}}$ preserves $K^{\infty}$-invariance and smoothness, hence restricts to an isomorphism $\pi_{j}\cong\pi_{j'}$.

Finally, we group the $\pi_{j}$ into isomorphism classes. By the preceding two paragraphs and Schur's lemma for unitary representations (see~\cite[\S3.1]{folland2016course}), closures within one class are unitarily isomorphic, while closures in different classes are orthogonal. Each class $c$ is finite, since $\#c\cdot\dim\pi_{j_{1}}^{K^{\infty}} \leq \dim V^{K^{\infty}}$ for $j_{1}\in c$ and any $K^{\infty}$ with $\pi_{j_{1}}^{K^{\infty}}\neq 0$. As $V_{\mathrm{sm}}$ is dense in $V$, the subspaces $I_{c} := \overline{\bigoplus_{j\in c}\pi_{j}}$ yield an orthogonal decomposition $V = \widehat{\bigoplus_{c}}I_{c}$. It therefore remains to prove the following claim, applied to $N = I_{c}$ and $D = \bigoplus_{j\in c}\pi_{j}$: if a closed invariant subspace $N\subseteq V$ contains a dense invariant subspace $D\subseteq V_{\mathrm{sm}}$ isomorphic to $\pi^{\oplus r}$, with $\pi$ irreducible smooth and $r<\infty$, then $N\cong\overline{\pi}^{\,\oplus r}$ unitarily. This can be proved by induction on $r$. The case $r=1$ is clear by the previous paragraph. For the inductive step, pick a summand $\pi'\subseteq D$. Consider the orthogonal projection of $N$ onto $N\ominus\overline{\pi'}$, where $N \ominus \overline{\pi'}$ is the orthogonal complement of $\overline{\pi'}$ inside $N$.  This orthogonal projection preserves smoothness and maps $D$ onto a dense invariant subspace of $N\ominus\overline{\pi'}$. The image of the projection is a quotient of $D\cong\pi^{\oplus r}$, hence isomorphic to $\pi^{\oplus(r-1)}$ by a count of $K^{\infty}$-fixed vectors. Now the inductive hypothesis applies and the claim follows.
\end{proof}

\begin{prop}[(Frobenius reciprocity)] \label{prop: frobenius_rec}
Maintain the standing assumptions on $H$ and its representation $(\rho, W)$, and let $(\pi, V)$ be a smooth representation of $\tilde{T}_{z}(\mathbb{A})$. There is a natural isomorphism
\begin{align*}
\Phi:\mathrm{Hom}_{\tilde{T}_{z}(\mathbb{A})}(V, \mathrm{Ind}_{H}^{\tilde{T}_{z}(\mathbb{A})}W) &\xrightarrow{\sim} \mathrm{Hom}_{H}(\mathrm{Res}\,V, W)\\
\alpha &\mapsto \Phi(\alpha): v\mapsto \alpha(v)(1).
\end{align*}
\end{prop}
\begin{proof}
It is straightforward to check that $\Phi(\alpha)$ thus defined is continuous and is an $H$-morphism indeed. We define
$$\Psi:\mathrm{Hom}_{H}(\mathrm{Res}\,V, W) \to   \mathrm{Hom}_{\tilde{T}_{z}(\mathbb{A})}(V, \mathrm{Ind}_{H}^{\tilde{T}_{z}(\mathbb{A})}W)$$ 
as follows: for an $H$-morphism $\beta: V\to W$, we let $\Psi(\beta)(v)$ be the map that sends any $t \in \tilde{T}_{z}(\mathbb{A})$ to $\beta(\pi(t)v) \in W$.

Let $h \in H$ and $t \in \tilde{T}_{z}(\mathbb{A})$. $\Psi(\beta)(v)$ sends $ht$ to $\beta(\pi(ht)v) = \rho(h)\beta(\pi(t)v)$, since $\beta$ intertwines $\mathrm{Res}\,V$ with $\rho$. Since $H \backslash \tilde{T}_{z}(\mathbb{A})$ is compact (as $H$ contains $T(\mathbb{A})$), $\Psi(\beta)(v)$ is square-integrable. According to Proposition~\ref{prop: automatic_smoothness}, it remains to show that $\Psi(\beta)(v)$ is smooth under the action of $\tilde{T}_{z} (\mathbb{A}^{\infty})$. Since $(\pi,V)$ is a smooth representation of $\tilde{T}_{z}(\mathbb{A})$,  there exists a compact open subgroup $K \leq \tilde{T}_{z}(\mathbb{A}^{\infty})$ such that $\pi(K)v = v$. Then we can show that $K$ also fixes $\Psi(\beta)(v)$ when acting by right translation. Indeed, for any $k \in K$ and $t \in \tilde{T}_z(\mathbb{A})$:
$$ [\Psi(\beta)(v)](tk) = \beta(\pi(tk)v) = \beta(\pi(t)\pi(k)v) = \beta(\pi(t)v) = [\Psi(\beta)(v)](t). $$
Thus, $\Psi(\beta)(v)$ is right $K$-invariant.

Routine checks show that $\Psi(\beta)$ is $\tilde{T}_{z}(\mathbb{A})$-intertwining. Furthermore, we must verify that the linear map $\Psi(\beta) \colon V \to \mathrm{Ind}_{H}^{\tilde{T}_z(\mathbb{A})} W$ is continuous. Because $V$ is the strict inductive limit of $V^{K_{n}}$, which are Fr\'{e}chet spaces, it suffices to show that the restriction of $\Psi(\beta)$ to each $V^{K_{n}}$ is continuous. By the preceding paragraph, $\Psi(\beta)$ maps $V^{K_{n}}$ into the Fr\'{e}chet space $(\mathrm{Ind}_{H}^{\tilde{T}_z(\mathbb{A})} W)^{K_{n}}$. Suppose a sequence $v_n \to v$ in $V^{K_{n}}$ and $\Psi(\beta)(v_n) \to f$ in $(\mathrm{Ind}_{H}^{\tilde{T}_z(\mathbb{A})} W)^{K_{n}}$. Because point-evaluation $\varphi \mapsto \varphi(t)$ is continuous in the Fr\'{e}chet topology, evaluating at any $t \in \tilde{T}_z(\mathbb{A})$ gives
$$ f(t) = \lim \Psi(\beta)(v_n)(t) = \lim \beta(\pi(t)v_n) = \beta(\pi(t)v) = \Psi(\beta)(v)(t). $$ 
Thus, $f = \Psi(\beta)(v)$, which implies that the restriction $\Psi(\beta)|_{V^{K_{n}}}$ has a closed graph. By the Closed Graph Theorem for Fr\'{e}chet spaces (see~\cite[\S17]{treves2016topological}), the restriction $\Psi(\beta)|_{V^{K_{n}}}$ is continuous. Thus we have shown that $\Psi(\beta)$ is continuous on $V$, hence lies in $\mathrm{Hom}_{\tilde{T}_{z}(\mathbb{A})}(V, \mathrm{Ind}_{H}^{\tilde{T}_{z}(\mathbb{A})}W)$. 

Finally, it is easy to check that $\Phi$ and $\Psi$ are inverses of each other.
\end{proof}

\begin{prop}[(Induction in steps)]
    \label{prop: induction_in_steps} 
    Let $H \leq H'$ be closed unimodular subgroups of $\tilde{T}_z(\mathbb{A})$ (here we need not assume they contain $T(\mathbb{A})$). Let $(\rho, W)$ be a unitary representation of $H$. Let $V_{\mathrm{in}} := L^{2}\text{-}\mathrm{Ind}_{H}^{H'}W$. Then there is a natural $\tilde{T}_{z}(\mathbb{A})$-isomorphism of LF-spaces:
    \begin{align}\label{equation: induction_in_steps}
        \mathrm{Ind}_{H'}^{\tilde{T}_{z}(\mathbb{A})}V_{\mathrm{in}} = \left(L^2\text{-}\mathrm{Ind}_{H'}^{\tilde{T}_{z}(\mathbb{A})}V_{\mathrm{in}}\right)_{\mathrm{sm}} \xrightarrow{\sim} \mathrm{Ind}_{H}^{\tilde{T}_{z}(\mathbb{A})} W.    
    \end{align}
\end{prop}
\begin{proof}
    By the induction-in-steps theorem for unitary representations~\cite[Thm.~6.14]{folland2016course}, there is a unitary equivalence:
    \begin{align}\label{equation: L2_induction_in_steps}
        L^2\text{-}\mathrm{Ind}_{H'}^{\tilde{T}_{z}(\mathbb{A})}\left(L^{2}\text{-}\mathrm{Ind}_{H}^{H'}W\right) \xrightarrow{\sim} L^2\text{-}\mathrm{Ind}_{H}^{\tilde{T}_{z}(\mathbb{A})} W,
    \end{align}
    which induces a natural topological isomorphism between the smooth parts:
    $$ \left(L^2\text{-}\mathrm{Ind}_{H'}^{\tilde{T}_{z}(\mathbb{A})}V_{\mathrm{in}}\right)_{\mathrm{sm}} \xrightarrow{\sim} \mathrm{Ind}_{H}^{\tilde{T}_{z}(\mathbb{A})} W.$$
\end{proof}

\begin{rmk}
The isomorphism~\eqref{equation: induction_in_steps} can be explicitly given as:
$$\Phi(f)(t) = f(t)(1_{H'}),$$
for $f \in \mathrm{Ind}_{H'}^{\tilde{T}_{z}(\mathbb{A})}V_{\mathrm{in}}$ and $t \in \tilde{T}_{z}(\mathbb{A})$, with inverse
$$\Psi(\phi)(t)(h') = \phi(h't),$$
for $\phi \in \mathrm{Ind}_{H}^{\tilde{T}_{z}(\mathbb{A})} W$, $t \in \tilde{T}_{z}(\mathbb{A})$ and $h' \in H'$. We omit the proof of this since it will not be used later. 
 
Moreover, we can show that $f \in \mathrm{Ind}_{H'}^{\tilde{T}_{z}(\mathbb{A})}V_{\mathrm{in}}$ must take values in  $(V_{\mathrm{in}})_{\mathrm{sm}} = \mathrm{Ind}_{H}^{H'}W$. For this reason, by abuse of notation, we simply denote the left-hand side of~\eqref{equation: induction_in_steps} by $\mathrm{Ind}_{H'}^{\tilde{T}_{z}(\mathbb{A})}\left(\mathrm{Ind}_{H}^{H'}W\right)$.     
\end{rmk}

\subsection{A preliminary decomposition}\label{sec: pre-decomp}
As the first step towards decomposing $L^{2}(A_{T}\tilde{T}_{z}(F) \backslash \tilde{T}_{z}(\mathbb{A}))$, we extract from it a dense subspace $\mathcal{A}(\tilde{T}_{z})$, on which $\tilde{T}_{z}(\mathbb{A})$ acts smoothly. 

We consider an intermediate closed subgroup $T(\mathbb{A})\tilde{T}_{z}(F)$ of $\tilde{T}_{z}(\mathbb{A})$, which contains $T(\mathbb{A})$ as a subgroup of finite index:
\begin{align*}
    1 \to T(\mathbb{A}) \to T(\mathbb{A})\tilde{T}_{z}(F) \to A(F)^{[z]} \to 1.
\end{align*}
We endow the coset space $A_{T}\tilde{T}_{z}(F)\backslash T(\mathbb{A})\tilde{T}_{z}(F)$  with the quotient topology and normalize the measure on it following the conventions in \S\ref{section: measures}. There is a natural homeomorphism $p$:
\begin{align*}
A_{T}T(F)\backslash T(\mathbb{A}) \xrightarrow[\sim]{p}
A_{T}\tilde{T}_{z}(F)\backslash T(\mathbb{A})\tilde{T}_{z}(F),
\end{align*}
given by sending the $A_{T}T(F)$-coset of $x$ to the $A_{T}\tilde{T}_{z}(F)$-coset of $x$, for any $x \in T(\mathbb{A})$. Indeed, one can see its inverse $p^{-1}$  is given by sending the $A_{T}\tilde{T}_{z}(F)$-coset of $xt$ to the $A_{T}T(F)$-coset of $t^{-1}xt$, for any $x \in T(\mathbb{A})$ and $t \in \tilde{T}_{z}(F)$. Using the fact that $T(\mathbb{A})$ is of finite index in $T(\mathbb{A})\tilde{T}_{z}(F)$, continuity of $p$ and $p^{-1}$ can be checked immediately. 

Furthermore, one can check that $p$ is measure-preserving and $T(\mathbb{A})$-equivariant by construction. However, we warn the reader that $p$ is not an isomorphism between quotient groups since $A_{T}\tilde{T}_{z}(F)$ need not be normal in $T(\mathbb{A})\tilde{T}_{z}(F)$ at all. Nonetheless, $p$ induces the identification
\begin{align}
L^{2}\left(A_{T}T(F)\backslash T(\mathbb{A})\right)\xleftarrow[\sim]{p^{*}}L^{2}(A_{T}\tilde{T}_{z}(F)\backslash T(\mathbb{A})\tilde{T}_{z}(F) ), \label{ident}
\end{align}
where $p^{*}$ is the pullback along $p$. With $T(\mathbb{A})$ acting on both sides by right translation, $p^{*}$ is an isomorphism of unitary $T(\mathbb{A})$-representations. The larger group $T(\mathbb{A})\tilde{T}_{z}(F)$ also acts on the right-hand side of~\eqref{ident} by right translation. We will describe this action in Fact~\ref{fact: Hecke_ch} below.

Let $H(T)$ be the set of all Hecke characters of $T$ that are trivial on $A_{T}$. These are precisely the characters of the compact group $A_{T}T(F)\backslash T(\mathbb{A})$, and hence are unitary. We let $\tilde{T}_{z}(F)$ act on $H(T)$ by conjugation,
\begin{align*}
    (t_{0}\cdot\chi)(x) := \chi(t_{0}^{-1}xt_{0}), \qquad x \in T(\mathbb{A}),\ t_{0} \in \tilde{T}_{z}(F).
\end{align*} 

Given $\chi \in H(T)$, let $\tilde{\chi} := (p^{*})^{-1}(\chi)$ be its lift to $T(\mathbb{A})\tilde{T}_{z}(F)$. Explicitly, $\tilde{\chi}$ is the element of $L^{2}(A_{T}\tilde{T}_{z}(F)\backslash T(\mathbb{A})\tilde{T}_{z}(F))$ characterized by $\tilde{\chi}(xt) = \chi(t^{-1}xt)$ for $x\in T(\mathbb{A})$ and $t\in\tilde{T}_{z}(F)$. In particular, $\tilde{\chi}$ restricts to $\chi$ on $T(\mathbb{A})$. The following fact states that, under the right-translation action of $T(\mathbb{A})\tilde{T}_{z}(F)$ on $L^{2}(A_{T}\tilde{T}_{z}(F)\backslash T(\mathbb{A})\tilde{T}_{z}(F))$, the subgroup $T(\mathbb{A})$ scales each $\tilde{\chi}$, while $\tilde{T}_{z}(F)$ permutes them.
\begin{fact} \label{fact: Hecke_ch}
With respect to right translation,
\begin{align*}
    x_{0}\cdot \tilde{\chi} = \chi(x_{0})\tilde{\chi},   \qquad x_{0} \in T(\mathbb{A}),
    \end{align*}
    and 
\begin{align*}
    t_{0}\cdot\tilde{\chi} = \widetilde{t_{0}\cdot \chi}, \qquad t_{0} \in \tilde{T}_{z}(F).
\end{align*}    
\end{fact}

\begin{proof}
The first identity follows from the $T(\mathbb{A})$-equivariance of $p^{*}$: right translation by $x_{0}\in T(\mathbb{A})$ multiplies $\chi$ by $\chi(x_{0})$, and hence multiplies $\tilde{\chi} = (p^{*})^{-1}(\chi)$ by the same scalar.

For the second, fix $t_{0}\in \tilde{T}_{z}(F)$. For every $x \in T(\mathbb{A})$,
\begin{align*}
(t_{0}\cdot\tilde{\chi})(x) = \tilde{\chi}(xt_{0}) = \chi(t_{0}^{-1}xt_{0}) = (t_{0}\cdot\chi)(x) = \widetilde{t_{0}\cdot\chi}(x),
\end{align*}
where the second equality is the defining property of $\tilde{\chi}$ and the last uses that $\widetilde{t_{0}\cdot\chi}$ restricts to $t_{0}\cdot\chi$ on $T(\mathbb{A})$. Since every coset in $A_{T}\tilde{T}_{z}(F)\backslash T(\mathbb{A})\tilde{T}_{z}(F)$ has a representative in $T(\mathbb{A})$, we conclude $t_{0}\cdot\tilde{\chi} = \widetilde{t_{0}\cdot\chi}$.
\end{proof}

The Peter--Weyl Theorem ensures:
\begin{align*}
    \bigoplus_{\chi \in H(T)} \mathbb{C}\cdot\chi\;  \text{ is a dense subspace of } \; L^{2}\left(A_{T}T(F)\backslash T(\mathbb{A})\right).
\end{align*}
Lifting the above subspace using $(p^{*})^{-1}$, we define:
\begin{align*}
    W := \bigoplus_{\chi\in H(T)} \mathbb{C}\cdot\tilde{\chi}\,,  \text{ which is a dense subspace of } \;  L^{2}(A_{T}\tilde{T}_{z}(F)\backslash T(\mathbb{A})\tilde{T}_{z}(F) ).
\end{align*}
Moreover, $W$ is stable under the right translation by $T(\mathbb{A})\tilde{T}_{z}(F)$, by virtue of Fact~\ref{fact: Hecke_ch}.

In Claim~\ref{claim: isom_for_each_block}, we show that $W$ decomposes into finite-dimensional blocks indexed by the $\tilde{T}_{z}(F)$-orbits in $H(T)$, each an irreducible unitary representation of $T(\mathbb{A})\tilde{T}_{z}(F)$. Given $\chi_{0} \in H(T)$, denote its $\tilde{T}_{z}(F)$-orbit by
\begin{align*}
\mathcal{O}_{\chi_{0}} := \{t\cdot \chi_{0} \mid t\in \tilde{T}_{z}(F)\}.
\end{align*}
Because $\tilde{T}_{z}(F)/T(F)$ is finite, the orbit $\mathcal{O}_{\chi_{0}}$ is finite. Thus, the subspace
\begin{align*}
W_{\chi_{0}} := \bigoplus_{\chi \in \mathcal{O}_{\chi_{0}}} \mathbb{C}\cdot \tilde{\chi} \subset W
\end{align*}
is a finite-dimensional unitary representation of $T(\mathbb{A})\tilde{T}_{z}(F)$. Moreover, $W_{\chi_{0}}$ is an irreducible $T(\mathbb{A})\tilde{T}_{z}(F)$-representation. Indeed, $W_{\chi_{0}}$ is the direct sum of the one-dimensional lines $\mathbb{C}\cdot\tilde{\chi}$, on which $T(\mathbb{A})$ acts through the character $\chi$ (by Fact~\ref{fact: Hecke_ch}), and these characters are pairwise distinct. Therefore, every $T(\mathbb{A})$-invariant subspace of $W_{\chi_{0}}$ is a sum of these lines. By Fact~\ref{fact: Hecke_ch} again, $\tilde{T}_{z}(F)$ permutes these lines transitively since $\mathcal{O}_{\chi_{0}}$ is a single orbit. It follows that the only nonzero subspace invariant under both $T(\mathbb{A})$ and $\tilde{T}_{z}(F)$ is $W_{\chi_{0}}$ itself.

Let $\chi_{0}^{*}$ be the character on $T(\mathbb{A})\tilde{T}_{z}(F)^{\chi_{0}}$ that extends $\chi_{0}$ trivially: 
$$\chi_{0}^{*}(xt) = \chi_{0}(x),\qquad x \in T(\mathbb{A}), t \in \tilde{T}_{z}(F)^{\chi_{0}}.$$ 

\begin{claim}\label{claim: isom_for_each_block}
There is a natural isomorphism of irreducible $T(\mathbb{A})\tilde{T}_{z}(F)$-representations:
\begin{align*}
    W_{\chi_{0}} := \bigoplus_{\chi \in \mathcal{O}_{\chi_{0}} } \mathbb{C}\cdot \tilde{\chi} \cong \mathrm{Ind}_{T(\mathbb{A})\tilde{T}_{z}(F)^{\chi_{0}}}^{T(\mathbb{A})\tilde{T}_{z}(F)} \chi_{0}^{*}\,.
\end{align*}
\end{claim}

\begin{proof}
Since $\tilde{T}_{z}(F)^{\chi_{0}}$ is of finite index in $\tilde{T}_{z}(F)$, we may fix a set of left (resp. right) coset representatives $t_{i}$'s (resp. $t_{i}^{-1}$'s). Then the $T(\mathbb{A})\tilde{T}_{z}(F)$-isomorphism above can be written explicitly as
\begin{align*}
    \sum c_{j}\widetilde{t_{j}\cdot \chi_{0}} \mapsto f:T(\mathbb{A})\tilde{T}_{z}(F) &\to \mathbb{C},\\
     xt_{j}^{-1} &\mapsto \chi_{0}^{*}(x)c_{j}, \quad x \in T(\mathbb{A})\tilde{T}_{z}(F)^{\chi_{0}},
\end{align*} 
with inverse
\begin{align*}
     \sum f(t_{j}^{-1})\widetilde{t_{j}\cdot \chi_{0}} \mapsfrom f \in \mathrm{Ind}_{T(\mathbb{A})\tilde{T}_{z}(F)^{\chi_{0}}}^{T(\mathbb{A})\tilde{T}_{z}(F)} \chi_{0}^{*}\,.\qquad \qquad \qquad \qquad \;\;
\end{align*}
Finally, $T(\mathbb{A})\tilde{T}_{z}(F)$-equivariance can be checked straightforwardly.
\end{proof}

Since $\chi_{0}^{*}$ is a character and $T(\mathbb{A})\tilde{T}_{z}(F)^{\chi_{0}}$ is of finite index in $T(\mathbb{A})\tilde{T}_{z}(F)$, the smooth induction in Claim~\ref{claim: isom_for_each_block} is the same as the unitary induction: $\mathrm{Ind}_{T(\mathbb{A})\tilde{T}_{z}(F)^{\chi_{0}}}^{T(\mathbb{A})\tilde{T}_{z}(F)} \chi_{0}^{*} = L^2\text{-}\mathrm{Ind}_{T(\mathbb{A})\tilde{T}_{z}(F)^{\chi_{0}}}^{T(\mathbb{A})\tilde{T}_{z}(F)} \chi_{0}^{*}$.

By applying Claim~\ref{claim: isom_for_each_block} to each orbit, we obtain:
\begin{align*}
W = \bigoplus_{\chi_{0}\in H(T)/\tilde{T}_{z}(F)} W_{\chi_{0}}
=\bigoplus_{\chi_{0}\in H(T)/\tilde{T}_{z}(F)} \mathrm{Ind}_{T(\mathbb{A})\tilde{T}_{z}(F)^{\chi_{0}}}^{T(\mathbb{A})\tilde{T}_{z}(F)} \chi_{0}^{*}\,.
\end{align*}
Because $W$ is dense in $L^{2}(A_{T}\tilde{T}_{z}(F)\backslash T(\mathbb{A})\tilde{T}_{z}(F))$, we have actually shown:
\begin{align}\label{equation: L2-decomp-of-W}
L^{2}(A_{T}\tilde{T}_{z}(F)\backslash T(\mathbb{A})\tilde{T}_{z}(F) ) = \widehat{\bigoplus_{\chi_{0}\in H(T)/\tilde{T}_{z}(F)}} \mathrm{Ind}_{T(\mathbb{A})\tilde{T}_{z}(F)^{\chi_{0}}}^{T(\mathbb{A})\tilde{T}_{z}(F)} \chi_{0}^{*}\,.
\end{align}
Here, the wide hat over the algebraic direct sum denotes its $L^{2}$-completion.

Next, we consider the smooth and unitary induced representations of each block $W_{\chi_{0}}$ from $T(\mathbb{A})\tilde{T}_{z}(F)$ to $\tilde{T}_{z}(\mathbb{A})$. In view of induction in steps (see~\eqref{equation: induction_in_steps} and~\eqref{equation: L2_induction_in_steps}), we have:
\begin{align*}
U_{\chi_{0}} := \mathrm{Ind}_{T(\mathbb{A})\tilde{T}_{z}(F)}^{\tilde{T}_{z}(\mathbb{A})} W_{\chi_{0}} = \mathrm{Ind}_{T(\mathbb{A})\tilde{T}_{z}(F)}^{\tilde{T}_{z}(\mathbb{A})}\left(\mathrm{Ind}_{T(\mathbb{A})\tilde{T}_{z}(F)^{\chi_{0}}}^{T(\mathbb{A})\tilde{T}_{z}(F)}\chi_{0}^{*} \right) = \mathrm{Ind}_{T(\mathbb{A})\tilde{T}_{z}(F)^{\chi_{0}}}^{\tilde{T}_{z}(\mathbb{A})} \chi_{0}^{*}\,
\end{align*}
and 
\begin{align*}
\overline{U}_{\chi_{0}} := L^2\text{-}\mathrm{Ind}_{T(\mathbb{A})\tilde{T}_{z}(F)}^{\tilde{T}_{z}(\mathbb{A})} W_{\chi_{0}} = L^2\text{-}\mathrm{Ind}_{T(\mathbb{A})\tilde{T}_{z}(F)}^{\tilde{T}_{z}(\mathbb{A})}\left(L^2\text{-}\mathrm{Ind}_{T(\mathbb{A})\tilde{T}_{z}(F)^{\chi_{0}}}^{T(\mathbb{A})\tilde{T}_{z}(F)}\chi_{0}^{*} \right) = L^2\text{-}\mathrm{Ind}_{T(\mathbb{A})\tilde{T}_{z}(F)^{\chi_{0}}}^{\tilde{T}_{z}(\mathbb{A})} \chi_{0}^{*}\,.
\end{align*}

Now, we can piece together the decompositions obtained so far. Based on~\eqref{equation: L2-decomp-of-W} and induction in steps~\eqref{equation: L2_induction_in_steps}, we have
\begin{align*}
L^{2}(A_{T}\tilde{T}_{z}(F)\backslash \tilde{T}_{z}(\mathbb{A})) &= L^2\text{-}\mathrm{Ind}_{T(\mathbb{A})\tilde{T}_{z}(F)}^{\tilde{T}_{z}(\mathbb{A})} \left[L^2\text{-}\mathrm{Ind}_{A_{T}\tilde{T}_{z}(F)}^{T(\mathbb{A})\tilde{T}_{z}(F)} \mathbb{1}\right]\\
&\overset{\eqref{equation: L2-decomp-of-W}}{=} L^2\text{-}\mathrm{Ind}_{T(\mathbb{A})\tilde{T}_{z}(F)}^{\tilde{T}_{z}(\mathbb{A})} \left(\widehat{\bigoplus_{\chi_{0}\in H(T)/\tilde{T}_{z}(F)}} L^2\text{-}\mathrm{Ind}_{T(\mathbb{A})\tilde{T}_{z}(F)^{\chi_{0}}}^{T(\mathbb{A})\tilde{T}_{z}(F)} \chi_{0}^{*}\right)\\
&= \widehat{\bigoplus_{\chi_{0}\in H(T)/\tilde{T}_{z}(F)}} L^2\text{-}\mathrm{Ind}_{T(\mathbb{A})\tilde{T}_{z}(F)}^{\tilde{T}_{z}(\mathbb{A})} \left(L^2\text{-}\mathrm{Ind}_{T(\mathbb{A})\tilde{T}_{z}(F)^{\chi_{0}}}^{T(\mathbb{A})\tilde{T}_{z}(F)} \chi_{0}^{*}\right)\\
&= \widehat{\bigoplus_{\chi_{0}\in H(T)/\tilde{T}_{z}(F)}} L^2\text{-}\mathrm{Ind}_{T(\mathbb{A})\tilde{T}_{z}(F)^{\chi_{0}}}^{\tilde{T}_{z}(\mathbb{A})} \chi_{0}^{*}\,\\
&= \widehat{\bigoplus_{\chi_{0}\in H(T)/\tilde{T}_{z}(F)}} \overline{U}_{\chi_{0}}.
\end{align*}
As before, each wide hat over an algebraic direct sum denotes its $L^{2}$-completion.

We define $\mathcal{A}(\tilde{T}_{z})$ as the algebraic direct sum of the blocks $U_{\chi_{0}}$:
\begin{align}\label{equation: preliminary_decomp}
    \mathcal{A}(\tilde{T}_{z}) := \bigoplus_{\chi_{0}\in H(T)/\tilde{T}_{z}(F)} U_{\chi_{0}} = \bigoplus_{\chi_{0}\in H(T)/\tilde{T}_{z}(F)} \mathrm{Ind}_{T(\mathbb{A})\tilde{T}_{z}(F)^{\chi_{0}}}^{\tilde{T}_{z}(\mathbb{A})} \chi_{0}^{*}\,.
\end{align}
This is a dense subspace of $L^{2}(A_{T}\tilde{T}_{z}(F)\backslash \tilde{T}_{z}(\mathbb{A}))$. Indeed, each block $\overline{U}_{\chi_{0}}$ satisfies the hypotheses of Corollary~\ref{cor: admissibility}, so its smooth part $U_{\chi_{0}}$ is dense in $\overline{U}_{\chi_{0}}$. Moreover, $U_{\chi_{0}}$ is semisimple, as will be shown in Corollary~\ref{corollary: semisimplicity&decomp_2}. Therefore, Corollary~\ref{cor: admissibility}(3), applied to each block of the orthogonal decomposition displayed above, shows that $L^{2}(A_{T}\tilde{T}_{z}(F)\backslash \tilde{T}_{z}(\mathbb{A}))$ decomposes as a Hilbert direct sum of irreducible unitary representations, and that $\mathcal{A}(\tilde{T}_{z})$ has the same irreducible constituents as the $L^{2}$-space with the same multiplicities.

\subsection{Preparation for multiplicity calculation}\label{section: prepformult}

In this section, we lay the foundation for the calculation of multiplicity on the automorphic side by investigating the interaction between restricted tensor products and induced representations. As a consequence, we show the semisimplicity of $\mathcal{A}(\tilde{T}_{z})$ as a $\tilde{T}_{z}(\mathbb{A})$-module. Throughout this part, we fix $\chi \in H(T)$. In other words, $\chi$ is a fixed Hecke character of $T(\mathbb{A})$ that is trivial on $A_{T}$.

To start with, we factorize the induced representation $\mathrm{Ind}_{T(\mathbb{A})}^{\tilde{T}_{z}(\mathbb{A})}\chi$ into a restricted tensor product of local induced representations. Let $S'$ denote the union of $S$ (the set of places fixed in \S\ref{section:model}) with the set of places where $\chi$ ramifies. We understand $\chi \cong \otimes'_{v}\chi_{v}$ as the restricted tensor product of $\chi_{v}$, where the restriction is taken with respect to $1 \in V_{\chi_{v}}$ for every place $v \notin S'$.

A $\tilde{T}_{z}(F_{v})$-representation is said to be $\tilde{T}_{z}(\mathcal{O}_{v})$-unramified if it contains a nonzero vector fixed by $\tilde{T}_{z}(\mathcal{O}_{v})$.

\begin{prop}\label{prop:local-unram}
Suppose $v \notin S'$, then there exists a unique irreducible $\tilde{T}_{z}(\mathcal{O}_{v})$-unramified constituent $\eta^{0}_{v}$ of $\mathrm{Ind}_{T(F_{v})}^{\tilde{T}_{z}(F_{v})} \chi_{v}$. Moreover, the subspace of $\tilde{T}_{z}(\mathcal{O}_{v})$-fixed vectors of $\eta_{v}^{0}$ is spanned by
the distinguished element $f^{0}_{v}\in \mathrm{Ind}_{T(F_{v})}^{\tilde{T}_{z}(F_{v})}\chi_{v}$, characterized by taking value $1$ on $\tilde{T}_{z}(\mathcal{O}_{v})$. 
\end{prop}

\begin{proof}
Recall we have $\tilde{T}_{z}(F_{v}) = T(F_{v})\tilde{T}_{z}(\mathcal{O}_{v})$ for $v \notin S'$. On the one hand, when $v \notin S'$, since $\chi_{v}$ is unramified, the element $f^{0}_{v}$ is well-defined and $\tilde{T}_{z}(\mathcal{O}_{v})$-fixed. On the other hand, whenever $f_{v} \in \mathrm{Ind}_{T(F_{v})}^{\tilde{T}_{z}(F_{v})} \chi_{v}$ is $\tilde{T}_{z}(\mathcal{O}_{v})$-fixed, it will be fully determined by $f_{v}(1)$ and we actually have $f_{v} = f_{v}(1)f^{0}_{v}$.

It follows that the unique irreducible $\tilde{T}_{z}(\mathcal{O}_{v})$-unramified constituent $\eta^{0}_{v}$ is generated by $f_{v}^{0}$. 
\end{proof}

We can now form a restricted tensor product of the local induced representations
\begin{align*}
\sideset{}{'}\bigotimes_{v}  \mathrm{Ind}_{T(F_{v})}^{\tilde{T}_{z}(F_{v})}\chi_{v},
\end{align*}
where the restriction is taken with respect to the distinguished vectors $f^{0}_{v}$ introduced in the previous proposition. There is a natural embedding
\begin{align*}
J: \sideset{}{'}\bigotimes_{v}  \mathrm{Ind}_{T(F_{v})}^{\tilde{T}_{z}(F_{v})} \chi_{v} \hookrightarrow \mathrm{Ind}_{T(\mathbb{A})}^{\tilde{T}_{z}(\mathbb{A})} \chi
\end{align*}
given by $\displaystyle J(\otimes_{v} f_{v})(t) = \prod_{v} f_{v}(t_{v})$. The image consists of smooth vectors. Indeed, since each $f_{v}$ is smooth and $f_{v}=f^{0}_{v}$ takes value $1$ on $\tilde{T}_{z}(\mathcal{O}_{v})$ for almost all $v$, the product $\prod_{v}f_{v}(t_{v})$ is smooth at the archimedean places and $K^{\infty}$-invariant for some open compact subgroup $K^{\infty}\leq \tilde{T}_{z}(\mathbb{A}^{\infty})$.

\begin{claim}
The natural embedding $J$ is also surjective. \label{claim: inductionsurjective}
\end{claim}

\begin{proof}
We recall the natural projection map $\displaystyle p: \tilde{T}_{z}(\mathbb{A}) \to A(\mathbb{A})^{[z]} = \prod A(F_{v})^{[z_{v}]}$. We pick a set-theoretic section $s$ of $p$ such that $s_{v}$ takes values in $\tilde{T}_{z}(\mathcal{O}_{v})$ for $v \notin S'$. Such a section exists since $S'$ contains all the exceptional places where  $\tilde{T}_{z}(F_{v}) = T(F_{v})\tilde{T}_{z}(\mathcal{O}_{v})$ fails. We note that $s$ thus constructed is continuous.

Let $f \in \mathrm{Ind}_{T(\mathbb{A})}^{\tilde{T}_{z}(\mathbb{A})} \chi$. By definition, $f$ is smooth on $\tilde{T}_{z}(\mathbb{A})$. In particular, $f$ is $K^{\infty}$-invariant for some open compact subgroup $K^{\infty}\leq \tilde{T}_{z}(\mathbb{A}^{\infty})$. We claim that the composition $f\circ s: A(\mathbb{A})^{[z]} \to \mathbb{C}$ is again smooth (locally constant). Indeed, since $K^{\infty}$ is open in the restricted direct product $\tilde{T}_{z}(\mathbb{A}^{\infty})$, it contains $\prod_{v \notin \Sigma_{0}} \tilde{T}_{z}(\mathcal{O}_{v})$ for some finite set of places $\Sigma_{0}$, which we may take to contain $S'$ and all the infinite places. Suppose that $a, a' \in A(\mathbb{A})^{[z]}$ agree at every place in $\Sigma_{0}$. Then the element $\kappa := s(a)^{-1}s(a')$ has trivial component at each place in $\Sigma_{0}$, while its component at each place $v \notin \Sigma_{0}$ is $s_{v}(a_{v})^{-1}s_{v}(a'_{v}) \in \tilde{T}_{z}(\mathcal{O}_{v})$, thanks to our choice of the section $s$. Therefore, $\kappa$ lies in $\prod_{v \notin \Sigma_{0}} \tilde{T}_{z}(\mathcal{O}_{v}) \subseteq K^{\infty}$, and the right $K^{\infty}$-invariance of $f$ yields
$$f(s(a')) = f(s(a)\kappa) = f(s(a)).$$
In other words, $f \circ s$ is constant on each coset of the open subgroup $\textstyle\prod_{v \in \Sigma_{0}} \{1\} \times \prod_{v \notin \Sigma_{0}} A(F_{v})^{[z_{v}]}$, and the claim follows. Here we recall that $A(\mathbb{A})^{[z]}$ is endowed with product topology, hence compact. Therefore, the image of $f\circ s$ is finite, which we denote by $\{c_{1},\dots, c_{n}\}$. Now we note that $A(\mathbb{A})^{[z]}$ becomes a finite disjoint union of the preimages $(f\circ s)^{-1}(c_{i})$, which are both open and closed. In particular, each preimage $(f\circ s)^{-1}(c_{i})$ is compact open, hence a union of finitely many cylinder sets, each of the form~\eqref{cylinder}. We note that a finite union of cylinder sets is again a cylinder set. Therefore, each preimage is itself in the form
\begin{align}\label{cylinder}
(f\circ s)^{-1}(c_{i}) = U_{i} \times \prod_{v\notin S_{i}}A(F_{v})^{[z_{v}]},
\end{align}
where $S_i$ is a finite set of places containing $S'$, and $U_i \subseteq \prod_{v \in S_i} A(F_v)^{[z_v]}$.

By the finiteness of the image, the union $\displaystyle \Sigma := \bigcup_{i} S_{i}$ is a finite set. Therefore, we can factorize $f\circ s$: 
\begin{align*}
f\circ s = g_{1}\otimes g_{2},
\end{align*}
where  $\displaystyle g_{2} \equiv 1: \prod_{v\notin \Sigma}A(F_{v})^{[z_{v}]} \to \mathbb{C}$ is the constant function with value $1$ and $\displaystyle g_{1}: \prod_{v\in \Sigma}A(F_{v})^{[z_{v}]} \to \mathbb{C}$ is a function on the finite product. This further induces a factorization of $f$:
\begin{align*}
    f = f_{\Sigma} \otimes \left(\bigotimes_{v\notin \Sigma} f^{0}_{v} \right),
\end{align*}
where $f^{0}_{v}$ is the distinguished element defined in Proposition~\ref{prop:local-unram}, and 
$$f_{\Sigma}\left((x_{v})_{v\in \Sigma}\right) := \left( \prod_{v\in \Sigma} \chi_{v}(t_{v}) \right) g_{1}\left((a_{v})_{v \in \Sigma}\right),$$ 
where $x_{v} = t_{v}s_{v}(a_{v})$, lies in 
\begin{align*}
\mathrm{Ind}_{\prod_{v\in \Sigma}T(F_{v})}^{\prod_{v\in \Sigma}\tilde{T}_{z}(F_{v})} \bigotimes_{v\in \Sigma} \chi_{v}.
\end{align*}
Due to the fact that $\Sigma$ is finite and $T(F_{v})$ is of finite index in $\tilde{T}_{z}(F_{v})$ for each $v$, the induction commutes with tensor product, hence the above space is naturally isomorphic to $ \bigotimes_{v \in \Sigma}\mathrm{Ind}_{T(F_{v})}^{\tilde{T}_{z}(F_{v})} \chi_{v}$. Therefore, we conclude that $f$ lies in  
\begin{align*}
\bigotimes_{v \in \Sigma}\mathrm{Ind}_{T(F_{v})}^{\tilde{T}_{z}(F_{v})} \chi_{v} \otimes \left(\bigotimes_{v\notin \Sigma} f^{0}_{v} \right) \subset \sideset{}{'}\bigotimes_{v}  \mathrm{Ind}_{T(F_{v})}^{\tilde{T}_{z}(F_{v})} \chi_{v}.
\end{align*}
\end{proof}

\begin{cor} \label{corollary: semisimplicity&decomp_1}
$\mathrm{Ind}_{T(\mathbb{A})}^{\tilde{T}_{z}(\mathbb{A})}\chi$ is semisimple as a $\tilde{T}_{z}(\mathbb{A})$-module. 
\end{cor}
\begin{proof} 
Locally, $\mathrm{Ind}_{T(F_{v})}^{\tilde{T}_{z}(F_{v})} \chi_{v}$ is semisimple for each place $v$. We consider the following collection of irreducible subrepresentations of $\mathrm{Ind}_{T(\mathbb{A})}^{\tilde{T}_{z}(\mathbb{A})}\chi$:
\begin{align} \label{defn: S_chi}
\mathcal{S}_{\chi} = \{\otimes_{v}'\eta_{v}: \eta_{v} \text{ is an irreducible constituent of }\mathrm{Ind}_{T(F_{v})}^{\tilde{T}_{z}(F_{v})}\chi_{v}, \text{ with }\eta_{v} = \eta^{0}_{v} \text{ for almost all }v \},
\end{align}
where the restriction in $\otimes_{v}'\eta_{v}$ is taken with respect to the distinguished vectors $f^{0}_{v}$ defined in Proposition~\ref{prop:local-unram}. By Claim~\ref{claim: inductionsurjective} and the local semisimplicity, writing $\mathrm{Ind}_{T(F_{v})}^{\tilde{T}_{z}(F_{v})}\chi_{v} = \bigoplus_{i=0}^{n_{v}}\eta_{v}^{i}$ with the indexing chosen so that the $i=0$ term $\eta_{v}^{0}$ is the unramified constituent of Proposition~\ref{prop:local-unram}, we obtain
\begin{align*}
\mathrm{Ind}_{T(\mathbb{A})}^{\tilde{T}_{z}(\mathbb{A})}\chi = \sideset{}{'}\bigotimes_{v}  \mathrm{Ind}_{T(F_{v})}^{\tilde{T}_{z}(F_{v})} \chi_{v} = \sideset{}{'}\bigotimes_{v} \left(\bigoplus_{i=0}^{n_{v}} \eta_{v}^{i}\right).
\end{align*}
Here the restriction is taken with respect to the distinguished vectors $f^{0}_{v}\in\eta_{v}^{0}$. Distributing the restricted tensor product over the direct sums expresses this as $\bigoplus_{(i_{v})_{v}}\otimes'_{v}\eta_{v}^{i_{v}}$, where for almost all $v$ one has $i_{v}=0$ (so that $\eta_{v}^{i_{v}}=\eta_{v}^{0}$). Each such summand therefore lies in $\mathcal{S}_{\chi}$ (with certain multiplicities).
\end{proof}

\begin{cor}\label{corollary: semisimplicity&decomp_2}
$U_{\chi}$ and $\mathcal{A}(\tilde{T}_{z})$ are semisimple $\tilde{T}_{z}(\mathbb{A})$-modules.
\end{cor}
\begin{proof}
In view of decomposition~\eqref{equation: preliminary_decomp}, it suffices to show $U_{\chi}$ is semisimple for any $\chi$. We notice that $U_{\chi} = \mathrm{Ind}_{T(\mathbb{A})\tilde{T}_{z}(F)^{\chi}}^{\tilde{T}_{z}(\mathbb{A})} \chi^{*}$ is a submodule of $\mathrm{Ind}_{T(\mathbb{A})}^{\tilde{T}_{z}(\mathbb{A})}\chi$, hence semisimplicity follows. 
\end{proof}

\begin{rmk}\label{remark:addchi}
In the next section, we will need to work with $\tilde{T}_{z}(\mathbb{A})^{\chi} = \prod'_{v}\tilde{T}_{z}(F_{v})^{\chi_{v}}$, the stabilizer of $\chi$ in $\tilde{T}_{z}(\mathbb{A})$ under the conjugation action. 

After replacing $\tilde{T}_{z}(\mathbb{A})$ with $\tilde{T}_{z}(\mathbb{A})^{\chi}$, analogues of Claim~\ref{claim: inductionsurjective}, Corollary~\ref{corollary: semisimplicity&decomp_1}, and Corollary~\ref{corollary: semisimplicity&decomp_2} still hold, with the same proofs. For almost all $v$ ($v\notin S'$), we define $\bar{f}^{0}_{v}\in \mathrm{Ind}_{T(F_{v})}^{\tilde{T}_{z}(F_{v})^{\chi_{v}}}\chi_{v}$ characterized by taking value $1$ on $\tilde{T}_{z}(\mathcal{O}_{v})^{\chi_{v}}$. The same reasoning from the proof of Proposition~\ref{prop:local-unram} implies the unique irreducible unramified constituent of $\mathrm{Ind}_{T(F_{v})}^{\tilde{T}_{z}(F_{v})^{\chi_{v}}}\chi_{v}$, which we denote by $\bar{\eta}^{0}_{v}$, is spanned by $\bar{f}^{0}_{v}$. Similar to the collection $\mathcal{S}_{\chi}$~\eqref{defn: S_chi}, we may define the following collection of irreducible subrepresentations of $\mathrm{Ind}_{T(\mathbb{A})}^{\tilde{T}_{z}(\mathbb{A})^{\chi}} \chi$:
\begin{align}\label{defn: S_chi_bar}
\bar{\mathcal{S}}_{\chi} = \{\otimes_{v}'\bar{\eta}_{v}: \bar{\eta}_{v} \text{ is an irreducible constituent of }\mathrm{Ind}_{T(F_{v})}^{\tilde{T}_{z}(F_{v})^{\chi_{v}}}\chi_{v} \text{ and }\bar{\eta}_{v} = \bar{\eta}^{0}_{v} \text{ for almost all }v \},
\end{align}
where the restriction in $\otimes_{v}'\bar{\eta}_{v}$ is taken with respect to the $\tilde{T}_{z}(\mathcal{O}_{v})^{\chi_{v}}$-fixed vectors $\bar{f}^{0}_{v}$. It is clear that $\mathrm{Ind}_{T(\mathbb{A})}^{\tilde{T}_{z}(\mathbb{A})^{\chi}}\chi$ and $\mathrm{Ind}_{T(\mathbb{A})\tilde{T}_{z}(F)^{\chi}}^{\tilde{T}_{z}(\mathbb{A})^{\chi}}\chi^{*}$ are semisimple with constituents in $\bar{\mathcal{S}}_{\chi}$.
\end{rmk}

\begin{fact}\label{finitedimbar}
Each element $\bar{\eta} = \otimes_{v}'\bar{\eta}_{v} \in \bar{\mathcal{S}}_{\chi}$ is finite-dimensional.
\end{fact}

\begin{proof}
    We first note that every $\bar{\eta}_{v}$ is finite-dimensional. Since $\bar{\eta}_{v} = \bar{\eta}^{0}_{v}$ for almost all $v$, it then suffices to show that $\bar{\eta}^{0}_{v}$ is $1$-dimensional for $v\notin S'$. 
    
    Recall that for $v \notin S'$, the extension $K_{v}/F_{v}$ is unramified, $T$ splits over $K_{v}$, the $1$-cocycle $z_{v}$ takes values in $T(\mathcal{O}_{K_{v}})$, and $\tilde{T}_{z}(\mathcal{O}_{v})$ consists of the fixed points of the $z_{v}$-twisted Galois action on $(T \rtimes A)(\mathcal{O}_{K_{v}})$ (see \S\ref{section:model}). Because $T(\mathcal{O}_{K_{v}})$ is cohomologically trivial as a $\mathrm{Gal}(K_{v}/F_{v})$-module, there exists some $d \in T(\mathcal{O}_{K_{v}})$ such that $z_{v}(\sigma) = d^{-1}\sigma(d)$ for all $\sigma \in \mathrm{Gal}(K_{v}/F_{v})$. As introduced in Proposition~\ref{prop:genericlem}, this $d$ induces an isomorphism:
    \begin{align*}
    \gamma_{d}: T(F_{v})\rtimes A(F_{v})^{\chi_{v}} &\xrightarrow{\sim} \tilde{T}_{z}(F_{v})^{\chi_{v}}\\
    (t,a) &\mapsto (td^{-1}a(d), a).
    \end{align*}
    We define $\chi_{v}^{*}$ as the one-dimensional character on $\tilde{T}_{z}(F_{v})^{\chi_{v}}$ obtained by extending $\chi_{v}$ trivially across $\gamma_{d}$, meaning $\chi_{v}^{*}(\gamma_{d}(t,a)) = \chi_{v}(t)$.
    
    We now verify that $\chi_{v}^{*}$ is unramified. Notice that the map $\gamma_{d}$ is given by conjugation by $(d^{-1}, 1)$ within  $(T \rtimes A)(K_{v})$. Because $d \in T(\mathcal{O}_{K_{v}})$, this conjugation preserves the integral points $(T \rtimes A)(\mathcal{O}_{K_{v}})$. Moreover, because $z_{v}(\sigma) = d^{-1}\sigma(d)$, $\gamma_{d}$ intertwines the standard Galois action on $T \rtimes A$ with the $z_{v}$-twisted Galois action on $\tilde{T}_{z}$. After taking Galois invariants, $\gamma_{d}$ descends to an isomorphism of the $\mathcal{O}_{v}$-points:
    \begin{align*}
    \gamma_{d}: T(\mathcal{O}_{v}) \rtimes A(F_{v})^{\chi_{v}} \xrightarrow{\sim} \tilde{T}_{z}(\mathcal{O}_{v})^{\chi_{v}}.
    \end{align*}
    Consequently, any element in the compact open subgroup $\tilde{T}_{z}(\mathcal{O}_{v})^{\chi_{v}}$ can be written as $\gamma_{d}(t, a)$ for some $t \in T(\mathcal{O}_{v})$ and $a \in A(F_{v})^{\chi_{v}}$. Since $\chi_{v}$ is unramified, we have $\chi_{v}^{*}(\gamma_{d}(t, a)) = \chi_{v}(t) = 1$ for all $t \in T(\mathcal{O}_{v})$. This proves that $\chi_{v}^{*}$ is the unique unramified irreducible constituent of $\mathrm{Ind}_{T(F_{v})}^{\tilde{T}_{z}(F_{v})^{\chi_{v}}}\chi_{v}$. Therefore, $\bar{\eta}^{0}_{v}$ must coincide with $\chi_{v}^{*}$, and hence is $1$-dimensional.
\end{proof}

As in the local setting, transferring from $\tilde{T}_{z}(\mathbb{A})$ to $\tilde{T}_{z}(\mathbb{A})^{\chi}$ will facilitate the analysis. The following propositions are the global analogues of Claim~\ref{claim:groupside1} and Claim~\ref{claim:groupside2}.

\begin{prop} \label{globalinductionisotypic}
Let $\bar{\eta} = \otimes'_{v}\bar{\eta}_{v} \in \bar{\mathcal{S}}_{\chi}$. Then the restriction of $\bar{\eta}$ to $T(\mathbb{A})$ is $\chi$-isotypic. 
\end{prop}
\begin{proof}
According to Claim~\ref{claim:groupside1}, at each place $v$, $\bar{\eta}_{v}$ is $\chi_{v}$-isotypic after being restricted to $T(F_{v})$. Thus the statement follows. 
\end{proof}

\begin{prop}\label{globalinductionbij}
The induction functor $\mathrm{Ind}_{\tilde{T}_{z}(\mathbb{A})^{\chi}}^{\tilde{T}_{z}(\mathbb{A})}: \bar{\mathcal{S}}_{\chi} \to \mathcal{S}_{\chi}$ is bijective. Its inverse sends $\otimes'_{v}\eta_{v} \in \mathcal{S}_{\chi}$ to the unique irreducible constituent of $\mathrm{Res}_{\tilde{T}_{z}(\mathbb{A})^{\chi}}^{\tilde{T}_{z}(\mathbb{A})} (\otimes'_{v}\eta_{v})$ that lies in $\bar{\mathcal{S}}_{\chi}$.
\end{prop}
\begin{proof}
As shown in the proof of Fact~\ref{finitedimbar}, $\bar{\eta}_{v}$ are $1$-dimensional characters for almost all $v$. Therefore, the reasoning in the proof of Claim~\ref{claim: inductionsurjective} can be adapted to prove
\begin{align*}
\mathrm{Ind}_{\tilde{T}_{z}(\mathbb{A})^{\chi}}^{\tilde{T}_{z}(\mathbb{A})} \left(\otimes_{v}'\bar{\eta}_{v}\right) \cong \sideset{}{'}\bigotimes_{v}  \mathrm{Ind}_{\tilde{T}_{z}(F_{v})^{\chi_{v}}}^{\tilde{T}_{z}(F_{v})}\bar{\eta}_{v}.
\end{align*}
Claim~\ref{claim:groupside2} shows that, for each place $v$, the local induction 
$\mathrm{Ind}_{\tilde{T}_{z}(F_{v})^{\chi_{v}}}^{\tilde{T}_{z}(F_{v})}$ gives a bijection between $\mathrm{Irr}(\tilde{T}_{z}(F_{v})^{\chi_{v}}, \chi_{v})$ and $\mathrm{Irr}(\tilde{T}_{z}(F_{v}), \chi_{v})$. Therefore, in view of the above relation, the functor $\mathrm{Ind}_{\tilde{T}_{z}(\mathbb{A})^{\chi}}^{\tilde{T}_{z}(\mathbb{A})}$ is also bijective. 

For each place $v$, the inverse of $\mathrm{Ind}_{\tilde{T}_{z}(F_{v})^{\chi_{v}}}^{\tilde{T}_{z}(F_{v})}: \mathrm{Irr}(\tilde{T}_{z}(F_{v})^{\chi_{v}}, \chi_{v}) \to \mathrm{Irr}(\tilde{T}_{z}(F_{v}), \chi_{v})$ is given by sending $\eta_{v} \in \mathrm{Irr}(\tilde{T}_{z}(F_{v}), \chi_{v})$ to the unique irreducible constituent of $\mathrm{Res}_{\tilde{T}_{z}(F_{v})^{\chi_{v}}}^{\tilde{T}_{z}(F_{v})}\,\eta_{v}$ which lies in $\mathrm{Irr}(\tilde{T}_{z}(F_{v})^{\chi_{v}}, \chi_{v})$. Hence the analogous global statement follows. 
\end{proof}

\subsection{Multiplicity on the automorphic side}\label{section:mult-group}
In the proofs of Corollary~\ref{corollary: semisimplicity&decomp_1} and Corollary~\ref{corollary: semisimplicity&decomp_2}, we have seen that irreducible constituents of $U_{\chi} = \mathrm{Ind}_{T(\mathbb{A})\tilde{T}_{z}(F)^{\chi}}^{\tilde{T}_{z}(\mathbb{A})} \chi^{*}$ must lie in $\mathcal{S}_{\chi}$~\eqref{defn: S_chi}. Conversely, given $\eta = \otimes'_{v}\eta_{v}\in \mathcal{S}_{\chi}$, we hope to determine whether $\eta$ occurs in $U_{\chi}$, and more generally, with what multiplicity. 

For the rest of this section, we fix $\eta = \otimes'_{v}\eta_{v}\in \mathcal{S}_{\chi}$ and set
\begin{align*}
m_{\eta,\chi} :=  \mathrm{dim\ Hom}_{\tilde{T}_{z}(\mathbb{A})} (\eta, U_{\chi}).
\end{align*}
We warn the reader that there might be more than one $\tilde{T}_{z}(F)$-orbit of Hecke characters contributing to the multiplicity of $\eta$ in $\mathcal{A}(\tilde{T}_{z})$. If we denote
\begin{align*}
    m_{\eta} := \mathrm{dim\ Hom}_{\tilde{T}_{z}(\mathbb{A})} (\eta, \mathcal{A}(\tilde{T}_{z})),
\end{align*}
then in view of the decomposition~\eqref{equation: preliminary_decomp}, we actually have 
\begin{align*}
    m_{\eta} = \sum_{\chi\in H(T)/\tilde{T}_{z}(F)}m_{\eta,\chi}.
\end{align*}

To determine $m_{\eta,\chi}$ we pass to the intermediate subgroup $\tilde{T}_{z}(\mathbb{A})^{\chi}$, the stabilizer of $\chi$ in $\tilde{T}_{z}(\mathbb{A})$. Proposition~\ref{globalinductionbij} ensures that there exists a unique $\bar{\eta}\in \bar{\mathcal{S}}_{\chi}$ such that $\mathrm{Ind}_{\tilde{T}_{z}(\mathbb{A})^{\chi}}^{\tilde{T}_{z}(\mathbb{A})}\bar{\eta} = \eta$. Here, $\bar{\eta}$ is finite-dimensional by Fact~\ref{finitedimbar}. On the other hand, according to Proposition~\ref{prop: induction_in_steps}, we are able to rewrite $U_{\chi}$ 
via an induction in steps:
\begin{align}
    U_{\chi} = \mathrm{Ind}_{T(\mathbb{A})\tilde{T}_{z}(F)^{\chi}}^{\tilde{T}_{z}(\mathbb{A})} \chi^{*} = \mathrm{Ind}_{\tilde{T}_{z}(\mathbb{A})^{\chi}}^{\tilde{T}_{z}(\mathbb{A})} \left[\mathrm{Ind}_{T(\mathbb{A})\tilde{T}_{z}(F)^{\chi}}^{\tilde{T}_{z}(\mathbb{A})^{\chi}} \chi^{*}\right]. \label{indstep}
\end{align}
As we pointed out in Remark~\ref{remark:addchi}, the inner induction $\mathrm{Ind}_{T(\mathbb{A})\tilde{T}_{z}(F)^{\chi}}^{\tilde{T}_{z}(\mathbb{A})^{\chi}} \chi^{*}$ is semisimple with irreducible constituents from $\bar{\mathcal{S}}_{\chi}$. Write the inner induction as $\bigoplus_{i}\bar{\eta}_{i}^{\oplus m_{i}}$, where the $\bar{\eta}_{i}\in\bar{\mathcal{S}}_{\chi}$ are pairwise non-isomorphic and finite-dimensional (Fact~\ref{finitedimbar}), hence closed and unitary. Corollary~\ref{cor: admissibility} applies to the inner induction, with the ambient group $\tilde{T}_{z}(\mathbb{A})^{\chi} = \prod'_{v}\tilde{T}_{z}(F_{v})^{\chi_{v}}$ in place of $\tilde{T}_{z}(\mathbb{A})$ (the results of the smooth-induction preliminaries hold verbatim for this ambient group, which satisfies Assumption~\ref{assump:GH}): by part~(3), its $L^{2}$-completion is unitarily isomorphic to $\widehat{\bigoplus_{i}}\,\bar{\eta}_{i}^{\oplus m_{i}}$. Since unitary induction commutes with Hilbert direct sums, the $L^{2}$-completion $\overline{U}_{\chi}$ (see the proof of Proposition~\ref{prop: induction_in_steps}) is the Hilbert direct sum of closed subrepresentations isomorphic to $L^2\text{-}\mathrm{Ind}_{\tilde{T}_{z}(\mathbb{A})^{\chi}}^{\tilde{T}_{z}(\mathbb{A})}\bar{\eta}_{i}$, with $m_{i}$ copies each. Passing to smooth parts by Corollary~\ref{cor: admissibility}(2) (applied to $\overline{U}_{\chi}$) and using Definition~\ref{defn:smooth_ind}, we obtain
\begin{align*}
U_{\chi} = \bigoplus_{i}\left(\mathrm{Ind}_{\tilde{T}_{z}(\mathbb{A})^{\chi}}^{\tilde{T}_{z}(\mathbb{A})}\bar{\eta}_{i}\right)^{\oplus m_{i}}.
\end{align*}
By Proposition~\ref{globalinductionbij}, the inductions $\mathrm{Ind}_{\tilde{T}_{z}(\mathbb{A})^{\chi}}^{\tilde{T}_{z}(\mathbb{A})}\bar{\eta}_{i}$ are irreducible and pairwise non-isomorphic. Therefore, multiplicities are preserved by the induction functor $\mathrm{Ind}_{\tilde{T}_{z}(\mathbb{A})^{\chi}}^{\tilde{T}_{z}(\mathbb{A})}$, and we have:
\begin{align*}
m_{\eta,\chi} = \mathrm{dim\ Hom}_{\tilde{T}_{z}(\mathbb{A})^{\chi}}(\bar{\eta}, \mathrm{Ind}_{T(\mathbb{A})\tilde{T}_{z}(F)^{\chi}}^{\tilde{T}_{z}(\mathbb{A})^{\chi}} \chi^{*}).
\end{align*}
By Frobenius Reciprocity (Proposition~\ref{prop: frobenius_rec}), it can be further simplified as
\begin{align}
m_{\eta,\chi} =  \mathrm{dim\ Hom}_{T(\mathbb{A})\tilde{T}_{z}(F)^{\chi}}(\mathrm{Res}_{T(\mathbb{A})\tilde{T}_{z}(F)^{\chi}}^{{\tilde{T}_{z}(\mathbb{A})^{\chi}}}\bar{\eta},\chi^{*}). \label{simpstep2}
\end{align}

Next, we isolate the $\tilde{T}_{z}(F)^{\chi}$-action from the  $T(\mathbb{A})$-action. Indeed, by Proposition~\ref{globalinductionisotypic}, the restriction of $\bar{\eta}$ to $T(\mathbb{A})$ is $\chi$-isotypic, which agrees with $\chi^{*}$. Therefore, the multiplicity problem is reduced to analysing only the $\tilde{T}_{z}(F)^{\chi}$-action:
\begin{align}
m_{\eta,\chi} =  \mathrm{dim\ Hom}_{\tilde{T}_{z}(F)^{\chi}}(\mathrm{Res}_{\tilde{T}_{z}(F)^{\chi}}^{{\tilde{T}_{z}(\mathbb{A})^{\chi}}}\bar{\eta},\mathbb{1}). \label{simpstep3}
\end{align}

Finally, because the restriction of $\bar{\eta}$ to $T(\mathbb{A})$ is $\chi$-isotypic, its further restriction to $T(F)$ is trivial. Consequently, the representation $\mathrm{Res}_{\tilde{T}_{z}(F)^{\chi}}^{{\tilde{T}_{z}(\mathbb{A})^{\chi}}}\bar{\eta}$ descends to the quotient $\tilde{T}_{z}(F)^{\chi}/T(F) \cong A(F)^{[z],\chi}$. We simply denote the descended representation by $\bar{\eta}|_{A(F)^{[z],\chi}}$. More precisely, for each $a \in A(F)^{[z],\chi}$, we choose some $t_{a}\in\tilde{T}_{z}(F)^{\chi}$ lying above $a$ and define $\bar{\eta}|_{A(F)^{[z],\chi}}(a) := \bar{\eta}(t_{a})$, which does not depend on the choice of $t_{a}$. The multiplicity therefore becomes
\begin{align*}
m_{\eta,\chi} &= \mathrm{dim\ Hom}_{A(F)^{[z],\chi}}(\bar{\eta}|_{A(F)^{[z],\chi}}, \mathbb{1})\\
&= \frac{1}{|A(F)^{[z],\chi}|} \sum_{a\in A(F)^{[z],\chi}}\mathrm{tr}\,\bar{\eta}|_{A(F)^{[z],\chi}}(a).
\end{align*}

In summary, we have established:
\begin{prop}\label{prop: automult}
Let $\eta \in \mathcal{S}_{\chi}$ and let $\bar{\eta}\in \bar{\mathcal{S}}_{\chi}$ such that $\mathrm{Ind}_{\tilde{T}_{z}(\mathbb{A})^{\chi}}^{\tilde{T}_{z}(\mathbb{A})}\bar{\eta} = \eta$.  Then the multiplicity of $\eta$ in $U_{\chi}$ is given by
\begin{align*}
\mathrm{dim\ Hom}_{\tilde{T}_{z}(\mathbb{A})} (\eta, U_{\chi}) = \frac{1}{|A(F)^{[z],\chi}|} \sum_{a\in A(F)^{[z],\chi}}\mathrm{tr}\,\bar{\eta}|_{A(F)^{[z],\chi}}(a).
\end{align*}
\end{prop}

Let $\eta = \otimes'_{v}\eta_{v}$ be a smooth irreducible admissible representation of $\tilde{T}_{z}(\mathbb{A})$. By the definition of $\mathcal{S}_{\chi}$~\eqref{defn: S_chi}, for $\chi=\otimes_{v} \chi_{v} \in H(T)$, $\eta$ lies in $\mathcal{S}_{\chi}$ if and only if $\chi_{v}$ is a constituent of $\eta_{v}|_{T(F_{v})}$ for all places $v$. Such $\chi$ is generally far from unique (see~\ref{example: norm_torus}(i) for a concrete example). We let
\begin{align}\label{defn: X_eta}
\mathcal{X}(\eta) :=\{\chi\in H(T)\bigm | \eta \in \mathcal{S}_{\chi}\}. 
\end{align}

\begin{prop}\label{prop: X_eta_finite}
The set $\mathcal{X}(\eta)$ is finite.
\end{prop}
\begin{proof}
Let $S_{\eta}$ denote the union of $S$, the archimedean places, and the finitely many places where $\eta_{v}$ is not $\tilde{T}_{z}(\mathcal{O}_{v})$-unramified. For $v\notin S_{\eta}$, the $T(F_{v})$-translates of a nonzero $\tilde{T}_{z}(\mathcal{O}_{v})$-fixed vector span $\eta_{v}$ (recall $\tilde{T}_{z}(F_{v}) = T(F_{v})\tilde{T}_{z}(\mathcal{O}_{v})$ for $v\notin S$), and $T(\mathcal{O}_{v})$ fixes each translate because $T(F_{v})$ is commutative. Hence, $T(\mathcal{O}_{v})$ acts trivially on $\eta_{v}$, and every $\chi \in \mathcal{X}(\eta)$ is unramified outside $S_{\eta}$. Meanwhile, at each of the finitely many places $v\in S_{\eta}$, the component $\chi_{v}$ ranges over the finitely many irreducible constituents of the finite-dimensional representation $\eta_{v}|_{T(F_{v})}$. Now, two elements of $\mathcal{X}(\eta)$ with the same components at all $v\in S_{\eta}$ differ by a character of $T(\mathbb{A})$ trivial on the subgroup $A_{T}T(F)\prod_{v\in S_{\eta}}T(F_{v})\prod_{v\notin S_{\eta}}T(\mathcal{O}_{v})$, which is open in $T(\mathbb{A})$. Hence, the corresponding quotient of $T(\mathbb{A})$ is discrete, and it is compact as a quotient of $[T]$. It follows that there are only finitely many such characters and $\mathcal{X}(\eta)$ is finite.
\end{proof}

\begin{cor} \label{corollary: final_automult}
Let $\eta = \otimes'_{v}\eta_{v}$ be a smooth irreducible admissible representation of $\tilde{T}_{z}(\mathbb{A})$. The full multiplicity of $\eta$ in $\mathcal{A}(\tilde{T}_{z})$ is finite, and is given by the finite sum: 
\begin{align*}
\mathrm{dim\ Hom}_{\tilde{T}_{z}(\mathbb{A})} (\eta, \mathcal{A}(\tilde{T}_{z})) = \sum_{\chi \in \mathcal{X}(\eta)/\tilde{T}_{z}(F)} \frac{1}{|A(F)^{[z],\chi}|} \sum_{a\in A(F)^{[z],\chi}}\mathrm{tr}\,\bar{\eta}|_{A(F)^{[z],\chi}}(a).
\end{align*}
\end{cor}
\begin{proof}
According to the decomposition~\eqref{equation: preliminary_decomp}, we have 
$$\mathrm{dim\ Hom}_{\tilde{T}_{z}(\mathbb{A})} (\eta, \mathcal{A}(\tilde{T}_{z})) = \sum_{\chi \in H(T)/\tilde{T}_{z}(F)} \mathrm{dim\ Hom}_{\tilde{T}_{z}(\mathbb{A})} (\eta, U_{\chi}).$$
$U_{\chi}$ is a submodule of $\mathrm{Ind}_{T(\mathbb{A})}^{\tilde{T}_{z}(\mathbb{A})}\chi$, all of whose constituents lie in $\mathcal{S}_{\chi}$ by the proof of Corollary~\ref{corollary: semisimplicity&decomp_1}. Therefore, $\mathrm{Hom}_{\tilde{T}_{z}(\mathbb{A})} (\eta, U_{\chi}) \neq 0$ only if $\eta \in \mathcal{S}_{\chi}$ and it suffices to sum over the $\tilde{T}_{z}(F)$-orbits in $\mathcal{X}(\eta)$ in the above sum. The statement follows from Proposition~\ref{prop: automult}. Finiteness holds because the index set $\mathcal{X}(\eta)/\tilde{T}_{z}(F)$ is finite by Proposition~\ref{prop: X_eta_finite}, and each summand is finite.
\end{proof}

\section{The multiplicity formula for disconnected tori}

We retain the notations from the previous chapter. A pure inner form $\tilde{T}_{z}$ has been fixed.

\subsection{Settings on the dual side} \label{sec: settings_dual_side}
Let $W_{F}$ be the global Weil group of $F$. Let ${}^{L}T := \widehat{T}\rtimes W_{F}$ be the global $L$-group of $T$. A global $L$-parameter for $T$ is a continuous $L$-morphism $W_{F} \to {}^{L}T$. Two global $L$-parameters $\phi_{1},\phi_{2}: W_{F}\to {}^{L}T$ are said to be equivalent if they are conjugate under $\widehat{T}$. As in the local case, a global $L$-parameter can be regarded as a continuous $1$-cocycle $W_{F}\to \widehat{T}$ and the set of equivalence classes of global $L$-parameters is in 1-1 correspondence with the continuous cohomology group $H^{1}(W_{F},\widehat{T})$.

Let $v$ be a place of $F$. Given a global $L$-parameter $\phi \in Z^{1}(W_{F},\widehat{T})$, after composing it with the map $W_{F_{v}} \to W_{F}$ (see~\cite[p.8]{tate1979number}), we obtain a local $L$-parameter $\phi_{v}: W_{F_{v}}\to \widehat{T}$, which is called the local component of $\phi$ at $v$. 

The global Langlands correspondence states that there is a natural surjective morphism with finite kernel:
\begin{align} \label{equation: glc}
\mathfrak{L}: H^{1}(W_{F},\widehat{T})\twoheadrightarrow
\mathrm{Hom}_{\mathrm{cts}}(T(\mathbb{A})/T(F),\mathbb{C}^{\times}),
\end{align}
which is compatible with the local Langlands correspondence for $T$ at each place $v$. More precisely, by~\cite{langlands1997representations}, the kernel of  $\mathfrak{L}$~\eqref{equation: glc} coincides with
$$\mathrm{ker}\left[H^{1}(W_{F},\widehat{T})\to \prod_{v}H^{1}(W_{F_{v}}, \widehat{T})\right],$$
whose elements are called locally trivial classes. We note that two global $L$-parameters give rise to the same Hecke character if and only if they differ by a locally trivial class. This leads to a weaker equivalence relation on global $L$-parameters, termed \textit{near equivalence}:

\begin{defn}\label{defn: global_l_par_near_equiv}
Let $\phi_{1},\phi_{2}: W_{F} \to {}^{L}T$ be two global $L$-parameters. We say that $\phi_{1}$ is nearly equivalent to $\phi_{2}$ if, for each place $v$, the local components $(\phi_{1})_{v}$ and $(\phi_{2})_{v}$ are equivalent as local $L$-parameters.
\end{defn}

Under this notion, the global Langlands correspondence can be seen as a bijection between near-equivalence classes of global $L$-parameters for $T$ and Hecke characters of $T$. 

There is yet another description of $\mathrm{ker}\mathfrak{L}$. It is a well-known fact that 
\begin{align*}
\mathrm{ker}\left(H^{1}(F,T) \to H^{1}(\mathbb{A},T) \cong \bigoplus_{v}H^{1}(F_{v},T)\right)
\end{align*} 
is finite (cf.~\cite[p.306 Corollary]{platonov1993algebraic}), which we denote by $\Sha^{1}(F,T)$. Then $\mathrm{ker}\mathfrak{L}$ can be identified with the Pontryagin dual of 
$\Sha^{1}(F,T)$. When $\Sha^{1}(F,T)$ is trivial, we say that $T$ satisfies the Hasse principle. In this particular case, near equivalence coincides with equivalence, and the global Langlands correspondence becomes an isomorphism
\begin{align*}
H^{1}(W_{F},\widehat{T}) \xrightarrow{\sim}\mathrm{Hom}_{\mathrm{cts}}(T(\mathbb{A})/T(F),\mathbb{C}^{\times}).
\end{align*}

\subsection{Adelic $L$-packets}
For the remainder of this chapter, a ``global $L$‑parameter'' will always mean one whose associated Hecke character belongs to $H(T)$, the set of Hecke characters that are trivial on $A_{T}$. 

In the disconnected setting, we must further relax the notion of near equivalence by taking the action of $A$ into account (compare the local analogue in Definition~\ref{l-parameterequiv}). For the fixed pure inner form $\tilde{T}_{z}$, we make the following definition:

\begin{defn}\label{defn: global_l_par_near_AFz_equiv}
Let $\phi_{1},\phi_{2}: W_{F} \to {}^{L}T$ be two global $L$-parameters. We say that $\phi_{1}$ and $\phi_{2}$ are nearly $A(F)^{[z]}$-equivalent if there exists $a \in A(F)^{[z]}$ such that, for each place $v$, $(a\cdot \phi_{1})_{v}$ is equivalent to $(\phi_{2})_{v}$ as local $L$-parameters. In other words, we have $a\cdot \chi_{1} = \chi_{2}$, where $\chi_{i}$ is the Hecke character of $T$ determined by $\phi_{i}$ ($i = 1,2$). We denote the near-$A(F)^{[z]}$-equivalence class of $\phi$ as $[[\phi]]$.
\end{defn}

Let $\phi$ be a global $L$-parameter and let $\chi$ be its associated Hecke character. For each place $v$, the local Langlands correspondence for disconnected tori (Theorem~\ref{thm:llc_reinterthm}) provides a bijection
\begin{align*}
\iota_{v}: \Pi_{\phi_{v}}(\tilde{T}_{z_{v}}) \to \mathrm{Irr}(\pi_{0}(\tilde{S}_{\phi_{v}}^{[z_{v}]}), [z_{v}]).
\end{align*}

We define the adelic $L$-packet associated with $\phi$ as
\begin{align} \label{defn: adelic_L_packet}
\Pi_{\phi}(\tilde{T}_{z}) = \{\otimes'_{v}\;\eta_{v} \;\bigm|\; \eta_{v} \in \Pi_{\phi_{v}}(\tilde{T}_{z_{v}}), \;\iota_{v}(\eta_{v}) = \mathbb{1} \text{ for almost all\ } v\}.
\end{align}
We need to make the restricted tensor product in the above definition more precise. First, note that because $z$ is a global cocycle, its local cohomology class $[z_v] \in H^1(F_v, T)$ is trivial for almost all $v$. Thus, the trivial representation $\mathbb{1}$ does lie in $\mathrm{Irr}(\pi_0(\tilde{S}_{\phi_v}^{[z_v]}), [z_v])$ for such $v$.

Let $\otimes'_{v}\eta_{v} \in \Pi_{\phi}(\tilde{T}_{z})$. For each place $v$, by Claim~\ref{claim:groupside2}, $\eta_{v}$ can be written as $\eta_{v} = \mathrm{Ind}_{\tilde{T}_{z}(F_{v})^{\chi_{v}}}^{\tilde{T}_{z}(F_{v})}\bar{\eta}_{v}$ for a unique $\bar{\eta}_{v} \in \mathrm{Irr}(\tilde{T}_{z}(F_{v})^{\chi_{v}}, \chi_{v})$. Let $K$ be the finite Galois extension of $F$ and let $K_{v}$ be the completion of $K$ at a chosen place above $v$, both fixed in \S\ref{sec: adelic_points}. For almost all places $v$, the proof of Fact~\ref{finitedimbar} shows that there exists some $d_{v} \in T(\mathcal{O}_{K_{v}})$ that induces an isomorphism:
\begin{align*}
\gamma_{d_{v}}: T(F_{v})\rtimes A(F_{v})^{\chi_{v}} &\xrightarrow{\sim} \tilde{T}_{z}(F_{v})^{\chi_{v}}\\
(t,a) &\mapsto (td_{v}^{-1}a(d_{v}), a),
\end{align*}
which descends to an isomorphism of the $\mathcal{O}_{v}$-points:
\begin{align*}
\gamma_{d_{v}}: T(\mathcal{O}_{v}) \rtimes A(F_{v})^{\chi_{v}} \xrightarrow{\sim} \tilde{T}_{z}(\mathcal{O}_{v})^{\chi_{v}}.
\end{align*}

By the definition \eqref{defn: adelic_L_packet} of $\Pi_{\phi}(\tilde{T}_{z})$, for almost all places $v$, we have $\iota_{v}(\eta_{v}) = \mathbb{1}$. By Proposition~\ref{prop:genericlem}, this implies that $\bar{\eta}_{v}$ coincides with $\chi_{v}^{*}$, the character on $\tilde{T}_{z}(F_{v})^{\chi_{v}}$ obtained by extending $\chi_{v}$ trivially across $\gamma_{d_{v}}$, meaning $\chi_{v}^{*}(\gamma_{d_{v}}(t,a)) = \chi_{v}(t)$. Recall that any element in the compact open subgroup $\tilde{T}_{z}(\mathcal{O}_{v})^{\chi_{v}}$ can be written as $\gamma_{d_{v}}(t, a)$ for some $t \in T(\mathcal{O}_{v})$ and $a \in A(F_{v})^{\chi_{v}}$. Since $\chi_{v}$ is an unramified Hecke character, we have $\bar{\eta}_{v}(\gamma_{d_{v}}(t, a)) = \chi_{v}^{*}(\gamma_{d_{v}}(t, a)) = \chi_{v}(t) = 1$ for all $t \in T(\mathcal{O}_{v})$ and all $a \in A(F_{v})^{\chi_{v}}$. This proves that $\bar{\eta}_{v}$ is identically $1$ on $\tilde{T}_{z}(\mathcal{O}_{v})^{\chi_{v}}$, and hence is unramified for almost all $v$. 

Conversely, if an element $\otimes'_v \eta_v \in \mathcal{S}_\chi$ is given, then by definition $\eta_v$ is the unique unramified constituent $\eta_v^0$ for almost all $v$, and its corresponding $\bar{\eta}_v$ is the unramified character $\chi_v^*$. Indeed, since $\chi_{v}^{*} = 1$ on $\tilde{T}_{z}(\mathcal{O}_{v})^{\chi_{v}}$ and $\tilde{T}_{z}(F_{v}) = \tilde{T}_{z}(F_{v})^{\chi_{v}}\tilde{T}_{z}(\mathcal{O}_{v})$, the function taking value $1$ on $\tilde{T}_{z}(\mathcal{O}_{v})$ is a well-defined $\tilde{T}_{z}(\mathcal{O}_{v})$-fixed vector of $\mathrm{Ind}_{\tilde{T}_{z}(F_{v})^{\chi_{v}}}^{\tilde{T}_{z}(F_{v})}\chi_{v}^{*}$, so this induction equals $\eta_{v}^{0}$ by the uniqueness in Proposition~\ref{prop:local-unram}. By Proposition~\ref{prop:genericlem}, the trivial extension $\chi_v^*$ corresponds to the trivial representation $\mathbb{1}$ under $\iota_v$. In other words, $\iota_v(\eta_v) = \mathbb{1}$ for almost all $v$, which implies that $\otimes'_v \eta_v $ lies in the adelic $L$-packet.

As in the proof of Proposition~\ref{prop:local-unram}, for almost all places $v$, the space of $\tilde{T}_{z}(\mathcal{O}_{v})$-fixed vectors in $\eta_{v}=\mathrm{Ind}_{\tilde{T}_{z}(F_{v})^{\chi_{v}}}^{\tilde{T}_{z}(F_{v})}\bar{\eta}_{v}$ is one-dimensional and spanned by the distinguished vector $f_{v}^{0}$ that takes value $1$ on $\tilde{T}_{z}(\mathcal{O}_{v})$. The restricted tensor product in~\eqref{defn: adelic_L_packet} is taken with respect to these distinguished vectors $f_{v}^{0}$. In particular, for almost all $v$, the condition $\iota_{v}(\eta_{v})=\mathbb{1}$ is equivalent to $\eta_v$ being the unique unramified constituent $\eta_{v}^{0}$ appearing in Proposition~\ref{prop:local-unram} and the definition of $\mathcal{S}_{\chi}$~\eqref{defn: S_chi}. This matching implies that $\Pi_{\phi}(\tilde{T}_{z})$ and $\mathcal{S}_{\chi}$ coincide.

We record this fact: 

\begin{prop}\label{prop: adelicL-pktwelldef}
The adelic $L$-packet $\displaystyle \Pi_{\phi}(\tilde{T}_{z})$ coincides with $\displaystyle \mathcal{S}_{\chi}$. In particular, each element of $\displaystyle \Pi_{\phi}(\tilde{T}_{z})$ is a smooth irreducible admissible representation of $\displaystyle \tilde{T}_{z}(\mathbb{A})$.
\end{prop}

\begin{cor}\label{corollary: fall_in_adelic_Lpacket}
Any irreducible constituent of $\mathcal{A}(\tilde{T}_{z})$ lies in some adelic $L$-packet $\Pi_{\phi}(\tilde{T}_{z})$.
\end{cor}
\begin{proof}
By Corollary~\ref{corollary: semisimplicity&decomp_2}, any irreducible constituent of $\mathcal{A}(\tilde{T}_{z})$ lies in $\mathcal{S}_{\chi}$ for some $\chi \in H(T)$. The result then follows from Proposition~\ref{prop: adelicL-pktwelldef}.
\end{proof}

For a given smooth irreducible admissible $\tilde{T}_{z}(\mathbb{A})$-representation $\eta$, define 
\begin{align}\label{defn: Phi_eta}
\Phi(\eta) :=  \{\phi \text{ global }L \text{-parameter} \bigm | \eta \in \Pi_{\phi}(\tilde{T}_{z})\}.
\end{align}
We note that $\Phi(\eta)$ serves as the counterpart of $\mathcal{X}(\eta)$~\eqref{defn: X_eta} on the dual side.

Recall from Claim~\ref{claim:groupside2} that induction  gives a bijection between $\mathrm{Irr}(\tilde{T}_{z}(F_{v})^{\chi_{v}},\chi_{v})$ and $\Pi_{\phi_{v}}(\tilde{T}_{z_{v}})$. We let  $\bar{\iota}_{v} :=\iota_{v} \circ \mathrm{Ind}_{\tilde{T}_{z}(F_{v})^{\chi_{v}}}^{\tilde{T}_{z}(F_{v})}$, then for all places $v$, we have a bijection:
\begin{align*}
\bar{\iota}_{v}: \mathrm{Irr}(\tilde{T}_{z}(F_{v})^{\chi_{v}},\chi_{v}) \to \mathrm{Irr}(\pi_{0}(\tilde{S}_{\phi_{v}}^{[z_{v}]}), [z_{v}]).
\end{align*}
Analogous to~\eqref{defn: adelic_L_packet}, we can define a packet of $\tilde{T}_{z}(\mathbb{A})^{\chi}$-representations associated with $\phi$:
\begin{align} \label{defn: adelic_L_packet_chi_version}
\Pi_{\phi}(\tilde{T}_{z}^{\chi}) = \{\otimes'_{v}\bar{\eta}_{v} \;\bigm|\; \bar{\eta}_{v} \in \mathrm{Irr}(\tilde{T}_{z}(F_{v})^{\chi_{v}},\chi_{v}),\;\bar{\iota}_{v}(\bar{\eta}_{v}) = \mathbb{1} \text{ for almost all\ } v\}.
\end{align}
Here, the restricted tensor product is well defined because $\displaystyle\bar{\eta}_{v}$ is an unramified character for almost all $v$. Moreover, when comparing the construction of $\displaystyle\Pi_{\phi}(\tilde{T}_{z}^{\chi})$ with that of $\displaystyle\bar{\mathcal{S}}_{\chi}$~\eqref{defn: S_chi_bar}, we see that they coincide due to the same logic used in establishing Proposition~\ref{prop: adelicL-pktwelldef}.

\begin{prop}\label{prop: adelicL-pktwelldef_modified}
$\displaystyle \Pi_{\phi}(\tilde{T}_{z}^{\chi})$ coincides with $\displaystyle \bar{\mathcal{S}}_{\chi}$. In particular, each element of $\displaystyle \Pi_{\phi}(\tilde{T}_{z}^{\chi})$ is a smooth irreducible finite-dimensional $\displaystyle \tilde{T}_{z}(\mathbb{A})^{\chi}$-representation.
\end{prop}

\subsection{The pairing}
In this subsection, we construct a pairing that serves as the disconnected-tori counterpart to Kottwitz's multiplicity pairing~\eqref{eq:intro_mult}. Its purpose is to compute the automorphic multiplicity (Theorem~\ref{thm:finalthm}). We aim to show that the pairing is well-defined (Proposition~\ref{prop: pairing_trivial_aa}) and independent of the arbitrary local lift choices (Proposition~\ref{prop:pairing_welldef}).

Let $\phi$ be a global $L$-parameter and $\chi$ be the Hecke character determined by $\phi$. We hope to define a pairing
\begin{align*}
\langle\cdot,\cdot\rangle: A(F)^{[z],\chi} \times \Pi_{\phi}(\tilde{T}_{z}) &\to \mathbb{C}.
\end{align*}
Let $a \in A(F)^{[z],\chi}$. Clearly, $A(F)^{[z],\chi} \subseteq A(F_{v})^{[z_{v}], \chi_{v}}$ holds for any $v$. At any place $v$, there exists some $(s_{v},a) \in \tilde{S}_{\phi_{v}}^{[z_{v}]} = \mathrm{Cent}(\phi_{v},\widehat{T}\rtimes A(F_{v})^{[z_{v}]})$ lying above $a$. The inverse of $(s_{v},a)$ is $(a^{-1}(s_{v}^{-1}),a^{-1})\in\tilde{S}_{\phi_{v}}^{[z_{v}]}$, lying above $a^{-1}$. According to the short exact sequence~\eqref{SESglobalchi}, there exists some $(t,a) \in \tilde{T}_{z}(F)^{\chi}$ lying above $a$. This further implies $(t,a) \in \tilde{T}_{z}(F_{v})^{\chi_{v}}$ for all places $v$.

By Observation~\ref{obs:relation-groupside} and Observation~\ref{obs:relation-dualside}, we have
\begin{align*}
(\phi_{v}^{-1}, a^{-1}(s_{v}^{-1})) \in Z^{1}(W_{F_{v}}, \widehat{T}\xrightarrow{1-a^{-1}}\widehat{T})
\end{align*}
and 
\begin{align*}
(z_{v}^{-1}, t) \in Z^{1}(F_{v}, T\xrightarrow{1-a}T)
\end{align*}
for any $v$.

We define the pairing between $a \in A(F)^{[z],\chi}$ and $\eta = \otimes'_{v}\eta_{v} \in \Pi_{\phi}(\tilde{T}_{z})$ as
\begin{align}
\langle a,\eta \rangle := \prod_{v}\langle (z_{v}^{-1}, t), (\phi_{v}^{-1}, a^{-1}(s_{v}^{-1}))\rangle^{-1}_{\mathrm{TN}}\cdot \mathrm{tr}\left[\iota_{v}(\eta_{v})(s_{v},a)\right], \label{thepairing}
\end{align}
where the pairing $\langle \cdot,\cdot \rangle_{\mathrm{TN}}$ is the Tate--Nakayama pairing of hypercohomology~\eqref{TN-hyper-pairing}.  From now on, we denote $\rho_{v} := \iota_{v}(\eta_{v})$ for brevity.

In Proposition~\ref{prop: pairing_ch_eta} we ultimately prove, using the construction of the LLC, that $\langle -,\eta \rangle$ coincides with the character of the representation $\bar{\eta}|_{A(F)^{[z],\chi}}$ defined in \S\ref{section:mult-group}. In particular, it follows from Proposition~\ref{prop: pairing_ch_eta} that the pairing $\langle a,\eta \rangle$ is well-defined and does not depend on the choices of $(s_{v},a)$ or $(t,a)$. However, in this section, we give direct proofs of these facts that avoid any reliance on the construction of $\rho_{v}$ in terms of the LLC. More precisely, we use only that $\rho_{v}$ lies in $\mathrm{Irr}(\pi_{0}(\tilde{S}_{\phi_{v}}^{[z_{v}]}), [z_{v}])$ and is trivial for almost all $v$. 

\begin{prop}\label{prop: pairing_trivial_aa}
For almost all $v$, 
\begin{align*}
\langle (z_{v}^{-1}, t), (\phi_{v}^{-1}, a^{-1}(s_{v}^{-1}))\rangle_{\mathrm{TN}}^{-1}\cdot \mathrm{tr}\left[\rho_{v}(s_{v},a)\right] = 1.
\end{align*}
\end{prop}

\begin{proof}
    By the definition of the adelic $L$-packet $\Pi_{\phi}(\tilde{T}_{z})$, $\rho_{v} = \mathbb{1}$ holds for almost all $v$. 
    
    It remains to show that the first factor is trivial almost everywhere. Let $K$ be the finite Galois extension of $F$ fixed in \S\ref{sec: adelic_points}, chosen sufficiently large so that $\tilde{T}_{z}$ splits and $z$ factors through $\mathrm{Gal}(K/F)$. It is clear that almost all places $v$ are unramified in the extension $K/F$, and for almost all $v$, the image of $z_{v}^{-1} \in Z^{1}(K_{v}/F_{v}, T(K_{v}))$ lies in $T(\mathcal{O}_{K_{v}})$. Here, as in \S\ref{sec: adelic_points}, $K_{v}$ denotes the completion of $K$ at a chosen place above $v$. As established in the proof of Fact~\ref{finitedimbar}, the fact that $T(\mathcal{O}_{K_{v}})$ is cohomologically trivial as a $\mathrm{Gal}(K_{v}/F_{v})$-module implies the existence of some $d_{v} \in T(\mathcal{O}_{K_{v}})$ such that $z_{v}(\sigma_{v}) = d_{v}^{-1}\sigma_{v}(d_{v})$ for all $\sigma_{v} \in \mathrm{Gal}(K_{v}/F_{v})$. This $d_{v}$ induces the isomorphism of $\mathcal{O}_{v}$-points:
    \begin{align*}
    \gamma_{d_{v}}: T(\mathcal{O}_{v}) \rtimes A(F_{v})^{\chi_{v}} &\xrightarrow{\sim} \tilde{T}_{z}(\mathcal{O}_{v})^{\chi_{v}}\\
    (t',a) &\mapsto (t'd_{v}^{-1}a(d_{v}), a).
    \end{align*}
    Since $(t,a) \in \tilde{T}_{z}(F)^{\chi}$, we have $(t,a)\in \tilde{T}_{z}(\mathcal{O}_{v})^{\chi_{v}}$ for almost all $v$. In view of the isomorphism $\gamma_{d_{v}}$, there exists a unique $t' \in T(\mathcal{O}_{v})$ such that $(t,a) = \gamma_{d_{v}}(t',a) = (t'd_{v}^{-1}a(d_{v}), a)$. In particular, this gives $t = t'd_{v}^{-1}a(d_{v})$.
    
    On the other hand, notice that $z_{v}^{-1}(\sigma_{v}) = d_{v}\sigma_{v}(d_{v}^{-1})$. In the hypercohomology complex $T \xrightarrow{1-a} T$, the element $d_{v}^{-1}$ yields the $1$-hypercoboundary $(\partial(d_{v}^{-1}), (1-a)d_{v}^{-1}) = (z_{v}^{-1}, d_{v}^{-1}a(d_{v}))$. Hence, the hypercocycle $(z_{v}^{-1}, t) = (z_{v}^{-1}, t'd_{v}^{-1}a(d_{v}))$ decomposes as the product of $(1, t')$ and the hypercoboundary $(z_{v}^{-1}, d_{v}^{-1}a(d_{v}))$. Thus, $(z_{v}^{-1}, t)$ is cohomologous to $(1, t')$, meaning we can replace $(z_{v}^{-1}, t)$ by $(1, t')$ without changing the value of the Tate--Nakayama pairing. 
    
    Due to the compatibility stated in Proposition~\ref{TN-LLCcompatibility}, we have
    \begin{align*}
    \langle (z_{v}^{-1}, t), (\phi_{v}^{-1}, a^{-1}(s_{v}^{-1}))\rangle_{\mathrm{TN}} &= \langle (1, t'), (\phi_{v}^{-1}, a^{-1}(s_{v}^{-1}))\rangle_{\mathrm{TN}}\\
    &= [\phi_{v}]^{-1}(t')\\
    &= \chi_{v}(t')^{-1}.
    \end{align*}
    Because $t' \in T(\mathcal{O}_{v})$ and $\chi_{v}$ is unramified for almost all $v$, we conclude that $\chi_{v}(t') = 1$. Hence $\langle (z_{v}^{-1}, t), (\phi_{v}^{-1}, a^{-1}(s_{v}^{-1}))\rangle_{\mathrm{TN}} = 1$ for almost all $v$.
\end{proof}

\begin{prop}\label{prop:pairing_welldef}
$\langle a,\eta \rangle$ is independent of the auxiliary choices of $(s_{v}, a)\in \tilde{S}_{\phi_{v}}^{[z_{v}]}$ for each $v$ as well as the choice of $(t,a)\in \tilde{T}_{z}(F)^{\chi}$.
\end{prop}
\begin{proof}
Let $(s'_{v},a)\in \tilde{S}_{\phi_{v}}^{[z_{v}]}$ and $(t', a)\in \tilde{T}_{z}(F)^{\chi}$ be any other choices. Then we have
\begin{align*}
s'_{v} = y_{v}s_{v}
\end{align*}
and
\begin{align*}
t' = t_{0}t,
\end{align*}
for some $y_{v}\in \widehat{T}^{\,\Gamma_{v}}$ and $t_{0} \in T(F)$. Using Proposition~\ref{TN-LLCcompatibility} and the fact that $\rho_{v}$ restricts to $[z_{v}]$ on $\pi_{0}(\widehat{T}^{\,\Gamma_{v}})$, we have
\begin{align*}
{}&\langle (z_{v}^{-1}, t'), (\phi_{v}^{-1}, a^{-1}({s'_{v}}^{-1}))\rangle_{\mathrm{TN}}^{-1}\cdot \mathrm{tr}\left[\rho_{v}(s'_{v},a)\right]\\ 
=\ &\langle (z_{v}^{-1}, t) \cdot (1,t_{0}), (\phi_{v}^{-1},a^{-1}(s_{v}^{-1})) \cdot (1,a^{-1}(y_{v}^{-1}))\rangle_{\mathrm{TN}}^{-1}\cdot \mathrm{tr}\left[\rho_{v}(y_{v}, 1)\rho_{v}(s_{v},a)\right]\\
=\ &\langle (z_{v}^{-1}, t), (\phi_{v}^{-1}, a^{-1}(s_{v}^{-1}))\rangle_{\mathrm{TN}}^{-1}\cdot [\phi_{v}](t_{0})\cdot \cancel{[z_{v}](a^{-1}(y_{v}^{-1}))}\cdot \cancel{[z_{v}](y_{v})}\mathrm{tr}\left[\rho_{v}(s_{v},a)\right]\\
=\ &\langle (z_{v}^{-1}, t), (\phi_{v}^{-1}, a^{-1}(s_{v}^{-1}))\rangle_{\mathrm{TN}}^{-1}\cdot \chi_{v}(t_{0})\cdot \mathrm{tr}\left[\rho_{v}(s_{v},a)\right].
\end{align*}
The cancellation indicated above uses the fact that $a$ fixes $[z_{v}]$. Finally, $t_{0}\in T(F)$ implies  $\prod_{v}\chi_{v}(t_{0})= \chi(t_{0}) = 1$, hence we have shown
\begin{align*}
&\prod_{v}\langle (z_{v}^{-1}, t'), (\phi_{v}^{-1}, a^{-1}({s'_{v}}^{-1}))\rangle_{\mathrm{TN}}^{-1}\cdot \mathrm{tr}\left[\rho_{v}(s'_{v},a)\right] \\
= &\prod_{v}\langle (z_{v}^{-1}, t), (\phi_{v}^{-1}, a^{-1}(s_{v}^{-1}))\rangle_{\mathrm{TN}}^{-1}\cdot \mathrm{tr}\left[\rho_{v}(s_{v},a)\right].
\end{align*}
\end{proof}

\subsection{The multiplicity formula}
The pairing $\langle \cdot,\cdot \rangle: A(F)^{[z],\chi} \times \Pi_{\phi}(\tilde{T}_{z}) \to \mathbb{C}$ defined in the previous section has an incarnation on the automorphic side. Let $\eta = \otimes'_{v} \eta_{v}\in \Pi_{\phi}(\tilde{T}_{z}) = \mathcal{S}_{\chi}$. Let $\bar{\eta} \in \Pi_{\phi}(\tilde{T}_{z}^{\chi}) = \bar{\mathcal{S}}_{\chi}$ such that $\mathrm{Ind}_{\tilde{T}_{z}(\mathbb{A})^{\chi}}^{\tilde{T}_{z}(\mathbb{A})}\bar{\eta} = \eta$. By Proposition~\ref{prop: adelicL-pktwelldef_modified}, $\bar{\eta}$ is finite-dimensional. Moreover, in \S\ref{section:mult-group}, we have shown that $\bar{\eta}$ descends to $A(F)^{[z],\chi}$ and denote by $\bar{\eta}|_{A(F)^{[z],\chi}}$ the descended representation.

\begin{prop} \label{prop: pairing_ch_eta}
The function on $A(F)^{[z],\chi}$ defined as $a \mapsto \langle a, \eta\rangle$ coincides with the character of $\bar{\eta}|_{A(F)^{[z],\chi}}$.
\end{prop}
\begin{proof}
Denote $\rho_{v} := \iota_{v}(\eta_{v})$. We note that $(s_{v},a)^{-1} = (a^{-1}(s_{v}^{-1}),a^{-1})$ and rewrite:
\begin{align*}
\langle a,\eta \rangle &= \prod_{v}\langle (z_{v}^{-1}, t), (\phi_{v}^{-1}, a^{-1}(s_{v}^{-1}))\rangle^{-1}_{\mathrm{TN}}\cdot \mathrm{tr}\left[\rho_{v}(s_{v},a)\right]\\
&= \prod_{v}\langle (z_{v}^{-1}, t), (\phi_{v}^{-1}, a^{-1}(s_{v}^{-1}))\rangle^{-1}_{\mathrm{TN}}\cdot \mathrm{tr}\left[\rho_{v}(a^{-1}(s_{v}^{-1}), a^{-1})^{-1} \right].
\end{align*}
On the other hand, according to the intrinsic construction of the LLC which is characterized by the relation~\eqref{reinterrelation}, we have
\begin{align*}
\langle (z_{v}^{-1}, t), (\phi_{v}^{-1}, a^{-1}(s_{v}^{-1}))\rangle^{-1}_{\mathrm{TN}}\cdot \rho_{v}(a^{-1}(s_{v}^{-1}), a^{-1})^{-1} = \bar{\eta}_{v}(t,a).
\end{align*}
Thus we conclude
\begin{align*}
\langle a, \eta \rangle = \prod_{v}\mathrm{tr}\left[\bar{\eta}_{v}(t,a)\right] = \mathrm{tr}\left[\otimes'_{v}\bar{\eta}_{v}(t,a)\right] = \mathrm{tr}\left[\bar{\eta}|_{A(F)^{[z],\chi}}(a)\right].
\end{align*}
\end{proof}

Combining Proposition~\ref{prop: pairing_ch_eta} with Proposition~\ref{prop: automult}, Proposition~\ref{prop: X_eta_finite}, Corollary~\ref{corollary: final_automult} and Corollary~\ref{corollary: fall_in_adelic_Lpacket}, together with the definition of $\Phi(\eta)$~\eqref{defn: Phi_eta}, we arrive at the main theorem:
\begin{thm}\label{thm:finalthm}
Let $\chi$ be the Hecke character associated to $\phi$. For any $\eta \in \Pi_{\phi}(\tilde{T}_{z})$ (not necessarily automorphic), the contribution of the near-$A(F)^{[z]}$-equivalence class $[[\phi]]$ to the multiplicity of $\eta$ is
\begin{align*}
     m_{\eta,\phi} = \frac{1}{|A(F)^{[z],\chi}|}\sum_{a\in A(F)^{[z],\chi}}\langle a,\eta\rangle.
\end{align*}
Let $\eta$ be a smooth irreducible admissible representation of $\tilde{T}_{z}(\mathbb{A})$. Then the multiplicity of $\eta$ in $\mathcal{A}(\tilde{T}_{z})$ is the sum of all the $[[\phi]]$-contributions:
\begin{align*}
    m_{\eta} = \sum_{[[\phi]]:\phi\in \Phi(\eta)}m_{\eta,\phi}, 
\end{align*}
where $[[\phi]]$ runs over the near-$A(F)^{[z]}$-equivalence classes in $\Phi(\eta)$. The sum is finite and $m_{\eta}<\infty$. Consequently, we have the Hilbert-direct-sum decomposition:
\begin{align*}
L^{2}(A_{T}\tilde{T}_{z}(F)\backslash \tilde{T}_{z}(\mathbb{A})) = \widehat{\bigoplus_{\eta}}\ \eta^{\oplus m_{\eta}}.
\end{align*}
\end{thm}

\subsection{Comparison with the connected case}\label{section:kottwitz_analogue}
It is illuminating to compare Theorem~\ref{thm:finalthm} with the conjectural formula~\eqref{eq:intro_mult} for connected groups. 

Our notion of near-$A(F)^{[z]}$-equivalence in Definition~\ref{defn: global_l_par_near_AFz_equiv} is the analogue of the equivalence of global $L$-parameters defined by Kottwitz in~\cite[\S10.4]{kottwitz1984stable}. Following Kottwitz's definition there, in the setting of disconnected tori, we may say that two global $L$-parameters $\phi_1$ and $\phi_2$ are equivalent in Kottwitz's sense if there exists an element $(t, a) \in \widehat{T} \rtimes A(F)^{[z]}$ and a locally trivial $1$-cocycle $c \in Z^1(W_F, \widehat{T})$ such that
\begin{align*}
\phi_2 = (\mathrm{Int}(t, a) \circ \phi_1) \cdot c.
\end{align*}
By applying $\mathrm{Int}(t, a) = \mathrm{Int}(t) \circ \mathrm{Int}(1, a)$ and $\mathrm{Int}(1, a) \circ \phi_1 = a\cdot \phi_1$, we find $\mathrm{Int}(t, a) \circ \phi_1(w) = \mathrm{Int}(t)(a \cdot \phi_1)(w) = t (a \cdot \phi_1(w)) t^{-1} = t w(t)^{-1} (a \cdot \phi_1(w))$. We denote the $1$-coboundary $w \mapsto t w(t)^{-1}$ as  $\partial(t^{-1}) \in B^1(W_F, \widehat{T})$ and let $c' = \partial(t^{-1}) \cdot c$. Then $c$ is a locally trivial $1$-cocycle in $Z^1(W_F, \widehat{T})$ if and only if $c'$ is. Therefore, the equivalence in Kottwitz's sense exactly states that there exists a locally trivial $1$-cocycle $c' \in Z^1(W_F, \widehat{T})$ such that
\begin{align*}
\phi_2 = (a \cdot \phi_1) \cdot c'.
\end{align*}
This holds if and only if the local components $(\phi_2)_v$ and $(a \cdot \phi_1)_v$ are equivalent at all places $v$. This coincides with our definition of near-$A(F)^{[z]}$-equivalence in Definition~\ref{defn: global_l_par_near_AFz_equiv}.

Moreover, we will see that $A(F)^{[z],\chi}$ is the analogue of the finite group $\mathcal{S}_{\phi}$ in~\eqref{eq:intro_mult}. In the setting of connected reductive groups, we follow Kaletha's formulation in~\cite[\S4.5]{kaletha2018global} and recall that Kaletha considers the short exact sequence
\begin{align*}
1 \to Z(\widehat{G}^{*}) \to \widehat{G}^{*} \to \widehat{G}^{*}_{\mathrm{ad}} \to 1
\end{align*}
on which the Langlands group $L_{F}$ acts by $\mathrm{Ad}\circ \phi$. Consider the boundary map $\mathrm{Cent}(\phi, \widehat{G}^{*}_{\mathrm{ad}}) \to H^{1}(L_{F}, Z(\widehat{G}^{*}))$ and let $S_{\phi}^{\mathrm{ad}}$ be the subgroup of $\mathrm{Cent}(\phi, \widehat{G}^{*}_{\mathrm{ad}})$ consisting of elements whose image in $H^{1}(L_{F}, Z(\widehat{G}^{*}))$ is locally trivial. Then $\mathcal{S}_{\phi} = \pi_{0}(S_{\phi}^{\mathrm{ad}})$ is the finite group appearing in~\eqref{eq:intro_mult}. 

Now in the setting of disconnected tori, for the fixed global $L$-parameter $\phi:W_{F}\to {}^{L}T$ which determines the Hecke character $\chi$, we consider the short exact sequence
\begin{align*}
1 \to \widehat{T} \to \widehat{T} \rtimes A(F)^{[z]} \to A(F)^{[z]} \to 1
\end{align*}
on which the global Weil group $W_{F}$ acts by $\mathrm{Ad}\circ \phi$. To determine the boundary map $\delta: A(F)^{[z]} \to H^1(W_F, \widehat{T})$ induced by the above short exact sequence, we lift $a \in A(F)^{[z]}$ to $(1, a) \in \widehat{T} \rtimes A(F)^{[z]} $ and compute $(1,a)^{-1}\phi(w)(1,a)\phi(w)^{-1} = a^{-1}(\phi_{0}(w))\phi_{0}(w)^{-1}$, where we write $\phi(w) = (\phi_{0}(w), w)$. Therefore, the boundary map $\delta$ sends $a$ to $a^{-1}([\phi])[\phi]^{-1}$, where we regard $[\phi]$ as an element in $H^1(W_F, \widehat{T})$. We note that $a^{-1}([\phi])[\phi]^{-1}$ is locally trivial if and only if $a^{-1}(\chi)\chi^{-1}$ is trivial, i.e. $a\in A(F)^{[z],\chi}$. In other words, $A(F)^{[z],\chi}$ is precisely the analogue of $\mathcal{S}_{\phi}$ in our setting.

\subsection{Simplification for tori satisfying the Hasse principle}\label{subsec: hasse_simp}
In this section, we suppose that $T$ satisfies the Hasse principle. In this particular case, the pairing~\eqref{thepairing} can be significantly simplified. We recall from \S\ref{sec: settings_dual_side} that, when $T$ satisfies the Hasse principle, $H^{1}(W_{F},\widehat{T}) \xrightarrow{\sim} \mathrm{Hom}_{\mathrm{cts}}(T(\mathbb{A})/T(F),\mathbb{C}^{\times})$ is an isomorphism. In other words, the equivalence classes of $L$-parameters coincide with the near equivalence classes of $L$-parameters (see Definition~\ref{defn: global_l_par_near_equiv}). Consequently, we find that $A(F)^{[z],\chi} = A(F)^{[\phi],[z]}$ and $\tilde{T}_{z}(F)^{\chi} = \tilde{T}_{z}(F)^{[\phi]}$. 

Define $\tilde{S}_{\phi}^{[z]} := \mathrm{Cent}(\phi, \widehat{T}\rtimes A(F)^{[z]})$. We have a natural inclusion $\tilde{S}_{\phi}^{[z]} \subseteq \tilde{S}_{\phi_{v}}^{[z_{v}]}$ for all places $v$. Moreover, it is clear that $\tilde{S}_{\phi}^{[z]}$ and $\pi_{0}(\tilde{S}_{\phi}^{[z]})$ sit in short exact sequences similar to~\eqref{SES: dual1} and~\eqref{SES: dual2}.

For each $a \in A(F)^{[z],\chi} = A(F)^{[\phi],[z]}$, choose a lift $(s,a) \in \tilde{S}_{\phi}^{[z]}$ lying above $a$. Because $\tilde{S}_{\phi}^{[z]} \subseteq \tilde{S}_{\phi_{v}}^{[z_{v}]}$, the inverse $(a^{-1}(s^{-1}), a^{-1})$ belongs to $\tilde{S}_{\phi_{v}}^{[z_{v}]}$ for each $v$. In addition, we choose $(t,a)\in \tilde{T}_{z}(F)^{\chi} = \tilde{T}_{z}(F)^{[\phi]}$. To summarize, the following data will be used in computing $\langle a, \eta\rangle$:
\begin{align*}
(\phi_{v}^{-1}, a^{-1}(s^{-1})) \in Z^{1}(W_{F_{v}}, \widehat{T}\xrightarrow{1-a^{-1}}\widehat{T})
\end{align*}
for each $v$ and 
\begin{align*}
(z^{-1}, t) \in Z^{1}(F, T\xrightarrow{1-a}T).
\end{align*}

In view of the long exact sequence (see~\cite[Appendix~C.1]{kottwitz1999foundations})  
\begin{align*}
\cdots \to H^{1}(F,T \xrightarrow{1-a} T)  \to H^{1}(\mathbb{A},T \xrightarrow{1-a} T) \to H^{1}(\mathbb{A}/F,T \xrightarrow{1-a} T) \to \cdots,
\end{align*} 
the image of $[(z^{-1},t)]$ in $H^{1}(\mathbb{A}/F,T \xrightarrow{1-a} T)$ is the trivial class. By the compatibility between the local and global Tate--Nakayama dualities for hypercohomology (Theorem~\ref{thm: TN-duality-Global}), we have
\begin{align*}
&\prod_{v}\langle (z_{v}^{-1}, t), (\phi_{v}^{-1}, a^{-1}(s^{-1}))\rangle^{-1}_{\mathrm{TN,local}} \\
= &\langle (z^{-1}, t), (\phi^{-1}, a^{-1}(s^{-1}))\rangle^{-1}_{\mathrm{TN,global}} \\
= &\langle 1, (\phi^{-1}, a^{-1}(s^{-1}))\rangle^{-1}_{\mathrm{TN,global}} \\
= &1.
\end{align*}

On the other hand, due to the fact that $\iota_{v}(\eta_{v})$ is trivial for almost all $v$, the tensor product $\bigotimes_{v}\iota_{v}(\eta_{v})$ is a well-defined finite-dimensional representation of $\prod_{v}\pi_{0}(\tilde{S}_{\phi_{v}}^{[z_{v}]})$. We denote by the same symbol its pullback along the natural diagonal map $\pi_{0}(\tilde{S}_{\phi}^{[z]}) \to \prod_{v}\pi_{0}(\tilde{S}_{\phi_{v}}^{[z_{v}]})$. 

Finally, the pairing 
\begin{align*}
\langle\cdot,\cdot\rangle: A(F)^{[z],\chi} \times \Pi_{\phi}(\tilde{T}_{z}) &\to \mathbb{C}
\end{align*}
can now be simplified into
\begin{align}
    \langle a,\eta \rangle &= \prod_{v}\langle (z_{v}^{-1}, t), (\phi_{v}^{-1}, a^{-1}(s^{-1}))\rangle^{-1}_{\mathrm{TN,local}}\cdot \mathrm{tr}\left[\iota_{v}(\eta_{v})(s,a)\right]\nonumber\\
    &= \prod_{v}\mathrm{tr}\left[\iota_{v}(\eta_{v})(s,a)\right] \label{equation: simp_pairing_Hasse}\\ 
    &= \mathrm{tr}\left[\bigotimes_{v} \iota_{v}(\eta_{v})(s,a) \right],\nonumber
\end{align}
where $\bigotimes_{v} \iota_{v}(\eta_{v})$ is a finite-dimensional representation of $\pi_{0}(\tilde{S}_{\phi}^{[z]})$ (defined in the previous paragraph). 

It follows from Theorem~\ref{thm:finalthm} that the multiplicity of $\eta$ in $\mathcal{A}(\tilde{T}_{z})$ is 
\begin{align*}
    m_{\eta} = \sum_{[[\phi]]:\phi\in \Phi(\eta)} \frac{1}{|A(F)^{[z],\chi}|} \sum_{a\in A(F)^{[z],\chi}} \mathrm{tr}\left[\bigotimes_{v} \iota_{v}(\eta_{v})(s,a) \right].
\end{align*}

\subsection{Examples}
\subsubsection{Quasi-split form} \label{eg:first-example}
We start with the example of the \textit{quasi-split} form. For simplicity, assume that $T$ is an anisotropic torus over a number field $F$ and $A$ is a constant finite group scheme over $F$. Assume the action of $A$ on $T$ is defined over $F$ as well. Let $\tilde{T} := T\rtimes A$. Let $\phi$ be a global $L$-parameter of $T$ and let $\chi = \otimes_{v} \chi_{v}$ be its associated Hecke character. For each place $v$, by Claim~\ref{claim:groupside2}, induction gives a bijection between $\mathrm{Irr}(\tilde{T}(F_{v})^{\chi_{v}}, \chi_{v})$ and $\mathrm{Irr}(\tilde{T}(F_{v}), \chi_{v})$. Since $\tilde{T}(F_{v})^{\chi_{v}} = T(F_{v})\rtimes A(F_{v})^{\chi_{v}}$, we have 
$$\mathrm{Irr}(\tilde{T}(F_{v})^{\chi_{v}}, \chi_{v}) = \{\chi_{v}\boxtimes \xi_{v}\bigm | \xi_{v}\in\mathrm{Irr}(A(F_{v})^{\chi_{v}})\}.$$
On the dual side, we denote by $p_{v}$ the projection $\pi_{0}(\tilde{S}_{\phi_{v}}) \twoheadrightarrow A(F_{v})^{\chi_{v}}$. Under the LLC for $\tilde{T}$, $\chi_{v}\boxtimes \xi_{v} \in \mathrm{Irr}(\tilde{T}(F_{v})^{\chi_{v}}, \chi_{v})$ corresponds to the composition $\xi_{v}\circ p_{v}$. Let $\eta = \otimes'_{v} \mathrm{Ind}_{\tilde{T}(F_{v})^{\chi_{v}}}^{\tilde{T}(F_{v})}(\chi_{v}\boxtimes \xi_{v}) \in \Pi_{\phi}(\tilde{T})$. According to the definition of $\Pi_{\phi}(\tilde{T})$, $\xi_{v} = \mathbb{1}$ for all but finitely many places $v$. For the above $\eta$ and $a\in A(F)^{\chi}$, we choose the data $(1,a)\in \tilde{T}(F)^{\chi}$ and $(a^{-1}(s_{v}^{-1}),a^{-1}) \in \tilde{S}_{\phi_{v}}^{[z_{v}]}$, then
\begin{align*}
\langle a,\eta \rangle = \prod_{v} \langle (1,1), (\phi_{v}^{-1}, a^{-1}(s_{v}^{-1})) \rangle_{\mathrm{TN}}^{-1} \cdot \mathrm{tr}\,\xi_{v}(a)= \prod_{v}\mathrm{tr}\,\xi_{v}(a) = \mathrm{tr}\left[\bigotimes_{v} \xi_{v}(a) \right].
\end{align*}
Here, the second equality uses that $\langle (1,1), (\phi_{v}^{-1}, a^{-1}(s_{v}^{-1})) \rangle_{\mathrm{TN}} = [\phi_{v}^{-1}](1) = 1$, by Proposition~\ref{TN-LLCcompatibility}. Furthermore, the multiplicity of $\eta$ contributed by  $[[\phi]]$ is given by
\begin{align*}
m_{\eta,\phi} = \frac{1}{|A(F)^{\chi}|}\sum_{a\in A(F)^{\chi}} \mathrm{tr}\left[\bigotimes_{v} \xi_{v}(a) \right].
\end{align*}

\subsubsection{Norm torus} \label{example: norm_torus}
We conclude with a rank-1 example. Let $E/F$ be a quadratic extension of number fields. Let $T$ be the (anisotropic) norm torus (defined over $F$) associated with $E$. For $T$, the Hasse principle holds. Let the constant group scheme $A = \mathbb{Z}/2\mathbb{Z}$ act on $T$ by inverting and define $\tilde{T} := T\rtimes A$. Let $z \in Z^{1}(F,T)$ be a $1$-cocycle such that the cohomology class $[z]\in H^{1}(F,T)$ is nontrivial and of order $2$. We now illustrate several applications of the multiplicity formula to $\tilde{T}_{z}$.

We note that $A(F)^{[z]} = A(F)$ because $[z]$ has order $2$ by assumption. Define a finite set $\mathcal{S}:= \{v \bigm| [z_{v}] \text{ is non-trivial}\}$. For any $v \notin \mathcal{S}$, $\tilde{T}_{z}(F_{v})$ is isomorphic to the semi-direct product $T(F_{v})\rtimes A(F_{v})$. Let $\phi$ be a global $L$-parameter of $T$ and let $\chi = \otimes_{v} \chi_{v}$ be its associated Hecke character. We divide our discussion into two cases, according to whether $A(F)$ fixes $\chi$ (equivalent to $\chi^{2} = 1$).

\begin{enumerate}
    \item 
Let $\phi$ be a global $L$-parameter of $T$ that gives rise to a Hecke character $\chi$ satisfying $\chi^{2} \neq 1$. Then we have $A(F)^{[z],\chi} = 1$. Let  $\eta = \otimes'_{v}\eta_{v} \in \Pi_{\phi}(\tilde{T}_{z})$, where $\eta_{v}$ is an irreducible constituent of $\mathrm{Ind}_{T(F_{v})}^{\tilde{T}_{z}(F_{v})}\chi_{v}$. 
\begin{enumerate}
    \item For the places $v$ such that $\chi_{v}^{2} \neq 1$, we have $A(F_{v})^{[z_{v}],\chi_{v}} = 1$, which implies that the local $L$-packet $\Pi_{\phi_{v}}(\tilde{T}_{z_{v}})$ is the singleton $\{\eta_{v} = \mathrm{Ind}_{T(F_{v})}^{\tilde{T}_{z}(F_{v})}\chi_{v}\}$. Under the LLC, $\Pi_{\phi_{v}}(\tilde{T}_{z_{v}})$ corresponds to the singleton $\mathrm{Irr}(\pi_{0}(\tilde{S}_{\phi_{v}}^{[z_{v}]}),[z_{v}]) = \{[z_{v}]\}$.

    \item Let $v$ be a place where $\chi_{v}^{2} = 1$, then $A(F_{v})^{[z_{v}],\chi_{v}} = A(F_{v})$. When $v \notin \mathcal{S}$, the LLC for $\tilde{T}_{z_{v}}$ has been covered by Proposition~\ref{prop:genericlem}, since $\tilde{T}_{z_{v}}$ splits as a semi-direct product. When $v \in \mathcal{S}$, however, it is hard to describe the LLC straightforwardly (without explicitly appealing to the Tate--Nakayama pairing for hypercohomology). Nevertheless, both sides of the correspondence are clear. The local $L$-packet $\Pi_{\phi_{v}}(\tilde{T}_{z_{v}}) = \mathrm{Irr}(\tilde{T}_{z}(F_{v}),[\phi_{v}])$ consists of the two one-dimensional constituents of $\mathrm{Ind}_{T(F_{v})}^{\tilde{T}_{z}(F_{v})}\chi_{v}$, which are the two extensions of $\chi_{v}$ to $\tilde{T}_{z}(F_{v})$. Correspondingly, $\mathrm{Irr}(\pi_{0}(\tilde{S}_{\phi_{v}}^{[z_{v}]}),[z_{v}])$ consists of the two characters that extend $[z_{v}]$ to $\pi_{0}(\tilde{S}_{\phi_{v}}^{[z_{v}]})$.
\end{enumerate}
By~\eqref{equation: simp_pairing_Hasse}, after choosing $(1,1)\in \tilde{S}_{\phi}^{[z]}$ and $(1,1)\in \tilde{T}_{z}(F)^{\chi}$, we find 
$$\displaystyle\langle 1, \eta \rangle = \prod_{v}\mathrm{dim}\,\iota_{v}(\eta_{v}) = 1.$$ 
Since $A(F)^{[z],\chi} = 1$, the multiplicity of $\eta$ contributed by $[[\phi]]$ is $m_{\eta,\phi} = 1$. The following proposition can be shown by adapting the proof of Proposition~14.1 in~\cite{GanSavin2022}:
\begin{prop}
Let $\chi_{1}$ and $\chi_{2}$ be Hecke characters of $T$ that satisfy $\chi_{1,v} = \chi_{2,v}^{\pm1}$ for all $v$. Then we have $\chi_{1} = \chi_{2}^{\pm1}$.
\end{prop}
By this proposition, we see that $\mathcal{X}(\eta) = \{ \chi, \chi^{-1}\}$. Equivalently, we have $[[\Phi(\eta)]] = \{[[\phi]]\}$, where we denote by $[[\Phi(\eta)]]$ the set of near-$A(F)^{[z]}$-equivalence classes of the elements in $\Phi(\eta)$. In other words, $[[\phi]]$ is the unique near-$A(F)^{[z]}$-equivalence class that contributes to the multiplicity of $\eta$ in $L^{2}([\tilde{T}_{z}])$. Thus we conclude that the multiplicity of $\eta$ in $L^{2}([\tilde{T}_{z}])$ is $m_{\eta} = m_{\eta, \phi} =1$.

\item Let $\phi$ be a global $L$-parameter that gives rise to a Hecke character $\chi$ satisfying $\chi^{2} =1$. In particular, $\chi_{v}^{2} = 1$ for each place $v$. The LLC can be discussed in the same manner as in Case i(b). We focus on the special case when $\phi = 1$ and $\chi = 1$, and set out to compute $m_{\eta,1}$ precisely, where $\eta = \otimes'_{v}\eta_{v}\in \Pi_{1}(\tilde{T}_{z})$. In this case, $A(F_{v})^{[z_{v}],\chi_{v}} = A(F_{v})^{[z_{v}]} = A(F)^{[z],\chi} = A(F) = \{1,-1\}$ holds for any $v$. Due to the fact $\eta_{v}\in \mathrm{Irr}(\tilde{T}_{z}(F_{v}),1)$, we can write $\eta_{v} =\xi_{v}\circ q_{v}$, where $\xi_{v}\in \mathrm{Irr}(A(F_{v})^{[z_{v}]}) = \{\mathbb{1}, \mathrm{sgn}\}$ and $q_{v}$ is the natural projection $\tilde{T}_{z}(F_{v})\twoheadrightarrow A(F_{v})^{[z_{v}]}$. On the dual side, since $\phi =1$, we have $\pi_{0}(\tilde{S}_{\phi_{v}}^{[z_{v}]}) = \pi_{0}(\widehat{T}^{\,\Gamma_{v}})\rtimes A(F_{v})^{[z_{v}]}$.  Under the LLC, $\eta_{v}$ corresponds to $\rho_{v} = [z_{v}]\boxtimes \xi_{v} \in \mathrm{Irr}(\pi_{0}(\tilde{S}_{\phi_{v}}^{[z_{v}]}), [z_{v}])$. By the definition of $\Pi_{1}(\tilde{T}_{z})$, $\xi_{v} = 1$ holds for almost all places $v$. After choosing $(1,\pm1)\in \tilde{S}_{\phi}^{[z]}$, we use~\eqref{equation: simp_pairing_Hasse} to find that
\begin{align*}
\langle 1, \eta \rangle = \prod_{v}\mathrm{dim}\,\iota_{v}(\eta_{v}) = 1,
\end{align*}
and
\begin{align*}
\langle -1, \eta \rangle = \prod_{v}\mathrm{tr}[\,\iota_{v}(\eta_{v})(1,-1)] = \prod_{v}\xi_{v}(-1).
\end{align*}
Since the trivial representation is the unique constituent of $\eta_{v}\bigm |_{T(F_{v})}$ for all $v$, we have $\mathcal{X}(\eta) = \{1\}$ (or equivalently, $[[\Phi(\eta)]]$ contains only the trivial class). In other words, the trivial class is the unique near-$A(F)^{[z]}$-equivalence class that contributes to the multiplicity of $\eta$ in $L^{2}([\tilde{T}_{z}])$. Thus we conclude:
$$m_{\eta} = m_{\eta,1} = \frac{1}{2} \left(1 + \prod_{v}\xi_{v}(-1)\right).$$ 
Denote by $N$ the number of places $v$ for which $\eta_{v}$ descends to the sign character on $A(F_{v})$. Then, $\eta$ has multiplicity $1$ in $L^{2}([\tilde{T}_{z}])$ if $N$ is even; otherwise, $\eta$ is not automorphic.
\end{enumerate}

\begin{acknowledgements}
The author is grateful to Professor Wee Teck Gan for his invaluable insights, as well as for his guidance and proof-reading. The author also appreciates Professor Tasho Kaletha for his interest in this project and helpful suggestions. Appreciation is extended to the anonymous referee for their careful reading and constructive feedback. Finally, the author gratefully acknowledges the support of the NUS Research Scholarship.
\end{acknowledgements}

\providecommand{\bysame}{\leavevmode\hbox to3em{\hrulefill}\thinspace}
\providecommand{\MR}{\relax\ifhmode\unskip\space\fi MR }
\providecommand{\MRhref}[2]{%
  \href{http://www.ams.org/mathscinet-getitem?mr=#1}{#2}
}
\providecommand{\href}[2]{#2}

\newpage
\appendix
\section{Group hypercohomology and group hyperhomology}\label{appendix: hyper}
We collect the definitions and basic facts about group hyper(co)homology that are used in the body of this paper. None of the results below are original. Fuller treatments may be found in~\cite{weibel1994introduction} and~\cite{kottwitz1999foundations}.
\subsection{Definition}
Let $G$ be a (not necessarily finite) group. First we recall (unnormalised, in the term of~\cite{weibel1994introduction}) bar resolution of $\mathbb{Z}$ (as a trivial $G$-module). We set $B_{0} = \mathbb{Z}G$. When $n \geqslant 1$, we let $B_{n}$ be the free $\mathbb{Z}G$-module on the set of all symbols $[g_{1}\otimes g_{2}\otimes \dots \otimes g_{n}]$ with $g_{i}\in G$. Then we define the following free resolution:
\begin{align*}
0 \leftarrow \mathbb{Z} \xleftarrow{\epsilon} B_{0} \xleftarrow{\partial} B_{1} \xleftarrow{\partial} B_{2} \xleftarrow{\partial} \cdots 
\end{align*}
where augmentation map $\epsilon$ is the unique $\mathbb{Z}G$-morphism sending $1 \in \mathbb{Z}G$ to $1$, and for $n\geqslant 1$, $\partial:B_{n}\to B_{n-1}$ is defined to be $\partial := \sum_{i=0}^{n}(-1)^{i}\partial_{i}$ with
\begin{align*}
\partial_{0}([g_{1}\otimes \cdots\otimes g_{n}]) &= g_{1}[g_{2}\otimes \cdots\otimes g_{n}];\\
\partial_{i}([g_{1}\otimes \cdots\otimes g_{n}]) &= [g_{1}\otimes \cdots \otimes g_{i}g_{i+1} \otimes \cdots \otimes g_{n}] \text{ for } i=1,2,\dots,n-1;\\
\partial_{n}([g_{1}\otimes \cdots\otimes g_{n}]) &= [g_{1}\otimes \cdots \otimes g_{n-1}].
\end{align*}

Let $A^{\bullet}$ be a bounded complex of $G$-modules with differentials $f^{k}: A^{k} \to A^{k+1}$. Now we consider a cochain complex $C^{\bullet}(G,A^{\bullet})$ whose $n$-th term is
\begin{align*}
    C^{n}(G,A^{\bullet}) = \bigoplus_{k}\mathrm{Hom}_{\mathbb{Z}G}(B_{n-k},A^{k}),
\end{align*}
with differential $d^{n}:C^{n}(G,A^{\bullet})\to C^{n+1}(G,A^{\bullet})$ whose restriction on a factor $\mathrm{Hom}_{\mathbb{Z}G}(B_{n-k},A^{k})$ is the $\mathbb{Z}$-linear map
\begin{align*}
\mathrm{Hom}_{\mathbb{Z}G}(B_{n-k}, A^{k}) \to \mathrm{Hom}_{\mathbb{Z}G}(B_{n-k}, A^{k+1}) \oplus \mathrm{Hom}_{\mathbb{Z}G}(B_{n+1-k}, A^{k})  
\end{align*}
defined by 
\begin{align*}
    d^{n}c^{k} = f^{k}\circ c^{k} + (-1)^{k}c^{k}\circ \partial.
\end{align*}

Then we define the group hypercohomology $H^{\bullet}(G,A^{\bullet})$ to be the cohomology of $C^{\bullet}(G,A^{\bullet})$.

Similarly, we can define group hyperhomology. Let $M_{\bullet}$ be a bounded complex of $G$-modules with differentials $\delta_{k}: M_{k} \to M_{k-1}$. Then we define a chain complex $C_{\bullet}(G,M_{\bullet})$ whose $n$-th term is
\begin{align*}
    C_{n}(G,M_{\bullet}) = \bigoplus_{k}B_{n-k}\otimes_{\mathbb{Z}G} M_{k}, 
\end{align*}
with differential $d_{n}:C_{n}(G,M_{\bullet})\to C_{n-1}(G,M_{\bullet})$ whose restriction on a factor $B_{n-k}\otimes_{\mathbb{Z}G} M_{k}$ is the $\mathbb{Z}$-linear map
\begin{align*}
B_{n-k}\otimes_{\mathbb{Z}G} M_{k} \to B_{n-k}\otimes_{\mathbb{Z}G} M_{k-1} \oplus B_{n-k-1}\otimes_{\mathbb{Z}G} M_{k} 
\end{align*}
defined by 
\begin{align*}
    d_{n}(b\otimes m) = b\otimes \delta_{k}(m) + (-1)^{k}\partial(b)\otimes m.
\end{align*}
We define the group hyperhomology $H_{\bullet}(G,M_{\bullet})$ to be the homology of $C_{\bullet}(G,M_{\bullet})$.

The notions of $r$-th hyper(co)cycle and hyper(co)boundary are self-evident. 
 
\subsection{Hyper(co)homology in complexes of length 2}
Group hyper(co)homology used in this paper concerns complexes of length 2.

\subsubsection{Hypercohomology}
Let $A$ be a $G$-module. For $a \in A$, $\partial a$ denotes the $1$-coboundary sending $g\in G$ to $a^{-1}g(a)$.

Whenever we consider hypercohomology of a complex of $G$-modules of length 2, we implicitly regard it as a complex concentrated at degrees $0$ and $1$: $A^{0}\xrightarrow{f}A^{1}$. We explicate the definition of hypercohomology at degree $1$:
\begin{align*}
H^{1}(G,A^{0}\xrightarrow{f}A^{1}) = 
\frac{\{(z,a_{1})\bigm | f(z) = \partial a_1, z \in Z^{1}(G,A^{0}), a_{1}\in A^{1}\}}{\{(\partial a_{0}, f(a_{0}))\bigm | a_{0}\in A^{0}\}}.
\end{align*}

It is noteworthy that there is an important long exact sequence involving hypercohomology and cohomology groups of $A^{0}$ and $A^{1}$.
\begin{fact}
For any $r\geqslant 1$, the following sequence is exact:
\begin{align*}
\cdots \to H^{r-1}(G,A^{0}) \xrightarrow{f} H^{r-1}(G,A^{1}) \xrightarrow{i} H^{r}(G,A^{0}\xrightarrow{f} A^{1}) \xrightarrow{p} H^{r}(G, A^{0}) \to \cdots,
\end{align*}
where $i$ is defined by sending any $(r-1)$-cocycle $c_{1}$ to $r$-hypercocycle $(0,c_{1})$, and $p$ by sending any $r$-hypercocycle $(c_{0},c_{1})$ to $r$-cocycle $c_{0}$.
\end{fact}
\begin{proof}
One can apply Grothendieck's spectral sequence for hypercohomology to this special case. Alternatively, it can be directly proved as below. 

Exactness at $H^{r-1}(G,A^{1})$: $i(c_{1})$ is an $r$-hypercoboundary if and only if $f(c_{0}) -\partial c_{1}^{*} = c_{1}$ for some ($r-1$)-cocycle $c_{0}$ and ($r-2$)-cochain $c_{1}^{*}$. And this happens if and only if $c_{1}$ and $f(c_{0})$ are cohomologous.

Exactness at $H^{r}(G,A^{0}\xrightarrow{f} A^{1})$: $p\circ i = 0$ is straightforward by definition. For the other inclusion, if $p(c_{0},c_{1})$ is a coboundary, then $c_{0} = \partial c_{0}^{*}$ for some $(r-1)$-cochain $c_{0}^{*}$, and then $(c_{0},c_{1}) = (\partial c_{0}^{*}, c_{1})$ differs from $(0, c_{1}-f(c_{0}^{*}))$ by a hypercoboundary $(\partial c_{0}^{*}, f(c_{0}^{*}))$.

Exactness at $H^{r}(G,A^{0})$: Given an $r$-cocycle $c_{0}$ in $A^{0}$, by definition, there exists an $r$-hypercocycle $(c_{0},c_{1})$ for some $c_{1}$ if and only if $f(c_{0}) = \partial c_{1}$ for some $(r-1)$-cochain $c_{1}$ in $A^{1}$. 
\end{proof}

\subsubsection{Hyperhomology} \label{hyperhomologydefn}
When considering hyperhomology of a complex of length 2, we implicitly see it as a complex concentrated at degrees $0$ and $-1$: $A\xrightarrow{\delta} B$. We explicate the hyperhomology at degree $0$:
\begin{align*}
H_{0}(G,A\xrightarrow{\delta}B) = \frac{\{(a_{0},b_{1})\bigm | \delta(a_{0}) = \partial b_{1}, a_{0}\in A, b_{1}\in C_{1}(G,B)\}}{\{(\partial a_{1}, \delta(a_{1})-\partial b_{2}) \bigm | a_{1}\in C_{1}(G,A), b_{2} \in C_{2}(G,B)\}}.
\end{align*}
Similar to hypercohomology, there is a long exact sequence:
\begin{fact}\label{Exactforhyperho}
For any $r\geqslant 1$, the following sequence is exact:
\begin{align*}
\cdots \to H_{r}(G,A) \xrightarrow{\delta} H_{r}(G,B) \xrightarrow{i} H_{r-1}(G,A\xrightarrow{\delta}B) \xrightarrow{p} H_{r-1}(G, A) \to \cdots,
\end{align*}
where $i$ is defined by sending any $r$-cycle $b_{r}$ to $(r-1)$-hypercycle $(0,b_{r})$, and $p$ by sending any $(r-1)$-hypercycle $(a_{r-1},b_{r})$ to $(r-1)$-cycle $a_{r-1}$.
\end{fact}


\begin{thebibliography}{PRR93}

\bibitem[BB04]{bourbaki2004integration}
Nicolas Bourbaki and Sterling~K Berberian, \emph{Integration II}, Springer,
  2004.

\bibitem[Fol16]{folland2016course}
Gerald~B Folland, \emph{A course in abstract harmonic analysis}, vol.~29, CRC
  press, 2016.

\bibitem[GS22]{GanSavin2022}
Wee~Teck Gan and Gordan Savin, \emph{Twisted composition algebras and Arthur
  packets for triality $\operatorname{Spin}_{8}$}, Pure and Applied Mathematics
  Quarterly \textbf{18} (2022), no.~5, 1951--2130.

\bibitem[Isa94]{isaacs1994character}
I~Martin Isaacs, \emph{Character theory of finite groups}, vol.~69, Courier
  Corporation, 1994.

\bibitem[Kal16]{kaletha2016rigid}
Tasho Kaletha, \emph{Rigid inner forms of real and $p$-adic groups}, Annals of
  Mathematics \textbf{184} (2016), no.~2, 559--632.

\bibitem[Kal18]{kaletha2018global}
\bysame, \emph{Global rigid inner forms and multiplicities of discrete
  automorphic representations}, Inventiones mathematicae \textbf{213} (2018),
  no.~1, 271--369.

\bibitem[Kal22]{kaletha2022local}
\bysame, \emph{On the local Langlands conjectures for disconnected groups},
  arXiv preprint arXiv:2210.02519 (2022).

\bibitem[Kot84]{kottwitz1984stable}
Robert~E Kottwitz, \emph{Stable trace formula: Cuspidal tempered terms}, Duke
  Math. J. \textbf{51} (1984), no.~1, 611--650.

\bibitem[KS99]{kottwitz1999foundations}
Robert~E Kottwitz and Diana Shelstad, \emph{Foundations of twisted endoscopy},
  Ast{\'e}risque \textbf{255} (1999), 1--190.

\bibitem[Lan82]{langlands1982debuts}
Robert~P Langlands, \emph{Les d{\'e}buts d'une formule des traces stable},
  vol.~13, UER de math{\'e}matiques \& LA 212 du CNRS, 1982.

\bibitem[Lan97]{langlands1997representations}
Robert Langlands, \emph{Representations of abelian algebraic groups}, pacific
  journal of mathematics \textbf{181} (1997), no.~3, 231--250.

\bibitem[LL79]{labesse1979indistinguishability}
Jean-Pierre Labesse and Robert~P Langlands, \emph{L-indistinguishability for SL
  (2)}, Canadian Journal of Mathematics \textbf{31} (1979), no.~4, 726--785.

\bibitem[MW95]{moeglin_waldspurger_1995}
C.~Moeglin and J.~L. Waldspurger, \emph{Spectral Decomposition and Eisenstein
  Series: A Paraphrase of the Scriptures}, Cambridge Tracts in Mathematics,
  Cambridge University Press, 1995.

\bibitem[PRR93]{platonov1993algebraic}
Vladimir Platonov, Andrei Rapinchuk, and Rachel Rowen, \emph{Algebraic groups
  and number theory}, Academic press, 1993.

\bibitem[Ser97]{serre1997galois}
Jean-Pierre Serre, \emph{Galois cohomology}, Springer, 1997.

\bibitem[Tat79]{tate1979number}
John Tate, \emph{Number theoretic background}, Automorphic forms,
  representations and L-functions (Proc. Sympos. Pure Math., Oregon State
  Univ., Corvallis, Ore., 1977), Part, vol.~2, 1979, pp.~3--26.

\bibitem[Tre16]{treves2016topological}
Fran{\c{c}}ois Treves, \emph{Topological Vector Spaces, Distributions and
  Kernels: Pure and Applied Mathematics, Vol. 25}, vol.~25, Elsevier, 2016.

\bibitem[Wei94]{weibel1994introduction}
Charles~A Weibel, \emph{An introduction to homological algebra}, no.~38,
  Cambridge university press, 1994.

\end{thebibliography}
\end{document}